\documentclass[a4paper]{article}

\usepackage[english]{babel}
\usepackage[utf8]{inputenc}
\usepackage{amsmath}
\usepackage{amsthm}
\usepackage{mathrsfs}
\usepackage{amssymb}
\usepackage{color,soul}

\usepackage{graphicx}
\numberwithin{equation}{section}

\usepackage[colorinlistoftodos]{todonotes}
\usepackage{xcolor}

\newtheorem{theorem}{Theorem}[section]
\newtheorem{prop}[theorem]{Proposition}
\newtheorem{defn}[theorem]{Definition}
\newtheorem{cor}[theorem]{Corollary}
\newtheorem{lem}[theorem]{Lemma}

\theoremstyle{remark}
\newtheorem*{oss}{{\bf Remark}}

\newtheorem*{ass}{Assumption A}

\newcommand{\uhat}{{\hat{\mathbf u}}}
\newcommand{\utilde}{{\tilde{\mathbf u}}}
\newcommand{\uv}{{\mathbf u}}
\newcommand{\y}{{\mathbf y}}

\newcommand{\nuhat}{{{\hat \nu}_{t,x}}}
\newcommand{\nutilde}{{{\tilde \nu}_{t,x}}}

\newcommand{\E}{{\mathbb E}}
\newcommand{\R}{{\mathbb R}}
\newcommand{\Pbb}{{\mathbb P}}

\newcommand{\Pbbtilde}{{\tilde{\mathbb P}}}

\newcommand{\lambN}{{\lambda^N_{\beta, \bar p, \tau}}}
\newcommand{\lamb}{{\lambda_{\beta, \bar p, \tau}}}
\newcommand{\lambo}{{\lambda_{\beta, 0, \tau}}}
\newcommand{\barmi}[1][]{\ensuremath{\bar #1_{m,i}^{\rho, \bar p}}}
\newcommand{\barli}[1][]{\ensuremath{\bar #1_{l,i}^{\rho, \bar p}}}

\newcommand{\limN}{{\lim_{N \to \infty}}}
\newcommand{\weak}{\rightharpoonup}
\newcommand{\weakstar}{\overset{*}{\rightharpoonup}}

\newcommand{\normp}[2][]{\ensuremath{\left\lVert #1 \right\rVert}_{L^{#2}(Q_T)}}
\newcommand{\norminf}[1][]{\ensuremath{\left\lVert #1 \right\rVert}_{L^\infty}}
\newcommand{\normH}[1][]{\ensuremath{\left\lVert #1 \right\rVert}_{H^1}}
\newcommand{\normLdue}[1][]{\ensuremath{\left\lVert #1 \right\rVert}_{L^2}}

\newcommand{\expnu}[1][]{\langle #1, \nutilde \rangle}

\newcommand{\parr}[1][]{\ensuremath{\partial_{r_{#1}}}}
\newcommand{\parp}[1][]{\ensuremath{\partial_{p_{#1}}}}

\begin{document}
\title{Hydrodynamic Limit for an Anharmonic Chain \\ under Boundary Tension}
\author{Stefano Marchesani \\ Stefano Olla\\
}


\date{}
\maketitle

\abstract
 {
We study the hydrodynamic limit for the isothermal dynamics of an
anharmonic chain under hyperbolic space-time scaling under varying tension. 
The temperature is kept constant by a contact with a heat bath, realised via a
stochastic momentum-preserving noise added to the dynamics. The
noise is designed to be large  at the microscopic level, but vanishing
in the macroscopic scale. Boundary conditions are also considered:
one end of the  chain is kept fixed, while a time-varying tension
is applied to the other end. 
We show that the volume stretch and
momentum converge (in an appropriate sense) to a weak solution of a
system of hyperbolic conservation laws (isothermal Euler equations in
Lagrangian coordinates) with boundary conditions. 
This result includes the shock regime of the system. 
This is proven by adapting the theory of compensated compactness to a
stochastic setting, as developed by J. Fritz in} \cite{Fritz1}  {for the
same model without boundary conditions. 
Finally, changing the external tension allows us to define
thermodynamic isothermal transformations between equilibrium states. We use this
to deduce the first and the second principle of Thermodynamics for our
model.}

{\let\thefootnote\relax
\footnote{{\today}}
}

\section{Introduction} 
\label{sec:intro}

Hydrodynamic limits concern the deduction of macroscopic conservation laws from microscopic dynamics.
Ideally the microscopic dynamics should be deterministic and Hamiltonian but most
existing results are obtained using microscopic stochastic dynamics. 
Often the stochastic dynamics { models} the action of a heat bath thermalising a Hamiltonian dynamics.

For scalar hyperbolic conservation laws {these hydrodynamic} limits are well understood, 
even in presence of shocks \cite{Reza:1991} and of boundary conditions \cite{baha2012}. 
Much less is known for hyperbolic systems of conservation laws { with} boundary conditions, 
which have been understood only in the smooth regime \cite{even2010hydrodynamic}. 
In presence of shock waves, in infinite volume, only the hydrodynamic limit
 for the Leroux system \cite{toth2004} and the p-system \cite{Fritz1} have been obtained. 
Since uniqueness of entropy solutions is still an open problem for these systems, the result is 
intended here only in the sense that the limit distribution of the macroscopic profiles concentrates 
on the set of possible weak solutions. 

 { This article is a first attempt at understanding
the hydrodynamic limit in presence of 
boundary conditions and shocks in dynamics with more conservation laws. 
Changing boundary conditions 
(in particular time dependent tension) are important in order to perform 
isothermal transformations and study the corresponding first and second laws of thermodynamics. }

The model is an anharmonic chain of $N+1$ particles with a
time-dependent external force (tension) attached to one end of the chain
(particle number $N$). The other end of the chain (particle number
$0$) is kept fixed. 
\begin{center}
\includegraphics[width=0.8\textwidth]{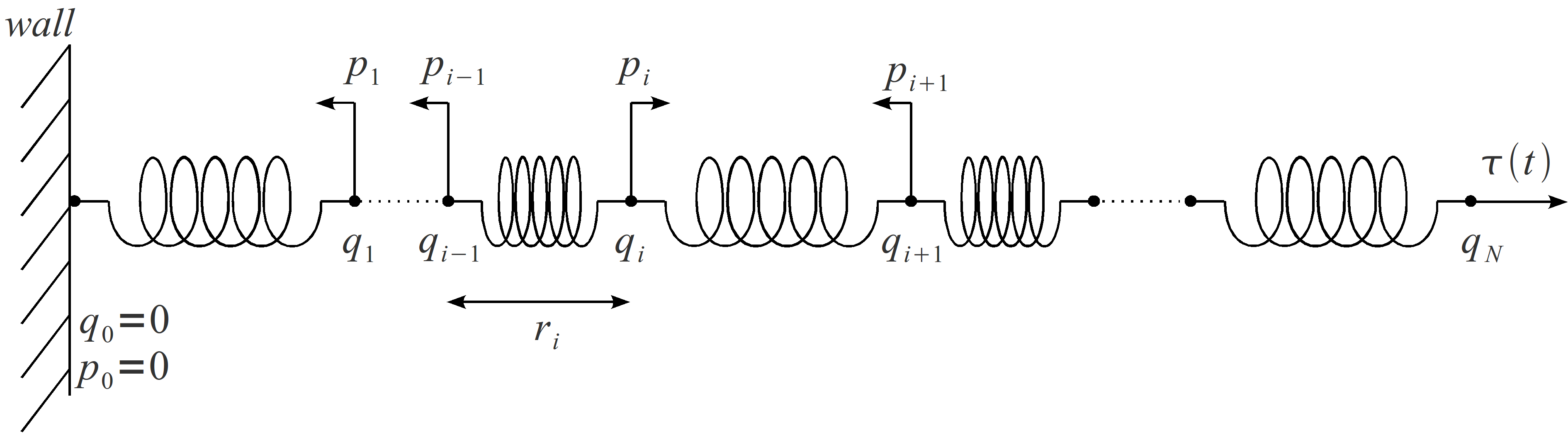}\\	
\end{center}

The system is in contact with a thermal bath
modelled by a stochastic dynamics chosen in such a way that: 
\begin{enumerate}
\item The total dynamics is ergodic,
\item The temperature of the chain is fixed to the value $\beta^{-1}$,
  i.e. the equilibrium stationary probability are given by canonical
  Gibbs measure at this temperature.
\item The momentum and volume are locally conserved, while the energy is not.
\item The strength of the noise is scaled in such a way that it does not appear in the macroscopic equations.
\end{enumerate}
 { This noise is realised by a random continuous exchange of momentum and volume stretch between nearest neighbour particles}.  
This is the same setup considered by Fritz \cite{Fritz1} in infinite volume in order to obtain the p-system:
\begin{equation}
  \label{eq:p1}
  \begin{cases}
 \partial_t  r (t,x) -  \partial_x  p(t,x)=0
 \\
 \partial_t  p(t,x)- \partial_x \tau_\beta(r(t,x))=0,
\end{cases}
\end{equation}
where $r(t,x)$ and $p(t,x)$ are the local volume strain and momentum of the chain, 
while  { $\tau_\beta(r)$, smooth and strictly increasing in $r$, is the equilibrium tension of the chain corresponding to the length $r$ 
and temperature $\beta^{-1}$}. Here $x$ is the Lagrangian material coordinate.
For the finite system $x\in [0,1]$. The physical boundary conditions that we impose microscopically are
\begin{itemize}
\item $p(t,0) = 0, \ t\ge 0$: the first particle is not moving,
\item $\tau_\beta(r(t,1)) = \bar \tau(t)$, where $\bar\tau(t)$ is the force (tension) applied to the last particle $N$, 
eventually changing on the macroscopic time scale.
\end{itemize}
In the shock regime, when weak non-smooth solutions are considered, one has to specify the meaning of 
these boundary conditions, since a discontinuity can be found right at the boundaries. 
The standard way to address this (see eg \cite{chen-frid1999})
 is to consider the special viscous approximation
\begin{equation}
  \label{eq:vpsystem-0}
\begin{cases}
 \partial_t  r^\delta (t,x) -  \partial^\delta_x  p(t,x)= \delta_1 \partial_{xx} \tau_\beta(r^\delta(t,x))
 \\
 \partial_t  p^\delta(t,x)- \partial_x \tau_\beta( r^\delta(t,x))= \delta_2 \partial_{xx} p^\delta(t,x)
\end{cases}.
\end{equation}
with the boundary conditions:
\begin{equation*}
 p^\delta(t,0)=0, \quad \tau_\beta(r^\delta(t,1)) = \bar \tau(t),
\quad \partial_x  p^\delta(t,1)= 0, \quad  \partial_x  r^\delta(t,0)= 0.
\end{equation*}
Then the \emph{vanishing viscosity solutions} of \eqref{eq:p1} are defined as the limit for 
$\delta = (\delta_1,\delta_2) \to 0$ for $ r^\delta,  p^\delta$. 
Notice that \eqref{eq:vpsystem-0} has two extra boundary conditions 
that will create a boundary layer in the limit $\delta\to 0$. The particular choice 
of the viscosity terms and boundary conditions in \eqref{eq:vpsystem-0} is done in such a way 
that we have the right thermodynamic entropy production  {(see appendix} \ref{sec:visc-appr}).
At the moment there is no uniqueness result for this vanishing viscosity limit, and in principle 
it may depend on the particular choice of the viscosity term. 

The stochastic perturbation of our microscopic dynamics is chosen so that it gives 
a microscopic stochastic version of \eqref{eq:vpsystem-0}. 
We prove 
that the distribution of the empirical profiles of strain and  
momentum, tested against functions with compact support on $(0,1)$, 
concentrate on weak solutions of \eqref{eq:p1}.
Unfortunately we are not able to prove that these limit profiles are the vanishing viscosity solutions 
with the right boundary conditions, but we conjecture that our limit distributions are concentrated 
on such vanishing viscosity solutions.


Hydrodynamic limits {in a smooth} 
 regime have been well understood so
far. The hydrodynamic limit for the 1D full $3 \times 3$ Euler system
in Lagrangian coordinates and boundary conditions has been studied in
\cite{even2010hydrodynamic}, while the 3D $3 \times 3$ Euler system in
Eulerian coordinates has been derived in
\cite{olla1993hydrodynamical}. Both \cite{even2010hydrodynamic} and
\cite{olla1993hydrodynamical} use the relative entropy method introduced in the diffusive setting by Yau
\cite{Yau1991}. 

The relative entropy method assumes the existence of strong
solutions to the macroscopic equation. 
Then one samples these solutions and builds a
family of time-dependent inhomogeneous Gibbs measure which are 
used for the relative entropy.




As an alternative to the relative entropy method, \cite{Fritz1} adapts
the techniques of the vanishing viscosity to a stochastic
setting, in conjunction with the approach of Guo-Papanicolau-Varadhan \cite{Guo:1988p18645}
 based on bounds on Dirichlet forms.
We extend the work in \cite{Fritz1} to our model by considering a
 finite chain and including boundary conditions, where the chain is
 attached to one point on one side and subject to a varying tension
 force on the other side.
We construct some averages of the conserved quantities $\uhat_N(t,x)$ which solve (in an
appropriate sense) equations that approximate in a mesoscopic scale the p-system we want to
derive as $N \to \infty$. Then we carefully perform the limit $N \to
\infty$ and obtain $L^2$-valued weak solutions to the p-system. 
The main technical challenge is then to prove that
we can commute the weak limits with composition with nonlinear
functions.  { This is done using a stochastic extension,
 introduced by Fritz in} \cite{Fritz1},  {of the Tartar-Murat 
compensated compactness lemma, properly adapted to the presence of the boundary conditions.
The compensated compactness was used originally by Di Perna} \cite{diperna}
 {in order to prove convergence of viscous approximation of hyperbolic systems.}

After proving the hydrodynamic limit, we exploit the external tension
$\bar \tau(t)$ in order to perform a thermodynamic transformation
between two equilibrium states. This is done by letting $\bar \tau(t)$
to change from a value $\tau_0$ at $t =0$ to a value $\tau_1$ as $t
\to \infty$. Correspondingly, the system is brought from the
equilibrium state $(\beta,\tau_0)$ to the state
$(\beta,\tau_1)$.  Since the temperature is fixed by the noise,
 this transformation is \emph{isothermal}.  

Isothermal transformations are of great importance in thermodynamics,
as they constitute, together with adiabatic transformations, the Carnot cycle. The study of the first and the second law of
thermodynamics for an isothermal transformation in \emph{smooth}
regime has been carried out in \cite{olla2014microscopic}, in a system
where energy and momentum are not conserved and a
diffusive scaling is performed. In that situation volume stretch evolves macroscopically 
accordingly to a nonlinear diffusive equation.  

The first law of thermodynamics is an energy balance, which takes into
account ``gains'' or ``losses'' of total internal energy via exchange of heat and work: 
one defines the internal energy $U$ and the work $W$ (which depends on the
external tension only) and proves that the difference of internal
energy between two equilibrium states is given by $W$ plus some extra
term, which we call \emph{heat} and denote by $Q$. In formulae,
$\Delta U = W + Q$. The heat depends on terms coming from the stochastic thermostats
which survive in the limit $N \to \infty$. We prove the first law in
exactly the same fashion as \cite{olla2014microscopic}. 

The second law 
states that, for an
isothermal transformation, the difference of thermodynamic entropy
$\Delta S$ is never smaller than $\beta Q$. The equality $\Delta S =
\beta Q$ occurs only for \emph{quasistatic transformations}. The
entropy $S$ is defined by $S=\beta(U-F)$, where $F$ is the free
energy.  The second law can then be restated as $\Delta F \le W$. This
is also known as \emph{inequality of Clausius}. In
\cite{olla2014microscopic}, this inequality is obtained at
the macroscopic level: the macroscopic equation is diffusive and the
system dissipates even if the solutions are smooth. This is not the
case in the present paper, as smooth solutions would always give a
Clausius \emph{equality}. 
The main assumption that we have to make in order to obtain  \emph{inequality of Clausius}
is that the distributions of the limit profiles concentrate on the vanishing viscosity solutions.
We will refer to this solutions as  {\emph{thermodynamic entropy solutions}}.  { 
In mathematical literature the term \emph{entropy solution} is referred 
to a more strict class of weak solutions (in principle). }


\section{The Model and the Main Theorem} \label{sec:mainthm}
We study a one-dimensional Hamiltonian system of $N+1 \in \mathbb N$
particles of unitary mass. The position of the $i$-th particle ($i \in
\{0,1,\dots, N\}$)  is denoted by $q_i \in \mathbb R$ and its momentum
by $p_i \in \mathbb R$. We assume that particle $0$ is kept fixed,
i.e. $(q_0,p_0) \equiv (0,0)$, while on particle $N$ is applied a
time-dependent force, $\bar \tau(t)$,  {bounded, with bounded derivative.}

Denote by ${\bf q} =(q_0, \dots, q_N)$ and ${\bf p} =(p_0,\dots,
p_N)$. The interaction between particles $i$ and $i-1$ is described by
the potential energy $V(q_i-q_{i-1})$ of an anharmonic spring. 

 {We take $V$ to be a mollification of the function}
\begin{equation} \label{eq:Vvera}
r \longmapsto \frac{1}{2}(1-\kappa)r^2+\frac{1}{2}\kappa r|r|_+,
\end{equation}
 {where $|r|_+=\max\{r,0\}$ and $\kappa \in (0,1/3)$.} 

In particular, $V$ is a  {uniformly convex function that grows
quadratically at infinity}: there exist constants $c_1$ and $c_2$ such
that for any $r \in \mathbb R$: 
\begin{equation}  \label{eq:Vconvex}
0<c_1 \le V''(r) \le c_2.
\end{equation}
Moreover, there are some positive constants $V''_+, V''_-, \alpha$ and $R$ such that
\begin{equation} \label{eq:AN}
\begin{gathered}
\left| V''(r) - V''_+ \right| \le e^{-\alpha r}, \quad r > R
\\
\left| V''(r) - V''_- \right| \le e^{\alpha r}, \quad r < -R.
\end{gathered}
\end{equation}
 {Finally, the choice of $\kappa$ is such that the macroscopic tension, defined below, is strictly convex.}

The energy is defined by the following Hamiltonian:
\begin{equation} 
\label{eq:hamiltonian}
\mathcal H_N({\bf q}, {\bf p}):= \sum_{i =0}^N \left( \frac{p_i^2}{2}+V(q_i-q_{i-1})\right),
\end{equation}
Since the interaction depends only on the distance between particles, we define
\begin{equation}
r_i := q_i-q_{i-1}, \quad i \in \{1,\dots, N\}.
\end{equation}
Consequently, the configuration of the system is given by  
$\left({\bf  r} =(r_1, \dots, r_N), {\bf p} =(p_0,\dots,p_N)\right)^\intercal$ and the
  phase space is given by $\mathbb R^{2N}$. 

 {Given the tension $\bar \tau(t)$, the dynamics of the system is determined by the generator}
\begin{equation} \label{eq:fullgen}
 \mathcal G_N^{\bar \tau(t)}:= N L_N^{\bar \tau(t)}+ N{\sigma} ( S_N + \tilde  S_N ).
\end{equation}
 {$\sigma=\sigma(N)$ is a positive number that tunes the strength of the noise. We need it to be big enough to provide ergodicity but small enough to disappear in the hydrodynamic limit:}
\begin{equation} \label{eq:siglim}
\lim_{N\to+ \infty} \frac{{\sigma}}{N}= \lim_{N \to \infty} \frac{N}{{\sigma}^2} = 0.
\end{equation}
The Liouville operator $L_N^{\bar \tau(t)}$ is given by
\begin{equation} \label{eq:hamgen}
L_N^{\bar \tau(t)}  = \sum_{i = 1}^N (p_i-p_{i-1}) \partial_{r_i} +
\sum_{i =1}^{N-1} \left(V'(r_{i+1}) -V'(r_i) \right)  \partial_{p_i} +(\bar \tau(t) - V'(r_N)) \partial_{p_N},
\end{equation}
where we have used the fact that $p_0 \equiv 0$. Note that the time
scale in the tension is chosen such that it changes smoothly on the
macroscopic scale. 

The operators $S_N$ and $\tilde S_N$ generate the stochastic part of the dynamics and are defined by
\begin{equation} \label{eq:SN}
S_N := -\sum_{i=1}^{N-1}  D^*_i D_i, \quad \tilde S_N := -\sum_{i=1}^{N-1}  \tilde D^*_i \tilde D_i,
\end{equation}
where
\begin{equation} \label{eq:DD*}
\begin{gathered}
\qquad D_i := \frac{\partial }{\partial p_{i+1}}-\frac{\partial }{\partial p_i}, \qquad D^*_i := p_{i+1}-p_i - \beta^{-1}D_i
\\
\tilde D_i  := \frac{\partial }{\partial r_{i+1}}-\frac{\partial }{\partial r_i}, \qquad \tilde D^*_i := V'(r_{i+1})- V'(r_i) -  \beta^{-1}\tilde D_i.
\end{gathered}
\end{equation}
They conserve total mass and momentum but not energy. The temperature
is fixed to the constant value $\beta^{-1}$, in the sense that the
only stationary measures of the stochastic dynamics  generated by $S_N
+\tilde S_N$ are given by the corresponding canonical Gibbs measure
at temperature $\beta^{-1}$, see definition below. 

The positions and the momenta of the particles then evolve in time accordingly to the following system of stochastic equations
\begin{equation} \label{eq:SDE}
\begin{cases}
d r_1 = Np_1 d t +N \sigma \left(V'(r_2)-V'(r_1)\right) d t - \sqrt{2 \beta^{-1}N\sigma}d \widetilde w_1
\\
d r_i = N(p_i-p_{i-1}) d t +N \sigma\left(V'(r_{i+1})+V'(r_{i-1})-2V'(r_i)\right) d t + \sqrt{2 \beta^{-1} N \sigma}
(d \widetilde w_{i-1}-d \widetilde w_i)
\\
d r_N = N(p_N-p_{N-1}) d t +N \sigma\left(V'(r_{N-1}) - V'(r_N)\right) d t + 
\sqrt{2\beta^{-1} N\sigma } d \widetilde w_{N-1}
\\
d p_1 = N(V'(r_2)-V'(r_1)) d t +N {\sigma} \left(p_2 - p_1\right) d t 
- \sqrt{2\beta^{-1}N {\sigma}}  { d w_1}
\\
d p_i = N(V'(r_{i+1})-V'(r_i)) d t +N \sigma \left(p_{i+1}+p_{i-1}-2p_i\right) d t + \sqrt{2\beta^{-1}N \sigma }(d  w_{i-1}-d  w_i)
\\
d p_N = N(\bar \tau(t)-V'(r_N)) d t +N \sigma \left(p_{N-1}-p_N\right) d t + \sqrt{2\beta^{-1}N \sigma }d  w_{N-1},
\end{cases},
\end{equation}
for $i \in \{2,\dots, N-1\}$.  { $\{ w_i\}_{i=1}^\infty$ and $\{\widetilde w_i\}_{i=1}^\infty$} are independent families of independent Brownian motions on a common probability space $(\Omega, \mathcal F, \Pbb)$. The expectation with respect to $\Pbb$ is denoted by $\E$.

For $\tau \in \mathbb R$ we define the canonical Gibbs function as
\begin{equation} \label{eq:Gdefn}
 { G(\beta,\tau):= \log \int_{\mathbb R} \exp\left(
  -\beta V(r) +\beta \tau r  \right)\; dr.}
\end{equation}
For $\rho \in \mathbb R$, the free energy is given by the Legendre transform of $G$:
\begin{equation}\label{eq:freedef}
F(\beta, \rho) := \sup_{\tau \in \mathbb R} \left\{ \tau \rho - \beta^{-1}G(\beta, \tau)\right\},
\end{equation}
so that its inverse is
\begin{equation}\label{eq:legInv}
G(\beta,\tau) = \beta \sup_{\rho \in \mathbb R} \left\{ \tau \rho -  F(\beta,\rho)\right\}.
\end{equation}
We denote by $\rho(\beta,\tau)$ and $\tau(\beta,\rho)$  the corresponding convex conjugate variables, that satisfy
\begin{equation}\label{eq:gamrho}
\rho(\beta, \tau) = \beta^{-1}\partial_\tau G(\beta,\tau), \quad  {\tau_\beta}(\rho) = \tau(\beta,\rho) = \partial_\rho F(\beta,\rho).
\end{equation}
On the one-particle state space $\mathbb R^2$ we define a family of probability measures
\begin{equation} \label{eq:lambdagamdef}
\lamb (d r, dp) :=  \exp\left(-\frac{\beta}{2}(p-\bar p)^2-\beta V(r) +\beta\tau r - G(\beta,\tau) \right)\; dr\;  {\frac{dp}{\sqrt{2\pi\beta^{-1}}}}.
\end{equation}
The mean deformation and momentum are
\begin{equation}\label{eq:rhopi}
{ \E_\lamb[r] = \rho(\beta,\tau), \quad \E_\lamb[p] = \bar p.}
\end{equation}
We have the relations
\begin{equation}\label{eq:temper}
{\E_\lamb[p^2]-\bar p^2 = \beta^{-1}, \quad \E_\lamb[V'(r)] = \tau}
\end{equation}
that identify $\beta^{-1}$ as the temperature and $\tau$ as the tension.

For constant $  \bar \tau$ in the dynamics, the family of product measures
\begin{equation}\label{eq:invmeas}
\lambda^N_{\beta,0, \bar \tau}(d{\bf r}, d{\bf p}) = \prod_{i=1}^N\lambda_{\beta,0, \bar \tau}(dr_i, dp_i)
\end{equation}
is stationary. These are the canonical Gibbs measures
at a temperature $\beta^{-1}$, pressure $ \bar \tau$ and velocity $0$. 

We need Gibbs measures with average velocity different from $0$ and we
use the following notation: 
\begin{equation}\label{eq:lamN}
\lambN(d{\bf r}, d{\bf p}) = \prod_{i=1}^N\lamb(dr_i, dp_i).
\end{equation}
Observe that $S_N$ and $\tilde S_N$ are symmetric with respect to $\lambN$ for any choice of $\bar p$ and $\tau$.

Denote by $\mu_t^N$ the probability measure, on $\R^{2N}$, of the system a time $t$. 
The density $f_t^N$ of $\mu_t^N$ with respect to  {$\lambda^N = \lambda^N_{\beta,0,0}$} solves the Fokker-Plank equation
\begin{equation}\label{eq:ftN}
\frac{\partial f_t^N}{\partial t} = \left(\mathcal G_N^{\bar \tau(t)} \right)^* f_t^N.
\end{equation}
Here $\left(\mathcal G_N^{\bar \tau(t)} \right)^* = -N L_N^{\bar \tau(t)} + N\bar\tau(t) p_N +N\sigma(S_N+\tilde S_N)$ 
is the adjoint of $\mathcal G_N^{\bar \tau(t)}$ with respect to $\lambda^N$.

Define the relative entropy
\begin{equation} \label{eq:HN}
H_N(f_t^N) :=\int_{\R^{2N}} f_t^N \log f_t^N d \lambda^N 
\end{equation}
and the Dirichlet forms 
\begin{equation} \label{eq:DN}
\begin{gathered}
\mathcal D_N(f_t^N) := \sum_{i=1}^{N-1}\int_{\R^{2N}} \frac{1}{4f_t^N}  
\left(\frac{\partial f_t^N}{\partial p_{i+1}}- \frac{\partial f_t^N}{\partial p_i}\right)^2d\lambda^N,
\\
\widetilde {\mathcal D}_N(f_t^N) :=\sum_{i=1}^{N-1} \int_{\R^{2N}} \frac{1}{4f_t^N}  
\left(\frac{\partial f_t^N}{\partial r_{i+1}}- \frac{\partial f_t^N}{\partial r_i}\right)^2d\lambda^N.
\end{gathered}
\end{equation}
We assume there is a constant $C_0$ independent of $N$ such that 
\begin{equation}
H_N(0) \le C_0N.  \label{eq:entbound}
\end{equation}
Since the noise does not conserve the energy, we are interested in the
macroscopic behaviour of the volume stretch and momentum of the particles,
at time $t$, as $N \to \infty$. Note that $t$ is already the
macroscopic time, as we have already multiplied by $N$ in the
generator. We shall use Lagrangian coordinates, that is our space
variables will belong to the lattice $\{1/N, \dots, (N-1)/N, 1\}$. 

 {Consequently, we set ${\bf u}_i := (r_i, p_i)^\intercal$. For a fixed macroscopic time $T$, we  introduce the empirical measures on $[0,T] \times [0,1]$ representing the space-time distributions on the interval $[0,1]$ of volume stretch and momentum: }
\begin{equation}\label{eq:empmeas}
{\pmb \zeta}_N(dx,dt) := \frac{1}{N} \sum_{i=1}^N \delta \left(x-\frac{i}{N}\right)  \uv_i(t) dx\; dt.
\end{equation}
We expect that the measures ${\pmb \zeta}^N(dx,dt)$ converge, as $N \to
\infty$ to an absolutely continuous measure with densities $r(t,x)$ and $p(t,x)$, satisfying the
following system of conservation laws: 
\begin{equation} \label{eq:psystem}
\begin{cases}
 \partial_t  r (t,x) -  \partial_x  p(t,x)=0
 \\
 \partial_t  p(t,x)- \partial_x \tau_\beta( r(t,x))=0,
\end{cases}
\qquad p(t,0)=0, \quad \tau_\beta(r(t,1)) = \bar \tau(t).
\end{equation}
Since \eqref{eq:psystem} is a hyperbolic system of nonlinear partial
differential equation, its solutions may develop shocks in a finite
time, even if smooth initial conditions are given. Therefore, we shall
look for \emph{weak} solutions, which are defined even if
discontinuities appear.

\begin{defn}
  { We say that $(r(t,x), p(t,x))^\intercal \in \left[L^2_{loc}(\mathbb R_+\times [0,1])\right]^2$} 
 { is a \emph{weak solution} of the system} \eqref{eq:psystem}  {provided}
  \begin{equation} \label{eq:maineqr} 
   {  \int_0^\infty \int_0^1\left( r(t,x) \partial_t \varphi(t,x) - p(t,x) \partial_x\varphi(t,x)\right) dx\;  dt }
    = 0
  \end{equation}
  \begin{equation} \label{eq:maineqp} 
  {   \int_0^\infty \int_0^1 \left( p(t,x) \partial_t\psi(t,x) -\tau_\beta( r(t,x)) \partial_x\psi(t,x) \right) dx\;dt}
    =0
  \end{equation}
  { for all functions $\varphi, \psi \in C^2(\mathbb R_+ \times [0,1])$ with compact support
  on $\mathbb R_+ \setminus \{0\} \times (0,1)$.}
\end{defn}
\begin{oss}
  {Notice that this definition of weak solution does not give any information on boundary conditions 
nor about initial conditions. }
\end{oss}

Denote by $\frak Q_N$ the probability distribution of ${\pmb \zeta}_N$ on $\mathcal M([0,T]\times[0,1])^2$. Observe that  ${\pmb \zeta}_N\in \mathcal C([0,T], \mathcal M([0,1])^2)$, where $\mathcal M([0,1])$ 
is the space of signed measures on $[0,1]$, endowed by the weak topology.
Our aim is to show the convergence
\begin{equation}\label{eq:conv}
{\pmb \zeta}_N(T,J) \to \left( \int_0^T \int_0^1 J(t,x)r(t,x) dx dt, \int_0^T\int_0^1 J(t,x) p(t,x) dx dt \right)^\intercal,
\end{equation}
where $r(t,x)$ and $p(t,x)$ satisfy \eqref{eq:maineqr}-\eqref{eq:maineqp}. 
 Since we do not have uniqueness for the solution of these equations, we need 
a more precise statement. 
\begin{theorem}[Main theorem] \label{thm:main}
Assume that the initial distribution satisfies the entropy bound \eqref{eq:HN}.
Then sequence $\mathfrak Q_N$ is compact and any limit point of $\mathfrak Q_N$ 
has support on absolutely continuous measures with densities $r(t,x)$ and $p(t,x)$
 solutions of \eqref{eq:maineqr}-\eqref{eq:maineqp}. 
\end{theorem}

\begin{oss}
Since we are dealing with possibly discontinuous solutions, 
it is not possible to use the entropy method to perform the hydrodynamic limit. 
Furthermore, we shall \emph{not assume} that solutions of \eqref{eq:maineqr}-\eqref{eq:maineqp} exists, 
but we \emph{prove} existence as part of the proof of Theorem \ref{thm:main}. 
\end{oss}

Following Theorem \ref{thm:main}, we discuss the thermodynamics of the system, 
 in particular that the isothermal transformation we have obtained in the hydrodynamic limit satisfies 
the first and second principle of thermodynamics. A mathematical deduction of this requires some 
further assumption that are:
\begin{itemize}
\item any limit distribution of the momentum and stretch profiles $\mathfrak Q$ 
is concentrated on certain \emph{vanishing viscosity solutions} 
(see definition in appendix \ref{sec:visc-appr}), 
\item these solutions reach equilibrium as time approach infinity.
\end{itemize}
A further technical assumption is that the hydrodynamic limit is valid for quadratic functions of the profiles, 
like the energy.
For this purpose we have to define the \emph{macroscopic} work $W$ done by the system. 
Under the weak formulation of the equations \eqref{eq:maineqr}-\eqref{eq:maineqp}
 this is impossible without further conditions. 
We obtain the following theorem.
\begin{theorem}\label{thm:thermo}
Let $\tau$, $U$, $W$, $Q$, $F$ as in \eqref{eq:gamrho}, \eqref{eq:internal}, \eqref{eq:lavoro}, \eqref{eq:Q}, \eqref{eq:freedef}. 
Then, under the assumptions in Section \ref{sec:thermo}, we have
\begin{equation} \label{eq:thermo1}
U(\beta, \tau_1) - U(\beta,\tau_0) = W+Q.
\end{equation}
and
\begin{equation} \label{eq:thermo2}
F(\beta, \tau_\beta^{-1}(\tau_1))-F(\beta,\tau_\beta^{-1}(\tau_0)) \le W.
\end{equation}
\end{theorem}
\begin{oss}
Equation \eqref{eq:thermo1} expresses the first law of thermodynamics, 
and is deduced directly from the microscopic dynamics. 
The main assumption here is the the convergence of the energy, which is quadratic in the positions and the momenta.
In fact, \ref{thm:main} allows us to pass the weak limit inside nonlinear functions
 with strictly less than quadratic growth, but we can say nothing if the growth is quadratic.

Equation \eqref{eq:thermo2} is the \emph{inequality of Clausius}. 
It is equivalent to the second law of thermodynamics for an isothermal transformation: 
$\Delta S \ge \beta Q$. From a PDE point of view, on the other hand, the inequality of Clausius 
reads as a Lax-entropy inequality, provided $W = 0$. 
The presence of $W$ is due to the presence of boundary terms.  
In fact, the work $W$ depends  on the external tension $\bar \tau$. 
The inequality of Clausius is strictly connected to the possible presence of shocks in the solutions obtained 
in Theorem \ref{thm:main}. 
 { In fact global smooth solutions imply \emph{equality} in} \eqref{eq:thermo2}.
\end{oss}

\section{The Hydrodynamic Limit} \label{sec:hydrolim}

Since the temperature $\beta^{-1}$ is fixed throughout the article, in
order to simplify notations we fix $\beta = 1$ in most sections.

\subsection{Approximate Solutions}
In this section we construct a family $\{\hat{\bf u}_N\}_{N \in \mathbb N}$ 
of stochastic processes which solve an approximate version of  \eqref{eq:maineqr}-\eqref{eq:maineqp}. 

For any $1 \le l \le N$ and $l \le i \le N-l+1$ we define the block average:
\begin{equation} \label{eq:uhat}
\uhat_{l,i}:=(\hat r_{l,i},\hat p_{l,i})^\intercal := \frac{1}{l} \sum_{|j|<l} \frac{l-|j|}{l}\uv_{i-j}.
\end{equation}
We choose $l = l(N)$ such that
\begin{equation}\label{eq:l(N)}
\lim_{N \to \infty}\frac{l}{\sigma} = \lim_{N \to \infty}\frac{N \sigma}{l^3} = 0
\end{equation}
and we define the following empirical process:
\begin{equation} \label{eq:uhatN}
\uhat_N(t,x):= (r_N(t,x),p_N(t,x))^\intercal:= \sum_{i=l}^{N-l+1} 1_{N,i}(x) \uhat_{l,i}(t), \quad (t,x) \in \mathbb R_+ \times [0,1],
\end{equation}
where $1_{N,i}$ is the indicator function of the ball of center $i/N$ and diameter $1/N$.
 Note that, since $l/N \to 0$, for $N$ large enough $\uhat_N(t,\cdot)$ is compactly supported in $(0,1)$.
 
 {We use the average} \eqref{eq:uhat}  {has it is smoother than the  mean} $\bar \uv_{l,i}:= \dfrac{1}{l} \overset{l}{\underset{j=1}{\sum}} \uv_{i-j}$  {and thus provides better estimates as $N\to \infty$} (see Lemma \ref{lem:twoblock}, Lemma \ref{lem:etahat} and Corollary \ref{cor:twoblock}).

The proof of the first part of Theorem \ref{thm:main} relies on the following lemma, 
which will be proven in Section \ref{sub:estimates}:
\begin{lem}[Energy estimate] \label{lem:enest}
{ For any time $t\ge 0$} there exists $C_e(t)$ independent of $N$ such that
\begin{equation}\label{eq:enest}
\E \left[\sum_{i=1}^N |\uv_i(t)|^2\right] \le C_e(t) N. 
\end{equation}
\end{lem}

\begin{lem} \label{lem:i)main}
 {For all $t\ge 0$, $\delta>0$ and any test function $J\in \mathcal C^1([0,1])$:}
\begin{equation} \label{eq:ukhatuN}
\lim_{N \to \infty}   \Pbb \left\{\left| \frac{1}{N} \sum_{i=1}^N  J\left(\frac{i}{N}\right)  \uv_i(t)- \int_0^1 J(x) \uhat_N(t ,x) dx \right|  > \delta \right\} = 0 ,
\end{equation}
\end{lem}
\begin{proof}
First observe that boundary terms are negligeable since
\begin{equation*}
  \left| \frac 1N \sum_{i=1}^l J\left(\frac{i}{N}\right) \uv_i \right| \le\left(\frac 1N \sum_{i=1}^l J\left(\frac{i}{N}\right)^2\right)^{1/2} 
  \left(\frac 1N \sum_{i=1}^N \uv_i^2\right)^{1/2} \le  \|J\|_\infty \sqrt{\frac{l}{N}} \left(\frac 1N \sum_{i=1}^N \uv_i^2\right)^{1/2}
\end{equation*}
and similarly on the other side. Then we estimate separately
\begin{equation}\label{2est}
\begin{gathered}
\left| \frac{1}{N} \sum_{i=l+1}^{N-l-1}  J\left(\frac{i}{N}\right)  \left(\uv_i - \uhat_{l,i}\right) \right|\\
 \left| \frac{1}{N} \sum_{i=l+1}^{N-l-1}  J\left(\frac{i}{N}\right)  \uhat_{l,i} - \int_0^1 J(x) \uhat_N(x)\; dx \right|.
\end{gathered}
\end{equation}
Using that $\dfrac{1}{l}\underset{|j|<l}{\sum}\dfrac{l-|j|}{l}=1$, we have
\begin{equation}
\begin{gathered}
\left|\frac{1}{N} \sum_{i=l+1}^{N-l-1}  J\left(\frac{i}{N}\right)  \left(\uv_i - \uhat_{l,i}\right) \right|
=\left| \frac{1}{N}\sum_{i=l}^{N-l+1}  \frac{1}{l}\sum_{|j|<l} \frac{l-|j|}{l} \left( J \left( \frac{i}{N}\right)-J \left(\frac{i+j}{N}\right) \right) \uv_i\right|\\
   \le \norminf[J']  \frac{1}{N}\sum_{i=l}^{N-l+1} \frac{1}{l} \sum_{|j|<l} \frac{l-|j|}{l} \frac{|j|}{N} | \uv_i|  
   \le \norminf[J']  \frac{l}{N^2}\sum_{i=1}^{N}  | \uv_i|  \\
   \le \norminf[J']  \frac{l}{N} \left(\frac 1N \sum_{i=1}^{N}  | \uv_i|^2\right)^{1/2}.
 \end{gathered}
\end{equation}
Similarly for the second of \eqref{2est}: 
\begin{equation*}
\begin{gathered}
\left| \frac{1}{N} \sum_{i=l}^{N-l+1} \left( J \left(\frac{i}{N} \right)- N \int_{i/N-1/(2N)}^{i/N+1/(2N)}J(x)dx \right) \uhat_{l,i} \right| \le 
 \frac{ \norminf[J']}{N^2}  \sum_{i=l}^{N-l+1} |\uhat_{l,i}| \le \frac{ \norminf[J']}{N^2}  \sum_{i=1}^{N} |\uv_i|. 
\end{gathered}
\end{equation*}
\end{proof}
 { It follows from Lemmas} \ref{lem:enest} and \ref{lem:i)main}  {that} $\uhat_N(t,x)$ has values in $L^2([0,T]\times[0,1])^2$. Let us denote 
by $\tilde{\mathfrak Q}_N$ the distribution of $\uhat_N(t,x)$ on $L^2([0,T]\times[0,1])$.
We will show in the following that any convergent subsequence of $\tilde{\mathfrak Q}_N$ is 
concentrated on the weak solutions of \eqref{eq:psystem}. By Lemma \ref{lem:i)main} this implies 
the conclusion of the main Theorem \ref{thm:main} for any limit point of ${\mathfrak Q}_N$.

From the interaction $V$ we define
\begin{equation}\label{eq:hatV'}
\hat V'_{l,i}(t):= \frac{1}{l}\sum_{|j|<l}\frac{l-|j|}{l}V'(r_{i-j}(t)).
\end{equation}
It follows from \eqref{eq:Vconvex} that $V'$ is linearly bounded. From Appendix \ref{app:tension}, so is  {$\tau = \tau_{\beta = 1}$} (as defined by \eqref{eq:gamrho}). Thus we easily obtain the following from lemma \ref{lem:i)main} :
\begin{cor} \label{cor:appr}
\begin{equation} \label{eq:V'appr}
\lim_{N \to \infty}   \Pbb \left\{\left| \frac{1}{N} \sum_{i=1}^N  J\left(\frac{i}{N}\right) V'(r_i(t))- \frac{1}{N}\sum_{i=l}^{N-l+1} J\left(\frac{i}{N}\right)\hat V'_{l,i}(t)  \right|  > \delta \right\} = 0,
\end{equation}
\begin{equation}
\lim_{N \to \infty}   \Pbb \left\{\left| \frac{1}{N} \sum_{i=l}^{N-l+1}  J\left(\frac{i}{N}\right) \tau(\hat r_{l,i}(t))-\int_0^1 J(x) \tau(\hat r_N(t,x))  \right|  > \delta \right\} = 0,
\end{equation}
for all $t\ge 0, \ \delta >0$ and all test functions  {$J \in \mathcal C^1( [0,1])$}.
\end{cor}
The following theorem will be proven in Section \ref{sub:estimates}. Recall that $l = l(N)$ that satisfies \eqref{eq:l(N)}.
\begin{theorem}[One-block estimate] \label{thm:1block}
\begin{equation}\label{eq:1block}
{\limN\E\left[\frac{1}{N} \sum_{i=l}^{N-l+1} \int_0^t \left( \hat V'_{l,i}(s)-\tau(\hat r_{l,i}(s))\right)^2ds \right] =0.}
\end{equation}
\end{theorem}
We are now in a position to prove the following:
\begin{prop}\label{prop:approxx}
Let $\varphi$ and $\psi$ be as in \eqref{eq:maineqr}-\eqref{eq:maineqp}. Then
\begin{equation} \label{eq:eqrappr}
\limN \Pbb \left\{\left |\int_0^\infty \int_0^1 \hat r_N(t,x) \partial_t \varphi(t,x) - \hat p_N(t,x)  \partial_x\varphi(t,x)dx dt \right| >\delta \right\} =0
\end{equation}
\begin{equation} \label{eq:eqpappr}
\limN \Pbb \left\{\left |\int_0^\infty \int_0^1  \hat  p_N(t,x) \partial_t\psi(t,x) -\tau( \hat r_N(t,x))  \partial_x\psi(t,x) dxdt 
\right|>\delta \right\}=0
\end{equation}
for any $\delta >0$.
\end{prop}
\begin{proof}
We prove \eqref{eq:eqpappr}, as the proof of \eqref{eq:eqrappr} is analogous and technically easier. We denote
\begin{equation}\label{eq:psii}
\psi_i(t) := \psi \left( t,\frac{i}{N}\right)
\end{equation}
and, for any sequence $(x_i)$,
\begin{equation}\label{eq:deltalap}
\nabla x_i := x_{i+1}-x_i, \quad \nabla^*x_i:= x_{i-1}-x_i, \quad \Delta x_i := x_{i+1}+x_{i-1}-2x_i.
\end{equation}
Using \eqref{eq:SDE} we can compute the time evolution:
\begin{equation}
  \label{eq:timeev}
  \begin{split}
    \frac{1}{N} \sum_{i=1}^N \psi_i(T) p_i(T) -  \frac{1}{N}
    \sum_{i=1}^N \psi_i(0) p_i(0)  =
    \int_0^T \frac{1}{N} \sum_{i=1}^N  \dot\psi_i(t) p_i(t) \; dt\\
    + \int_0^T  \left[\sum_{i=1}^{N-1}  \psi_i(t) \left(V'(r_{i+1}(t)) -
      V'(r_{i}(t)) \right) + \psi(t,1) \left(\bar\tau(t) -
      V'(r_N(t))\right)\right] dt \\
  +  \int_0^T \sigma
    \sum_{i=2}^{N-1} \psi_i(t)  \Delta p_i(t) dt 
 + \int_0^T \sqrt{2 \frac{\sigma}{N}}  \sum_{i=1}^{N-1} \psi_i(t)
 \nabla^* d  w_i(t) 
  \end{split}
\end{equation}
where in the above equation we have set $p_0, p_{N+1}, w_N$ identically equal to
$0$. 

The second line of \eqref{eq:timeev} is equal to
\begin{equation}
  \label{eq:2ndline}
  \int_0^T  \sum_{i=2}^{N} (\nabla^* \psi(t))_i V'(r_i(t))  dt
  - \int_0^T \psi_1(t)  V'(r_1(t)) dt  + \int_0^T \bar\tau(t)
  \psi(t,1) dt.
\end{equation}
Since $\psi(t,0) = 0 = \psi(t,1)$ and $\psi$ has continuous derivatives in
$[0,1]$, we have that $\psi_1(t) = \psi(t, N^{-1})\sim O(N^{-1})$ and the second term
of \eqref{eq:2ndline} is negligible, while the third is identically null.

The last line of \eqref{eq:timeev}, depending on $\sigma$, can be
rewritten as
\begin{equation}
  \label{eq:evsigma}
  \begin{split}
    \int_0^T \Big\{\frac{\sigma}{N^2} \sum_{i=2}^{N-1} N^2
      (\Delta \psi (t))_i p_i(t) 
  - \frac{\sigma}{N}\left[N (\nabla  \psi (t))_1 p_1(t) + N (\nabla^* \psi
  (t))_N p_N(t)\right] \\
  + \sigma \psi_N(t) \nabla p_{N-1}(t) -
  \sigma \psi_1(t) \nabla p_1(t) \Big\} dt\\
 +  \int_0^T \sqrt{2 \frac{\sigma}{N}} \left(\sum_{i=1}^{N-2} \nabla \psi_i
 dw_i  -  \psi_{N-1} dw_{N-1} \right).
\end{split}
\end{equation}
Since $\psi$ is twice differentiable,
\begin{equation}\label{eq:gradpsi}
N \nabla \psi_i(t) = \partial_x \psi\left(t, \frac{i}{N}\right) + \mathcal O\left(\frac{1}{N}\right), \quad  N^2 \Delta \psi_i(t) = \partial^2_{xx} \psi\left(t, \frac{i}{N}\right) + \mathcal O\left(\frac{1}{N}\right)
\end{equation}
as $N \to \infty$.
This, together with $\sigma/N \to 0$ and the energy estimate imply
that the first line of \eqref{eq:evsigma} vanish
as $N \to \infty$. For the same reason, the quadratic variation of the
stochastic integrals (last line of \eqref{eq:evsigma}) also vanish in
the limit.

Also the second line of \eqref{eq:evsigma} will be negligible for the same reason, since 
$\psi(t,0)= 0 = \psi(t,1)$.





The conclusion then follows by replacing the sums
with integrals, $p_i$ by $\hat p_N(t,x)$ and $V'(r_i)$ by $\tau(\hat
r_N(t,x))$ accordingly to Lemma \ref{lem:i)main}, Corollary
\ref{cor:appr} and Theorem \ref{thm:1block}. 
\end{proof}
Once we have \eqref{eq:eqrappr} and \eqref{eq:eqpappr}, we deduce that 
the distributions $\tilde{\mathfrak Q}_N$ of
 $\uhat_N(t,x)$ concentrate on the solutions of the macroscopic equations \eqref{eq:maineqr}, 
\eqref{eq:maineqp} as follows.
We associate to $\uhat_N(t,x)$ the random Young measure 
$\nuhat^N = \delta_{\uhat_N(t,x)}$. We show that the sequence $(\nuhat^N)_{N \ge 0}$ is compact, 
in an appropriate probability space, and converges weakly-* to a measure $\nutilde$. 
Since \eqref{eq:eqrappr} is linear, we are done for it.
Concerning \eqref{eq:eqpappr} (in which the nonlinear unbounded function $\tau$ appears) we prove that
\begin{equation}
\tau(\hat r_N(t,x)) = \int_{\R^2} \tau(y_1) d\nuhat^N(y_1,y_2) \overset{\ast}{\rightharpoonup} \int_{\R^2} \tau(y_1) d\nutilde(y_1,y_2) \qquad \text{in weak}^\ast\text{-}L^\infty,
\end{equation}
which is not obvious, since weak-* convergence is not enough to pass limits inside unbounded functions like $\tau$.

Finally, using the theory of compensated compactness, we reduce the support of the limit Young measure $\nutilde$ to a point, that is $\nutilde = \delta_{\utilde(t,x)}$, for some function $\utilde(t,x)=(\tilde r(t,x), \tilde p(t,x))^\intercal$, almost surely and for almost all $t,x$. This closes the equation, as it implies
\begin{equation}
\tau(\hat r_N(t,x)) \to \tau(\tilde r(t,x)).
\end{equation}
What we have just presented is only a sketchy statement of what is extensively proven in the rest of this paper. 
In particular, extra care is taken when applying the theory of Young measures and compensated compactness to a stochastic setting like ours.

\subsection{Convergence of the Empirical Process}
Proposition \ref{prop:approxx} was a first step in proving Theorem \ref{thm:main}. In this section we complete the proof.

This is done using random Young measures and a stochastic extension of the theory of compensated compactness. 
We refer to Sections 4 and 5 of \cite{stochasticYoung} for the definitions and results concerning random Young measures. 
\subsubsection{Random Young Measures and Weak Convergence}
Denote by $\nuhat^N = \delta_{\uhat_N(t,x)}$ the random Young measure on $\R^2$ associated to the {empirical} 
 process $\uhat_N(t,x)$:
\begin{equation}
 \int_{\R^2}f(\y)d\nuhat^N(\y) = f(\uhat_N(t,x))
\end{equation}
for any $f : \R^2 \to \R$.  {Set $Q_T=(0,T) \times (0,1)$ for any $T>0$.} Since $\uhat_N \in L^2(\Omega \times Q_T)^2$, we say that $\nuhat^N$ is a $L^2$-random Dirac mass. The chain of inequalities
\begin{equation}\label{eq:enyoung1}
\begin{gathered}
{\E \left[ \int_{Q_T} \int_{\R^2}|\y|^2 d \nuhat^N(\y)dxdt \right]=\E \left[ \normp[\uhat_N]{2}^2 \right] = \mathbb E \left[ \int_{Q_T} |\uhat_N(t,x)|^2dxdt \right]}
\\
\le \int_{Q_T} \sum_{i=l}^{N-l+1} 1_{N,i}(x) |\uhat_{l,i}(t)|^2dxdt \le \frac{1}{N}\int_0^T \sum_{i=l}^{N-l+1}|\uhat_{l,i}(t)|^2  \le 4 \int_0^T C_e(t)dt = \tilde C(T)
\end{gathered}
\end{equation}
with $\tilde C(T)$ independent of $N$, implies that there exists a subsequence of random Young measures $(\nuhat^{N_n})$ and a subsequence of real random variables $(\normp[\uhat_{N_n}]{2})$ that converge in law.

We can now apply the Skorohod's representation theorem to the laws of $( \nuhat^{N_n
},\normp[\uhat_{N_n}]{2})$ and find a common probability space such that the convergence happens almost surely. This proves the following proposition:
\begin{prop}\label{prop:skoro}
There exists a probability space $(\tilde \Omega, \tilde{\mathcal F}, \Pbbtilde)$,  random Young measures $\nutilde^n,\nutilde$ and real random variables $a_n,a$ such that $\nutilde^n$ has the same law of $\nuhat^{N_n}$, $a_n$ has the same law of $\normp[\uhat_{N_n}]{2}$ and $\nutilde^n \weakstar \nutilde$, $a_n \to a$, $\Pbbtilde$-almost surely.
\end{prop}
Since $\nuhat^{N_n}$ is a random Dirac mass and $\nutilde^n$ and $\nuhat^{N_n}$ have the same law, $\nutilde^n$ is a $L^2$-random Dirac mass, too: $\nutilde^n = \delta_{\utilde_n(t,x)}$ for some $\utilde_n \in L^2(\tilde \Omega \times Q_T)^2$.  $\utilde_n$ and $\uhat_{N_n}$ have the same law. Since $a_n \to a$ almost surely, we have that $(a_n)$ is bounded and so $\normp[\utilde_n]{2}$ is bounded uniformly in $n$ with $\Pbbtilde$-probability $1$. Since from a uniformly bounded sequence in $L^p$ we can extract a weakly convergent subsequence, we obtain the following proposition:
\begin{prop}\label{prop:weakconv}
There exist $L^2(Q_T)^2$-valued random variables $(\utilde_n), \utilde$ such that $\utilde_n$ and $\uhat_{N_n}$ have the same law and, $\Pbbtilde$-almost surely and up to a subsequence, $\utilde_n \weak \utilde$ in $L^2(Q_T)$.
\end{prop}
The condition $\nutilde^n \weakstar \nutilde$ in Proposition \ref{prop:skoro} reads
\begin{equation}
\lim_{n\to \infty}  \int_{\R^2} f(\y) d\nutilde^n(\y) dxdt =  \int_{\R^2}f(\y) d\nutilde(\y) dxdt
\end{equation}
for all continuous and \emph{bounded} $f : \mathbb R^2 \to \mathbb R$. The next proposition extend this result to functions $f$ with subquadratic growth:
\begin{prop} \label{prop:momconv}
There is a constant $C$ independent of $n$ such that
\begin{equation} \label{eq:ennutilde}
\E \left[\int_{Q_T}\int_{\R^2} |\y|^2 d\nutilde(\y)dxdt \right] \le C.
\end{equation}
Furthermore, let $J : Q_T \to \R$ and $f: \R^2 \to \R$ be continuous, with  $f(\y)/|\y|^2 \to 0$ as $|\y|\to \infty$. We have
\begin{equation} \label{eq:momconv}
\lim_{n\to \infty} \mathbb E \left[  \left| \int_{Q_T}\int_{\R^2}J(t,x) f(\y) d \nutilde^n(\y)dxdt- \int_{Q_T}\int_{\R^2}J(t,x) f(\y) d \nutilde(\y)dxdt \right| \right] = 0
\end{equation}
\end{prop}
\begin{proof}
Since $\nutilde^n$ and $\nuhat^{N_n}$ have the same law,  \eqref{eq:enyoung1} imply
\begin{equation}
\E \left[\int_{Q_T}\int_{\R^2} |\y|^2 d\nutilde^n(\y)dxdt \right] \le C
\end{equation}
for some constant $C$ independent of $n$.

Let $\chi : \mathbb R \to \mathbb R$ be a continuous, non-negative non-increasing function supported in $[0, 2]$ which is identically equal to $1$ on $[0,1]$. For $R >1$ and $a \in \mathbb R$ define $\chi_R(a):= \chi\left(a/R\right)$. By the monotone convergence theorem,
\begin{equation}
\int_{Q_T}\int_{\R^2}  |\y|^2 d\nutilde(\y)dxdt = \lim_{R \to \infty} \int_{Q_T}\int_{\R^2} |\y|^2 \chi_R(|\y|^2) d \nutilde(\y)dxdt.
\end{equation}
Since now $|\y|^2 \chi_R(|\y|^2)$ is continuous and bounded, we have, almost surely,
\begin{equation}
\int_{Q_T}\int_{\R^2}  |\y|^2 \chi_R(|\y|^2) d \nutilde(\y)dxdt = \lim_{n \to \infty}\int_{Q_T}\int_{\R^2}  |\y|^2 \chi_R(|\y|^2) d \nutilde^n(\y)dxdt.
\end{equation}
Then, applying the Fatou lemma twice, we get
\begin{equation}
\begin{gathered}
\mathbb E \left[ \int_{Q_T}\int_{\R^2}  |\y|^2 d \nutilde(\y)dxdt \right] \le \liminf_{R \to\infty} \mathbb E \left[ \int_{Q_T}\int_{\R^2}  |\y|^2 \chi_R(|\y|^2) d\nutilde(\y)dxdt\right]
\\
\le \liminf_{R \to \infty} \liminf_{n \to\infty} \mathbb E \left[ \int_{Q_T}\int_{\R^2}  |\y|^2 \chi_R(h(\xi))d\nutilde^n(\y)dxdt\right] \le C,
\end{gathered}
\end{equation}
which proves \eqref{eq:ennutilde}.

Define
\begin{equation}
I_n := \int_{Q_T}\int_{\R^2}  J(t,x) f(\y)d\nutilde^n(\y)dxdt, \quad I:=\int_{Q_T}\int_{\R^2}  J(t,x) f(\y) d\nutilde(\y)dxdt.
\end{equation}
so that \eqref{eq:momconv} reads
\begin{equation}
\lim_{n \to \infty} \mathbb E[|I_n -I|] = 0.
\end{equation}
We further define
\begin{equation}
\begin{gathered}
I^R_n := \int_{Q_T}\int_{\R^2} J(t,x) f(\y) \chi_R\left(\frac{|\y|^2}{1+|f(\y)|}\right)d\nutilde^n(\y)dxdt,
\\
I^R := \int_{Q_T}\int_{\R^2} J(t,x) f(\y) \chi_R\left(\frac{|\y|^2}{1+|f(\y)|}\right)d\nutilde^n(\y)dxdt,
\end{gathered}
\end{equation}
and estimate
\begin{equation}
\mathbb E[| I_n - I|] \le \mathbb E[|I_n - I_n^R|] + \mathbb E[|I_n^R - I^R|] + \mathbb E[|I^R - I|].
\end{equation}
The first term on the right hand side estimates as follows:
\begin{equation}
\mathbb E[|I_N - I_N^R|] \le \mathbb E \left[\int_{Q_T}\int_{\R^2} |J(t,x)| |f(\y)|\left(1-\chi_R\left(\frac{|\y|^2}{1+|f(\y)|}\right)\right)  d\nutilde^n(\y)dxdt \right].
\end{equation}
Since and $f(\y)/|\y|^2 \to 0$ as $|\xi| \to \infty$ we have
\begin{equation}
1-\chi_R\left(\frac{|\y|^2}{1+|f(\y)|}\right) \le  1_{\frac{|\y|^2}{1+|f(\y)|}> R}(\y) \le \frac{|\y|^2}{R(1+|f(\y)|)},
\end{equation}
for any $R >1$. This implies
\begin{equation}
\mathbb E[|I_n - I_n^R|] \le  \frac{\norminf[J]}{R} \mathbb E\left[\int_{Q_T}\int_{\R^2}|\y|^2 d\nutilde^n(\y)dxdt \right] \le \frac{C\norminf[J]}{R}. 
\end{equation}
For \eqref{eq:ennutilde} we have as well
\begin{equation}
\mathbb E[|I - I^R|] \le\frac{C\norminf[J]}{R},
\end{equation}
which gives
\begin{equation}
\mathbb E[|I_n - I|] \le  \frac{2C\norminf[J]}{R} + \mathbb E[|I_n^R - I^R|].
\end{equation}
Since $f$ may diverge only at infinity and $f(\y)/|\y|^2 \to 0$ as $|\y| \to \infty$, then $f(\y) \chi_R\left(\dfrac{|\y|^2}{1+|f(\y)|}\right)$ is continuous and bounded and hence $I_n^R \to I^R$  almost surely. Moreover, thanks to the uniform estimate
\begin{equation}
\mathbb E \left[ |I_n^R| \right] \le \norminf[J] C,
\end{equation}
the sequence $(I_n^R)$ is equi-integrable. Therefore, by the Vitali theorem,
\begin{equation}
\lim_{n \to \infty}\mathbb E [ |I_n^R - I^R |]=0,
\end{equation}
which completes the proof.

\end{proof}
We are interested in the weak limit of $\tau(\hat r_N(t,x))$. Since $\tau$ is linearly bounded, the previous proposition applies and the main theorem \ref{thm:main} is proved once we show that $\nutilde = \delta_{\utilde(t,x)}$, almost surely and for almost all $(t,x) \in Q_T$.

In the next two subsections we shall prove  that the support of $\nutilde$ is almost surely and almost everywhere a point. The result will then follow from the lemma:
\begin{lem}
$\nutilde = \delta_{\utilde(t,x)}$ almost surely and for almost all $(t,x) \in Q_T$ if and only if the support of $\nutilde$ is a point for almost all $(t,x) \in Q_T$. In this case, $\utilde_n \to \utilde$ in $L^p(Q_T)^2$-strong for all $1 \le p<2$.
\end{lem}
\begin{proof}
Suppose there is a measurable function ${\bf u}^*: Q_T \to \R^2$ such that $\nutilde = \delta_{{\bf u}^*(t,x)}$ for almost all $(t,x) \in Q_T$. For any test function ${\bf J} : Q_T \to \R^2$ consider the quantity
\begin{equation}
\int_{Q_T} {\bf J}(t,x) \cdot \utilde_n(t,x)dxdt = \int_{Q_T} \int_{\R^2} {\bf J}(t,x) \cdot \y d\nutilde^n(\y)dxdt.
\end{equation}
By taking the limit for $n \to \infty$ {in the sense} of $L^2$-weak first and in the sense of \eqref{eq:momconv} then, we obtain
\begin{equation}
\int_{Q_T} \int_{\R^2} {\bf J}(t,x) \cdot \utilde(t,x)dxdt =\int_{Q_T} \int_{\R^2} {\bf J}(t,x) \cdot \y d\nutilde(\y)dxdt = \int_{Q_T} \int_{\R^2} {\bf J}(t,x) \cdot {\bf u}^*(t,x)dxdt
\end{equation}
almost surely. Then $\utilde(t,x) = {\bf u}^*(t,x)$ for almost all $(t,x) \in Q_T$ follows from the fact that ${\bf J}$ was arbitrary. 

Next, fix $1<p<2$. Taking $f(\y) = |\y|^p$ in \eqref{eq:momconv} gives $\normp[\utilde_n]{p} \to \normp[\utilde]{p}$, which, together with $\utilde_n \weak \utilde$ in $L^p(Q_T)^2$ and the fact that $L^p(Q_T)^2$ is uniformly convex for $1<p<\infty$ implies strong convergence.

The case $p=1$ follows from the result for $p>1$ and {H\"{o}lder}'s inequality.
\end{proof}

\subsubsection{Reduction of the Limit Young Measure}
In this section we prove that the support of $\nutilde$ is almost surely and almost everywhere a point. 

We recall that Lax entropy-entropy flux pair for the system
\begin{equation} \label{eq:psystem1}
\begin{cases}
\partial_t r(t,x)- \partial_x p(t,x) =0
\\
\partial_t p(t,x) - \partial_x\tau(p(t,x))=0
\end{cases}
\end{equation}
is a pair of functions $\eta, q : \R^2 \to \R$ such that
\begin{equation}
\partial_t \eta(\uv(t,x)) + \partial_x q(\uv(t,x)) = 0
\end{equation} 
for any smooth solution $\uv(t,x) = (r(t,x),p(t,x))^\intercal$ of \eqref{eq:psystem1}. This is equivalent to the following:
\begin{equation}
\begin{cases}
 \partial_r\eta(r,p)+\partial_pq(r,p)=0
\\
\tau'(r)\partial_p\eta(r,p) +   \partial_rq(r,p) =0
\end{cases}.
\end{equation}
Under appropriate conditions on $\tau$, Shearer (\cite{Shearer1}) constructs a family of entropy-entropy flux pairs $(\eta,q)$ such that $\eta,q$, their first and their second derivatives are bounded. As we shall see in the appendix, our choice of the potential $V$ ensures that the tension $\tau$ has the required properties, so the result of Shearer applies to our case.

In particular,  {following Section 5 of }\cite{Shearer1}, we have that the support $\nutilde$ is almost surely and almost everywhere a point provided {Tartar's} 
commutation relation
\begin{equation} \label{eq:tartar}
\expnu[\eta_1q_2-\eta_2q_1] = \expnu[\eta_1]\expnu[q_2] - \expnu[\eta_2]\expnu[q_1]
\end{equation}
holds almost surely and almost everywhere for any bounded pairs $(\eta_1,q_1)$, $(\eta_2,q_2)$ with bounded first and second derivatives.

Obtaining \eqref{eq:tartar} in a deterministic setting is standard and relies on the div-curl and {Murat-Tartar} 
 lemma. Both of these lemmas have a stochastic extension ( {cf Appendix A of }\cite{Divcurl}) and what we ultimately need to prove in order to obtain \eqref{eq:tartar} is that the hypotheses for the stochastic Murat-Tartar lemma are satisfied (cf \cite{stochasticYoung} {, Proposition 5.6}). This is ensured by the following theorem, which we will prove in the next section.
\begin{theorem} \label{thm:Xn}
Let $(\eta,q)$ be a bounded Lax entropy-entropy flux pair with bounded first and second derivatives. Let $\varphi : \R^2 \to \R $ be such that $\varphi =  \phi \psi$, with $\phi$  {smooth and compactly supported} in $(0,\infty) \times (0,1)$ and $\psi \in L^\infty(\R_+\times [0,1]) \cap H^1(\R_+\times [0,1])$. Define
\begin{equation}\label{eq:Xn}
\tilde X_n(\varphi,\eta) :=\int_0^\infty \int_0^1 
\left[\partial_t \varphi(t,x) \eta(\utilde_n(t,x)) + \partial_x\varphi(t,x)q(\utilde_n(t,x))\right] \;dx\; dt.
\end{equation}
Then $\tilde X_n$ decomposes as
\begin{equation}
\tilde X_n = \tilde Y_n+\tilde Z_n
\end{equation}
 {and there are $A_n, B_n \in \mathbb R_+$ independent of $\psi$ such that}
\begin{equation} \label{eq:YNZN}
 {\E \left[\left|\tilde Y_n(\phi \psi,\eta)\right|\right] \le A_n \normH[\psi], \quad \E \left[\left| \tilde Z_n(\phi \psi,\eta)\right| \right]\le B_n \norminf[\psi]}
\end{equation}
with
\begin{equation}
\lim_{n\to \infty} A_n = 0, \quad \limsup_{n \to \infty} B_n < \infty.
\end{equation}
\end{theorem}
\begin{oss}
Recall that the $H^1$ norm of a function $f(t,x)$ is defined as $\normH[f] = \normLdue[f]+\normLdue[\partial_tf]+\normLdue[\partial_xf]$. Moreover, from now on, $\phi$ (and hence $\varphi$) will be supported in $[0,T] \times [x_-,x_+]$ for some fixed $T >0$ and $0<x_-<x_+<1$.  {The test function $\phi$ is used to localise the problem. In fact, Murat-Tartar lemma is obtained on bounded domains. Note that we already were on a bounded spacial domain. Nevertheless, $\phi$ ensures we stay away from the boundary, as we are not able to prove Theorem} \ref{thm:Xn}  {otherwise.}
\end{oss}
\subsubsection{Conditions for the Murat-Tartar Lemma}
This section is devoted to proving Theorem \ref{thm:Xn} through a series of lemmas. Since we are ultimately interested in taking expectations of functions of $\utilde_n$, and since $\utilde_n$ and $\uhat_{N_n}$ have the same law, we shall prove the theorem for
\begin{equation}
 X_N(\varphi, \eta) :=\int_0^\infty \int_0^1 \partial_t \varphi(t,x) \eta(\uhat_N(t,x)) + \partial_x\varphi(t,x)q(\uhat_N(t,x))dx dt.
\end{equation}
Recalling that $\uhat_N$ is built from a solution $\uhat_i=(r_i,p_i)^\intercal$ of the system of SDEs \eqref{eq:SDE}, and since $\varphi(t,\cdot)$ is compactly supported in $(0,1)$, It\^o formula yields, for large enough $N$,
\begin{equation} \label{eq:entropyDec}
{X_N = X_{a,N} + X_{s,N}+\tilde X_{s,N} + \mathcal M_N+\tilde{\mathcal M}_N+\mathcal N_N,}
\end{equation}
where
\begin{equation} \label{eq:XaN}
\begin{gathered}
 X_{a,N}(\varphi,\eta) =  \int_0^\infty   \sum_{i=l+1}^{N-l} \bar \varphi_i \left(  \partial_p\eta(\uhat_{l,i}) \nabla \hat V'_{l,i}- \partial_r\eta(\uhat_{l,i}) \nabla^*\hat p_{l,i} \right) dt +
\\
+ \int_0^\infty  \sum_{i=l+1}^{N-l} \bar \varphi_i \left( \partial_r q(\uhat_{l,i}) \nabla \hat r_{l,i} - \partial_pq(\uhat_{l,i})\nabla^* \hat p_{l,i} \right) dt,
\end{gathered}
\end{equation}
\begin{equation} \label{eq:XsN}
\begin{gathered}
X_{s,N}(\varphi,\eta)= \sigma \int_0^\infty \sum_{i=l+1}^{N-l} \bar \varphi_i    \partial_p\eta(\uhat_{l,i}) \Delta \hat p_{l,i}  dt+ \sigma \int_0^\infty \sum_{i=l+1}^{N-l+1}  \bar \varphi_i \partial^2_{pp}\eta(\uhat_{l,i}) ( \nabla^* d \hat w_{l,i})^2 ,
\end{gathered}
\end{equation}
\begin{equation} \label{eq:tilXsN}
\begin{gathered}
\tilde X_{s,N}(\varphi,\eta)= \sigma \int_0^\infty  \sum_{i=l+1}^{N-l} \bar \varphi_i   \partial_r\eta(\uhat_{l,i}) \Delta \hat V'_{l,i}dt+\sigma \int_0^\infty  \sum_{i=l+1}^{N-l+1} \bar \varphi_i \partial^2_{rr}\eta(\uhat_{l,i}) (\nabla^* d  \hat {\tilde w}_{l,i})^2,
\end{gathered}
\end{equation}
\begin{equation} \label{eq:MN}
\mathcal M_N(\varphi, \eta) = - \sqrt{2\frac{\sigma}{N}} \int_0^\infty   \sum_{i=l+1}^{N-l+1}  \bar \varphi_i    \partial_p\eta(\uhat_{l,i}) d \nabla^* \hat w_{l,i},
\end{equation}
\begin{equation} \label{eq:tildeMN}
\tilde{\mathcal M}_N(\varphi, \eta) = - \sqrt{2\frac{\sigma}{N}} \int_0^\infty   \sum_{i=l+1}^{N-l+1}  \bar \varphi_i    \partial_r\eta(\uhat_{l,i})   d \nabla^*\hat{\tilde w}_{l,i}
\end{equation}
and
\begin{equation} \label{eq:NN}
\begin{gathered}
\mathcal N_N(\varphi,\eta) = -\int_0^\infty \int_0^1  \partial_x\varphi(t,x) q(\uhat_N(t,x)) dx d t
\\
- \int_0^\infty  \sum_{i=l+1}^{N-l} \bar \varphi_i \left( \partial_r q(\uhat_{l,i}) \nabla \hat r_{l,i} - \partial_pq(\uhat_{l,i})\nabla^* \hat p_{l,i} \right) dt.
\end{gathered}
\end{equation}
We have set
\begin{equation}\label{eq:barphi}
{\bar \varphi_i(t) = N \int_0^1 \varphi(t,x)1_{N,i}(x) dx = N \int_{i/N-1/(2N)}^{i/N+1/(2N)}\varphi(t,x)dx}
\end{equation}
and
\begin{equation}
\hat w_{l,i} =  \frac{1}{l}\sum_{|j|<l}\frac{l-|j|}{l} w_{i-j}, \quad \hat{\tilde w}_{l,i} =  \frac{1}{l}\sum_{|j|<l}\frac{l-|j|}{l} \tilde w_{i-j}.
\end{equation}
 {The quadratic variation in} \eqref{eq:XsN}  {is evaluated thanks to Lemma} \ref{lem:etahat}  {and the formal identity $d w_i d w_j = \delta_{ij}dt$:}
\begin{equation}
 {(\nabla^*d\hat w_{l,i})^2 = \frac{1}{l^2}(d  \bar  w_{l,i-1+l}-d \bar  w_{l,i-1})^2 
= \frac{1}{l^4} \left( \sum_{j=0}^{l-1} d  w_{i-1+l-j} - \sum_{j=0}^{l-1} d  w_{i-1-j}\right)^2 = \frac{2}{l^3} dt.}
\end{equation}
 {Similarly, the analogous term in} \eqref{eq:tilXsN}  {gives}
\begin{equation}
 {(\nabla^*d\hat{\tilde w}_{l,i})^2  = \frac{2}{l^3} dt.}
\end{equation}
The proof of Theorem \ref{thm:Xn} will rely on the following theorems, which we will prove in Section \ref{sub:estimates}.
\begin{theorem}[One-block estimate - explicit bound]\label{thm:1blockexp}
There is $C_1(t)$ independent of $N$ such that
\begin{equation}\label{eq:1blockexp}
\E \left[ \frac{1}{N} \sum_{i=l}^{N-l+1}\int_0^t \left(\hat V'_{l,i}(s)-\tau(\hat r_{l,i}(s)) \right)^2 ds \right] \le  C_1(t) \left( \frac{1}{l} +\frac{l^2}{N\sigma} \right).
\end{equation}
\end{theorem}
\begin{theorem}[Two-block estimate]\label{thm:2block}
Let $\hat \zeta_{l,i} \in \{\hat r_{l,i},\hat p_{l,i}, \hat V'_{l,i}, \tau (\hat r_{l,i}) \}$. There is $C_2(t)$ independent of $N$ such that
\begin{equation}\label{eq:2block}
\E \left[\frac{1}{N} \sum_{i=l}^{N-l}\int_0^t \left(\hat \zeta_{l,i+1}-\hat \zeta_{l,i} \right)^2 ds \right] \le  C_2(t) \left( \frac{1}{l^3} +\frac{1}{N\sigma} \right).
\end{equation}
\end{theorem}
We prove Theorem \ref{thm:Xn} through a series of lemmas.
\begin{lem} \label{lem:typeZ}
Let $(a_i)_{i \in \mathbb N}$ and $(b_i)_{i \in \mathbb N}$ be families of $L^2(\mathbb R)$-valued random variables such that
\begin{equation}
\limsup_{N \to \infty}\left( \mathbb E \left[\sum_{i=1}^N \int_0^t a_i(s)^2ds \right]\mathbb E \left[\sum_{i=1}^N \int_0^t b_i(s)^2ds \right]\right) < \infty
\end{equation}
for all $t$. Let $\varphi$ be as in Theorem \ref{thm:Xn} and {$\bar \varphi_i$ as in} \eqref{eq:barphi}.Then 
\begin{equation}
\left |\sum_{i=1}^N \int_0^\infty \bar \varphi_i a_i b_i dt \right | \le  \tilde B_N\norminf[\psi],
\end{equation}
where $\tilde B_N$ is a $\mathbb R_+$-valued random variable independent of $\psi$ such that
\begin{equation}
\limsup_{N \to \infty}\E[ \tilde B_N]  < \infty.
\end{equation}
\end{lem}
\begin{proof}
\begin{equation}
|\bar \varphi_i| = \left|N \int_0^1 \varphi(t,x)1_{N,i}(x) dx \right| \le \| \varphi \|_{L^\infty} \le c_\phi \|\psi\|_{L^\infty},
\end{equation}
where $c_\phi = \|\phi\|_{L^\infty}$ depends on $\phi$ only. Finally, by the Cauchy-Schwarz inequality and since $\varphi(\cdot, x)$ is supported in $[0,T]$ we have
\begin{equation}
\left |\sum_{i=1}^N \int_0^\infty \bar \varphi_i a_i b_i dt\right | \le c_\phi \|\psi\|_{L^\infty} {  \left(\sum_{i=1}^N \int_0^T a_i^2dt\right)^{1/2} \left(  \sum_{i=1}^N \int_0^T b_i^2dt \right)^{1/2}}.
\end{equation}
\end{proof}
\begin{lem} \label{lem:typeY}
Let $(a_i)_{i \in \mathbb N}$ be a family of $L^2(\mathbb R)$-valued random variables such that
\begin{equation}
\limN\mathbb E \left[ \frac{1}{N}\sum_{i=1}^N \int_0^t a_i(s)^2ds \right] =0
\end{equation}
for all $t$.Then, for $\bar \varphi_i$ as in Lemma \ref{lem:typeZ}, we have
\begin{equation}
\sum_{i=1}^{N-1} \int_0^\infty a_i \left(\bar \varphi_{i+1}-\bar \varphi_i \right) dt = Y_N(\varphi) + Z_N(\varphi),
\end{equation}
with
\begin{equation}
|  Y_N(\varphi)|  \le  A_N\normH[\psi] , \quad |  Z_N(\varphi)| \le A_N \|\psi\|_{L^\infty}  ,
\end{equation}
where $A_N$ is a $\mathbb R_+$-valued random variable  independent of $\psi$ such that
\begin{equation}
\lim_{N \to \infty}\E[ A_N] = 0.
\end{equation}
\end{lem}
\begin{proof}
By Cauchy-Schwarz we have
\begin{equation} \label{eq:typeYCS}
\left|\sum_{i=1}^{N-1} \int_0^\infty  a_i\left( \bar \varphi_{i+1}-\bar \varphi_i\right) dt \right| \le {\left(\sum_{i=1}^{N-1}\int_0^\infty \left( \bar \varphi_{i+1}-\bar \varphi_i\right)^2 dt \right)^{1/2} {\left( \sum_{i=1}^{N-1}\int_0^T a_i^2 dt\right)^{1/2}}}.
\end{equation}
We write
\begin{equation}
\begin{gathered}
\bar\varphi_{i+1} -\bar \varphi_i =N \int_0^1 1_{N,i+1}(x) \varphi(t,x)dx- N\int_0^1 1_{N,i}(x) \varphi(t,x)dx
\\
= N\int_0^1 1_{N,i}(x) \left( \varphi \left(t, x+\frac{1}{N}\right)-\varphi(t,x)\right)dx
\\
=  {N\int_0^1 1_{N,i}(x) \int_x^{x+\frac{1}{N}}\partial_x \varphi(t,y)dydx},
\end{gathered}
\end{equation}
 {Thus, Cauchy-Schwarz inequality implies}
\begin{equation}
 {\left( \bar \varphi_{i+1}-\bar \varphi_i\right)^2 \le \frac{1}{N} \int_0^1\int_0^1 1_{N,i}(x) |\partial_x\varphi(t,y)|^2 dydx =\frac{1}{N^2} \int_0^1 |\partial_x\varphi(t,y)|^2 dy,}
\end{equation}
 {and so}
\begin{equation} \label{eq:gradphi}
 {\int_0^\infty \sum_{i=1}^{N-1}\left( \bar \varphi_{i+1}-\bar \varphi_i\right)^2 dt \le \frac{1}{N}  \int_0^\infty \int_0^1 \partial_x\varphi(t,x)^2 dxdt = \frac{1}{N} \normLdue[\partial_x\varphi]^2.}
\end{equation}
The conclusion finally follows from \eqref{eq:typeYCS} and
\begin{equation}
\normLdue[\partial_x\varphi]^2 \le 2 \|\phi\|_{L^\infty}^2 \|\partial_x\psi\|_{L^2}^2 + 2 \| \partial_x\phi\|_{L^2} \|\psi \|_{L^\infty}^2 \le C_\phi(\|\psi\|_{H^1}^2+\|\psi \|_{L^\infty}^2),
\end{equation}
where $C_\phi = 2 \max\{\norminf[\phi]^2,\normLdue[\partial_x\phi]^2\}$ depends on $\phi$ only.
\end{proof}
\begin{oss}
We will diffusely use summation by parts formulae like
\begin{equation}
\sum_{i=l+1}^{N-l} \bar \varphi_i \nabla^* \hat p_{l,i} = \sum_{i=l+1}^{N-l} \hat p_{l,i} \nabla \bar \varphi_i +\bar \varphi_{l+1} \hat p_{l,l} - \bar\varphi_{N-l+1}\hat p_{l,N-l}
\end{equation}
However, since $\varphi(t,\cdot)$ is compactly supported in $(0,1)$ and $l/N \to 0$, the boundary terms $\bar \varphi_{l+1}$ and $\bar \varphi_{N-l+1}$  will be identically zero for $N$ large enough. With this in mind, and since we will eventually take the limit $N\to \infty$, we shall simply write
\begin{equation}
\sum_{i=l+1}^{N-l} \bar \varphi_i \nabla^* \hat p_{l,i} = \sum_{i=l+1}^{N-l} \hat p_{l,i} \nabla \bar \varphi_i.
\end{equation}
\end{oss}
\begin{lem} \label{lem:antisym}
Let $\varphi$ be as in Theorem \ref{thm:Xn} and $X_{a,N}$ as in \eqref{eq:XaN}. Then there exist $\mathbb R_+$-valued random variables  $A_{a,N}$, $B_{a,N}$ independent of $\psi$ such that  $X_{a,N}= Y_{a,N}+ Z_{a,N}$, where
\begin{equation} \label{eq:YaNZaN}
|Y_{a,N}(\varphi,\eta)| \le \|\psi\|_{H^1} A_{a,N}  , \quad |Z_{a,N}(\varphi,\eta)| \le \|\psi\|_{L^\infty} B_{a,N}  
\end{equation}
and
\begin{equation}
\lim_{N \to \infty} \mathbb E[A_{a,N}  ] = \lim_{N \to \infty} \mathbb E[B_{a,N}  ] = 0.
\end{equation}
\end{lem}
\begin{proof}
\begin{equation}
\begin{gathered}
X_{a,N}(\varphi,\eta) = -\int_0^\infty \sum_{i=l+1}^{N-l} \bar \varphi_i \left(  \partial_r\eta(\uhat_{l,i}) +  \partial_pq(\uhat_{l,i}) \right) \nabla^* \hat p_{l,i} dt +
\\
+ \int_0^\infty \sum_{i=l+1}^{N-l} \bar \varphi_i \left(  \partial_p\eta(\uhat_{l,i}) \nabla \hat V'_{l,i} +  \partial_rq(\uhat_{l,i}) \nabla \hat r_{l,i} \right)  dt.
\end{gathered}
\end{equation}
We use the equations which define the entropy-entropy flux $(\eta, q)$, namely
\begin{equation}
\begin{cases}
 \partial_r\eta+\partial_pq=0
\\
\tau'(r)\partial_p\eta +   \partial_rq =0
\end{cases}
\end{equation}
to  obtain
\[
X_{a,N}(\varphi,\eta) = \sum_{i=l+1}^{N-l} \int_0^\infty  \bar \varphi_i  \partial_p\eta(\uhat_{l,i}) \left(  \nabla \hat V'_{l,i}  - \tau'(\hat r_{l,i})\nabla \hat r_{l,i} \right) dt.
\]
\begin{equation} \label{eq:XaN1}
= \sum_{i=l+1}^{N-l} \int_0^\infty \bar \varphi_i \partial_p\eta(\uhat_{l,i}) \nabla ( \hat V'_{l,i}- \tau(\hat r_{l,i}))dt+
\end{equation}
\begin{equation} \label{eq:XaN2}
+ \sum_{i=l+1}^{N-l}\int_0^\infty \bar \varphi_i \partial_p\eta(\uhat_{l,i}) (\nabla \tau(\hat r_{l,i})- \tau'(\hat r_{l,i}) \nabla \hat r_{l,i})dt.
\end{equation}
After a summation by parts, \eqref{eq:XaN1} gives
\begin{equation} \label{eq:XaN3}
\sum_{i=l+1}^{N-l} \int_0^\infty \bar \varphi_i \partial_p\eta(\uhat_{l,i}) \nabla ( \hat V'_{l,i}- \tau(\hat r_{l,i}))dt = Q_{a,N}(\varphi,\eta) + Z_{a,N}(\varphi,\eta),
\end{equation}
where
\begin{equation}
Q_{a,N}(\varphi,\eta) = \int_0^\infty \sum_{i=l+1}^{N-l} (\nabla^* \bar \varphi_i) {\partial_p\eta(\uhat_{l,i-1})}(\hat V'_{l,i}-\tau(\hat r_{l,i})) dt
\end{equation}
and
\begin{equation}
{Z_{a,N}(\varphi,\eta) = \int_0^\infty \sum_{i=l+1}^{N-l}  \bar \varphi_i (\nabla^* \partial_p\eta(\uhat_{l,i}))(\hat V'_{l,i}-\tau(\hat r_{l,i})) dt.}
\end{equation}
 {$\partial_p\eta$ is bounded;} moreover, Theorem \ref{thm:1blockexp} implies
\begin{equation}
{\lim_{N \to \infty}\mathbb E\left[\frac{1}{N} \int_0^T \sum_{i=l+1}^{N-l} \int (\hat V'_{l,i}-\tau(\hat r_{l,i}))^2  dt \right]= 0,}
\end{equation}
for any $T >0$.  {Therefore we can apply Lemma} \ref{lem:typeY}  {to $Q_{a,N}$ and obtain}
\begin{equation}
{Q_{a,N}(\varphi,\eta) = Y_{a,N}(\varphi,\eta) + Z_{Qa,N}(\varphi,\eta),}
\end{equation}
where
\begin{equation}
|Y_{a,N}(\varphi,\eta)| \le \|\psi\|_{H^1} A_{a,N}  , \quad |Z_{Qa,N}(\varphi,\eta)| \le \|\psi\|_{L^\infty} A_{a,N}  ,
\end{equation}
for some  $\mathbb R_+$-valued random variable $A_{a,N}$ independent of $\psi$ and such that
\begin{equation}
\lim_{N \to \infty} \mathbb E [A_{a,N}  ] = 0.
\end{equation}
We can apply Lemma \ref{lem:typeZ} to $Z_{a1,N}$. In fact, since the second derivatives of $\eta$ are bounded, we have
\begin{equation}
(\nabla^* \partial_p\eta(\uhat_{l,i}))^2 \le C((\nabla^* \hat r_{l,i})^2 + (\nabla^* \hat p_{l,i})^2),
\end{equation}
for some $C>0$. Furthermore, Theorems \ref{thm:1blockexp} and \ref{thm:2block} imply
\begin{equation}
\begin{gathered}
{\E\left[ \int_0^T \sum_{i=l+1}^{N-l}  (\hat r_{l,i}-\hat r_{l,i-1})^2  dt\right]} \E\left[ \int_0^T \sum_{i=l+1}^{N-l}  (\hat V'_{l,i}-\tau(\hat r_{l,i}))^2  dt\right]+
\\
+\E\left[ \int_0^T \sum_{i=l+1}^{N-l}  (\hat p_{l,i}-\hat
  p_{l,i-1})^2 dt\right]\E\left[ \int_0^T \sum_{i=l+1}^{N-l}  (\hat V'_{l,i}-\tau(\hat r_{l,i}))^2  dt\right]
\\
  \le  2C_1(T)C_2(T) \left( \frac{N}{l^2}+\frac{l}{\sigma}  \right)^2,
\end{gathered}
\end{equation}
which vanishes as $N \to \infty$ for any $T >0$. Therefore,
\begin{equation}
|Z_{a1,N}(\varphi,\eta)| \le \|\psi\|_{L^\infty} B_{a1,N}  ,
\end{equation}
where $B_{a1,N}$ is a functional independent of $\psi$ such that
\begin{equation}
\lim_{N \to \infty} \mathbb E [ B_{a1,N}  ] = 0.
\end{equation}
Finally, Lemma \ref{lem:typeZ} applies to \eqref{eq:XaN2}, too. We set
\begin{equation}
Z_{a2,N}(\varphi,\eta) = \sum_{i=l+1}^{N-l}\int_0^\infty \bar \varphi_i \partial_p\eta(\uhat_{l,i}) (\nabla \tau(\hat r_{l,i})- \tau'(\hat r_{l,i}) \nabla \hat r_{l,i})dt
\end{equation}
and write
\begin{equation}
\nabla \tau(\hat r_{l,i}) - \tau'(\hat r_{l,i})\nabla \hat r_{l,i} =
(\tau' (\tilde r_{l,i})-\tau'(\hat r_{l,i}))\nabla \hat r_{l,i} =
\tau''(\tilde{\tilde r}_{l,i})(\tilde r_{l,i}-\hat r_{l,i})\nabla \hat
r_{l,i},
\end{equation}
where $\tilde r_{l,i}$ is between $\hat r_{l,i+1}$ and $\hat r_{l,i}$, while $\tilde {\tilde  r}_{l,i}$ is between $\hat r_{l,i}$ and $\tilde r_{l,i}$. With this in mind and using the fact (proven in  Appendix \ref{app:tension}) that $\tau''$ is bounded, we obtain
\begin{equation}
\left | \nabla \tau(\hat r_{l,i}) - \tau'(\hat r_{l,i})\nabla \hat r_{l,i}  \right | \le \|\tau''\|_{L^\infty} |\nabla \hat r_{l,i}|^2.
\end{equation}
Finally, since, for any $T>0$,
\begin{equation}
{\E \left[\sum_{i=l+1}^{N-l} \int_0^T  (\hat r_{l,i+1}-\hat r_{l,i})^2  dt\right] \le C_2(T) \left( \frac{N}{l^3} + \frac{1}{\sigma}\right) \to 0}
\end{equation}
as $N \to \infty$, we obtain
\begin{equation}
|Z_{a2,N}(\varphi,\eta)| \le  \| \psi\|_{L^\infty} B_{a2,N}  ,
\end{equation}
where $B_{a2,N}$ is independent of $\psi$ and
\begin{equation}
\lim_{N \to \infty} \mathbb E [ B_{a2,N}  ] = 0.
\end{equation}
Putting everything together, we have obtained 
\begin{equation}
X_{a,N} = Y_{a,N} + Z_{a,N},
\end{equation}
where
\begin{equation}
Z_{a,N} = Z_{Qa,N} + Z_{a1,N}+ Z_{a2,N},
\end{equation}
and $Y_{a,N}$ and $Z_{a,N}$ have the claimed properties.
\end{proof}
\begin{lem} \label{lem:symmetric}
Let $\varphi$ be as in Theorem \ref{thm:Xn} and let $\tilde X_{s,N}$ be as in \eqref{eq:tilXsN}. Then there exist  $\mathbb R_+$-valued random variables $\tilde A_{s,N}$, $\tilde B_{s,N}$, $\tilde B_{s,N}^*$, independent of $\psi$ such that  $\tilde X_{s,N} = \tilde Y_{s,N}+ \tilde Z_{s,N}+ \tilde Z_{s,N}^*$, where
\begin{equation}
\begin{gathered}
|\tilde Y_{s,N}(\varphi,\eta)| \le \|\psi\|_{H^1} \tilde A_{s,N}  , \quad |\tilde Z_{s,N}(\varphi,\eta)| \le \|\psi\|_{L^\infty} \tilde B_{s,N}  ,
\\
  |\tilde Z^*_{s,N}(\varphi,\eta)| \le \|\psi\|_{L^\infty} \tilde B^*_{s,N}  
\end{gathered}
\end{equation}
and
\begin{equation}
\lim_{N \to \infty} \mathbb E [ \tilde A_{s,N}  ] =\lim_{N \to \infty} \mathbb E [ \tilde B_{s,N}  ] = 0, \quad \limsup_{N \to \infty} \mathbb E[ \tilde B_{s,N}^*  ] < \infty.
\end{equation}
\end{lem}
\begin{proof}
We look at the term involving $V'$, first. We write
\begin{equation} \label{eq:XsN1}
\begin{gathered}
\sigma \int_0^\infty  \sum_{i=l+1}^{N-l}
\bar \varphi_i   \partial_r\eta(\uhat_{l,i}) \Delta \hat V'_{l,i}dt =- \sigma \int_0^\infty  \sum_{i=l+1}^{N-l}
\bar \varphi_i   \partial_r\eta(\uhat_{l,i}) \nabla^* \nabla \hat V'_{l,i}dt
\\
= \tilde Q_{s,N}(\varphi,\eta) + \tilde Z^*_{s,N}(\varphi,\eta),
\end{gathered}
\end{equation}
where
\begin{equation} \label{eq:tilQsN}
 \tilde Q_{s,N}(\varphi,\eta)= - \sigma \int_0^\infty  \sum_{i=l+1}^{N-l}
(\nabla^*\bar \varphi_i)   \partial_r\eta(\uhat_{l,i}) \nabla \hat
V'_{l,i}dt  
\end{equation}
and
\begin{equation} \label{eq:Z*sN}
\tilde Z^*_{s,N}(\varphi,\eta)= - \sigma \int_0^\infty  \sum_{i=l+1}^{N-l}
 {\bar \varphi_{i-1}} (\nabla^*   \partial_r\eta(\uhat_{l,i})) \nabla \hat
V'_{l,i}dt.
\end{equation}
Since $\partial_r \eta$ is bounded and
\begin{equation}
\limsup_{N \to \infty}\left( \frac{\sigma^2}{N} \E \left[\int_0^T \sum_{i=l+1}^{N-l}  (\hat V'_{l,i+1}-\hat V'_{l,i})^2  dt\right]\right) \le C_2(T) \lim_{N\to \infty} \left(\frac{\sigma^2}{l^3}+ \frac{\sigma}{N} \right) = 0
\end{equation}
for any $T >0$, Lemma \ref{lem:typeY} applies to $\tilde Q_{s,N}$, yielding
\begin{equation}
 \tilde Q_{s,N}(\varphi, \eta) = \tilde Y_{s,N}(\varphi, \eta) + \tilde Z_{Qs,N}(\varphi, \eta),
\end{equation}
where
\begin{equation}
|\tilde Y_{s,N}(\varphi, \eta) | \le \|\psi\|_{H^1} \tilde A_{s,N}  , \quad  |\tilde Z_{sQ,N}(\varphi, \eta) | \le \|\psi\|_{L^\infty} \tilde A_{s,N}  
\end{equation}
with
\begin{equation}
\lim_{N \to \infty} \mathbb{E}[\tilde A_{s,N}  ]=0.
\end{equation}
From
\begin{equation}
(\nabla^* \partial_r\eta(\uhat_{l,i}))^2 \le C ((\nabla^* \hat r_{l,i})^2+(\nabla^* \hat p_{l,i})^2)
\end{equation}
and
\begin{equation}
\begin{gathered}
 \E\left[ \sigma \int_0^T \sum_{i=l+1}^{N-l}  (\hat r_{l,i}-\hat r_{l,i-1})^2  dt \right]\E\left[ \sigma \int_0^T \sum_{i=l+1}^{N-l}  (\hat V'_{l,i+1}-\hat V'_{l,i})^2  dt \right]+
\\
+ \E\left[ \sigma \int_0^T \sum_{i=l+1}^{N-l} (\hat p_{l,i}-\hat p_{l,i-1})^2 d\mu_t^N dt \right]\E\left[ \sigma \int_0^T \sum_{i=l+1}^{N-l} (\hat V'_{l,i+1}-\hat V'_{l,i})^2  dt \right]
\\
 \le 2C_1(T)C_2(T) \left(\frac{N  \sigma}{l^3}+1 \right)^2,
\end{gathered}
\end{equation}
which stays bounded as $N \to\infty$ for any $T>0$, Lemma \ref{lem:typeZ} applies to $\tilde Z^*_{s,N}$. Therefore, we have
\begin{equation}
|\tilde Z^*_{s,N}(\varphi,\eta)| \le \|\psi\|_{L^\infty} \tilde B^*_{s,N}  ,
\end{equation}
for some  $\mathbb R_+$-valued random variable $\tilde B^*_{s,N}$ independent of $\psi$ and such that
\begin{equation}
\limsup_{N \to \infty} \mathbb E[\tilde B^*_{s,N}  ] <\infty.
\end{equation}
We estimate the quadratic variations, namely
\begin{equation} \label{eq:XsN2}
\tilde Z_{s1,N}(\varphi,\eta) = \frac{2\sigma}{l^3} \sum_{i=l+1}^{N-l}\int_0^\infty \bar \varphi_i \partial^2_{rr}\eta(\uhat_{l,i}) dt.
\end{equation}
Therefore,
\begin{equation}
 | \tilde Z_{s1,N}(\varphi,\eta) | \le \norminf[\partial^2_{rr}\eta] \frac{\sigma}{l^3}\int_0^\infty \sum_{i=l+1}^{N-l} |\bar \varphi_i| dt \le C_{\eta,\phi}  \frac{N \sigma}{l^3}\|\psi\|_{L^\infty}.
\end{equation}
Since $N\sigma/l^3 \to 0$, as $N \to \infty$, the proof is completed if we set
\begin{equation}
\tilde Z_{s,N} = \tilde Z_{Qs,N} + \tilde Z_{s1,N}.
\end{equation}
\end{proof}
Similarly, we prove the following.
\begin{lem} \label{lem:symmetric2}
Let $\varphi$ be as in Theorem \ref{thm:Xn} and let $ X_{s,N}$ be as in \eqref{eq:XsN}. Then there exist $\R_+$-valued random variables   $ A_{s,N}$, $ B_{s,N}$, $ B_{s,N}^*$ independent of $\psi$ such that $ X_{s,N}$ decomposes as $ X_{s,N} =  Y_{s,N}+  Z_{s,N}+  Z_{s,N}^*$, where
\begin{equation}
\begin{gathered}
| Y_{s,N}(\varphi,\eta)| \le \|\psi\|_{H^1}  A_{s,N}  , \quad | Z_{s,N}(\varphi,\eta)| \le \|\psi\|_{L^\infty}  B_{s,N}  ,
\\
  | Z^*_{s,N}(\varphi,\eta)| \le \|\psi\|_{L^\infty}  B^*_{s,N}  
\end{gathered}
\end{equation}
and
\begin{equation}
\lim_{N \to \infty} \mathbb E [ A_{s,N}  ] =\lim_{N \to \infty} \mathbb E [  B_{s,N}  ] = 0, \quad \limsup_{N \to \infty} \mathbb E[  B_{s,N}^*  ] < \infty.
\end{equation}
\end{lem}
\begin{lem} \label{lem:martin}
Let $\varphi$ be as in Theorem \ref{thm:Xn} and let $
\mathcal M_N$ be as in \eqref{eq:MN}. Then there exist $ A_{\mathcal M,N}, B_{\mathcal M,N} \in \mathbb R_+$
independent of $\psi$ such that $\mathcal M_N = Y_{\mathcal M,N}+ Z_{\mathcal M,N}$, where
\begin{equation}
\mathbb E \left[| Y_{\mathcal M,N}(\varphi,\eta)|\right] \le \| \varphi\|_{H^1}
A_{\mathcal M,N}  , \quad \mathbb E\left[| Z_{\mathcal M,N}(\varphi,\eta)|\right] \le \| \varphi\|_{L^\infty}
B_{\mathcal M,N}  
\end{equation}
and
\begin{equation}
\lim_{N \to \infty}  A_{\mathcal M,N}   =  \lim_{N \to
\infty} B_{\mathcal M,N}   =0.
\end{equation}
\end{lem}
\begin{proof}
Recall $\varphi = \phi \psi$ and set  {$\bar \phi_i(t) =N \int_0^1 \phi(t,x)1_{N,i}(x)dx$, $\bar \psi_i(t) =N \int_0^1 \phi(t,x)1_{N,i}(x)dx$}. Summing by parts  {and thanks to the fact that $ \varphi_i =\phi_i\psi_i+\|\psi\|_{L^\infty} \mathcal O(1/N)$}, we obtain \\ ${\mathcal M_N = Y_{\mathcal M,n}+ Z_{\mathcal M 1,N}
+Z_{\mathcal M 2,N}}$, where
\begin{equation} \label{eq:YMN}
 {Y_{\mathcal M,N}(\varphi,\eta) =-\sqrt{2\frac{\sigma}{N}} \int_0^\infty
\sum_{i=l+1}^{N-l} \bar \phi_i(\nabla \bar \psi_i)  \partial_p\eta(\uhat_{l,i+1}) 
d\hat  w_{l,i}},
\end{equation} 
\begin{equation}\label{eq:Z1MN}
 {Z_{\mathcal M 1,N}(\varphi,\eta)  =-\sqrt{2\frac{\sigma}{N}} \int_0^\infty
\sum_{i=l+1}^{N-l} \bar \psi_{i+1}(\nabla \bar \phi_i)  \partial_p\eta(\uhat_{l,i+1}) 
d\hat  w_{l,i}},
\end{equation}
\begin{equation} \label{eq:Z2MN}
 {Z_{\mathcal M 2,N}(\varphi,\eta) =-\sqrt{2\frac{\sigma}{N}} \int_0^\infty
\sum_{i=l+1}^{N-l}  \bar \varphi_i  (\nabla\partial_p\eta(\uhat_{l,i}) )
d\hat  w_{l,i}}.
\end{equation}
We write
\begin{equation}
\begin{gathered}
| Y_{\mathcal M,N}(\varphi,\eta)| = \sqrt{2\frac{\sigma}{N}} \sqrt{ \left(\int_0^\infty
\sum_{i=l+1}^{N-l} \bar \phi_i(\nabla\bar \psi_i)  \partial_p\eta(\uhat_{l,i+1}) 
d\hat  w_{l,i}
  \right)^2}
\\
\le   \sqrt{2\frac{\sigma}{N}} \sqrt{ N\frac{1}{l^2} \sum_{i=l+1}^{N-l}
   \left( \sum_{|j|<l}\int_0^\infty\bar \phi_i(\nabla \bar \psi_i) \partial_p
    \eta(\uhat_{l,i+1}) \frac{l-|j|}{l} dw_{i-j}\right)^2}.
    \end{gathered}
\end{equation}
Now we write
\[
\left( \sum_{|j|<l}\int_0^\infty\bar \phi_i(\nabla \bar \psi_i) \partial_p
    \eta(\uhat_{l,i+1}) \frac{l-|j|}{l} dw_{i-j}\right)^2
\]
\begin{equation}
=  \sum_{|j|<l} \left(\int_0^\infty\bar \phi_i(\nabla \bar \psi_i) \partial_p
    \eta(\uhat_{l,i+1}) \frac{l-|j|}{l} dw_{i-j}\right)^2 +
\end{equation}
\[
+\sum_{k \neq j} \left(\int_0^\infty\bar \phi_i(\nabla \bar \psi_i) \partial_p
    \eta(\uhat_{l,i+1}) \frac{l-|k|}{l} dw_{i-k}\right) \left(\int_0^\infty\bar \phi_i(\nabla \bar \psi_i) \partial_p
    \eta(\uhat_{l,i+1}) \frac{l-|j|}{l} dw_{i-j}\right).
\]
This, together with It\^{o} isometry implies
implies
\begin{equation}
\begin{gathered}
\mathbb E \left[|Y_{\mathcal M,N}(\varphi,\eta)| \right]\le  
 \sqrt{\frac{2\sigma}{l^2}
\sum_{i=l+1}^{N-l} \sum_{|j|<l} \mathbb E\left[ \int_0^\infty \bar \phi_i^2 (\nabla
\bar \psi_i)^2 (\partial_p \eta(\uhat_{l,i+1}))^2
\left(\frac{l-|j|}{l}\right)^2 dt\right]}.
\\
\le  C_{\eta,\phi} \sqrt{
  \frac{\sigma}{Nl} \frac{1}{N}
  \sum_{i=l+1}^{N-l} \int_0^\infty N^2(\nabla \bar \psi_i)^2dt }     \le C_{\eta,\phi} \sqrt
\frac{\sigma}{Nl} \| \psi\|_{H^1},
\end{gathered}
\end{equation}
where $C_{\eta,\phi}$ is independent of $\psi$  and the coefficient of $\|\psi\|_{H^1}$ vanishes as $N \to\infty$. 

Similarly, we obtain
\begin{equation}
\mathbb E \left[|Z_{\mathcal M 1,N}(\varphi,\eta)|\right] \le C'_{\eta,\phi}  \sqrt \frac{\sigma}{Nl} \|\psi\|_{L^\infty}.
\end{equation}
Finally, recalling that $\varphi(\cdot, x)$ is supported in $[0,T]$, we estimate
\begin{equation}
\begin{gathered}
{\mathbb E \left[|Z_{\mathcal M 2,N}(\varphi,\eta)|\right] \le \|\psi\|_{L^\infty}  C''_{\eta,\phi} \sqrt{ \frac{\sigma}{l}
  \mathbb E \left[\sum_{i=l+1}^{N-l} \int_0^T (\nabla \hat p_{l,i})^2+(\nabla \hat r_{l,i})^2 dt \right]} }
\\
\le C''_{\eta,\phi} \sqrt{C_2(T)} \left( \frac{N \sigma}{l^4}+\frac{1}{l} \right)^{1/2}
\end{gathered}
\end{equation}
Since the last term at the right hand side vanishes as $N \to \infty$, the lemma is proven.
\end{proof}
Similarly, we prove the following.
\begin{lem} \label{lem:martintil}
Let $\varphi$ be as in Theorem \ref{thm:Xn} and let $
\tilde {\mathcal M}_N$ be as in \eqref{eq:tildeMN}. Then there exists $ \tilde A_{\mathcal M,N}, \tilde B_{\mathcal M,N} \in \mathbb R_+$
independent of $\psi$ such that $\tilde{\mathcal M}_N =  \tilde Y_{\mathcal M,N}+ \tilde Z_{\mathcal M,N}$, where
\begin{equation}
\mathbb E\left[| \tilde Y_{\mathcal M,N}(\varphi,\eta)|\right] \le \| \varphi\|_{H^1}
\tilde A_{\mathcal M,N}  , \quad \mathbb E\left[|\tilde Z_{\mathcal M,N}(\varphi,\eta)|\right] \le \| \varphi\|_{L^\infty}
\tilde B_{\mathcal M,N}  
\end{equation}
and
\begin{equation}
\lim_{N \to \infty}  \tilde A_{\mathcal M,N}   =  \lim_{N \to
\infty} \tilde B_{\mathcal M,N}   =0.
\end{equation}
\end{lem}
\begin{lem} \label{lem:Derivative}
Let $\varphi$ be as in Theorem \ref{thm:Xn} and let $ \mathcal N_N$ be as in \eqref{eq:NN}. Then there exists $\mathbb R_+$-valued random variables  $ A_{n,N}$ and $ B_{n,N}$, independent of $\psi$ such that  $\mathcal N_N =  Y_{n,N}+  Z_{n,N}$, where
\begin{equation}
| Y_{n,N}(\varphi,\eta)| \le \|\psi\|_{H^1}  A_{n,N}  , \quad | Z_{n,N}(\varphi,\eta)| \le \|\psi\|_{L^\infty}  B_{n,N}  ,
\end{equation}
and
\begin{equation}
\lim_{N \to \infty} \mathbb E [ A_{n,N}  ] =\lim_{N \to \infty} \mathbb E [  B_{n,N}  ] = 0.
\end{equation}
\end{lem}
\begin{proof}
As in the proof of Lemma \ref{lem:typeY}, we prove the statement for $\varphi$ smooth and compactly supported, and the general statement for $\varphi \in H^1\cap L^\infty$ will follow by approximating $\varphi$ with smooth and compactly supported functions.
\begin{equation}
\begin{gathered}
-\int_0^\infty \int_0^1 \partial_x\varphi(t,x) q(\uhat_N(t,x))dx dt 
\\
= - \int_0^\infty \sum_{i=l+1}^{N-l}\left( \int_0^1\partial_x \varphi(t,x) 1_{N,i}(x) dx\right) q(\uhat_{l,i})dt
\\
=-\int_0^\infty \sum_{i=l+1}^{N-l} \left( \varphi \left(t, \frac{i}{N}+\frac{1}{2N}\right)-\varphi \left(t, \frac{i}{N}-\frac{1}{2N}\right)\right) q(\uhat_{l,i})dt
\end{gathered}
\end{equation}
\begin{equation}
\begin{gathered}
= \int_0^\infty \sum_{i=l+1}^{N-l} \varphi \left(t, \frac{i}{N}-\frac{1}{2N}\right) \nabla q(\uhat_{l,i})dt 
\\
{= \int_0^\infty \sum_{i=l+1}^{N-l} \tilde \varphi_i \nabla q(\uhat_{l,i})dt}
\\
= \int_0^\infty \sum_{i=l+1}^{N-l} \tilde \varphi_i (\partial_rq(\utilde_{l,i}) \nabla \hat r_{l,i}+ \partial_pq(\utilde_{l,i}) \nabla \hat p_{l,i}),
\end{gathered}
\end{equation}
for some $\utilde_{l,i}$ {on the segment joining} $\uhat_{l,i}$ and $\uhat_{l,i+1}$ and where
\begin{equation}
\tilde \varphi_i(t) :=  \varphi \left(t, \frac{i}{N}-\frac{1}{2N}\right).
\end{equation}
Thus,
\begin{equation}
\begin{gathered}
 \mathcal N_N(\varphi,\eta) = \int_0^\infty \sum_{i=l+1}^{N-l} \tilde \varphi_i (\partial_rq(\utilde_{l,i}) \nabla \hat r_{l,i}+ \partial_pq(\utilde_{l,i}) \nabla \hat p_{l,i})dt +
\\
-\int_0^\infty \sum_{i=l+1}^{N-l} \bar \varphi_i \left(  \partial_rq(\uhat_{l,i})\nabla \hat r_{l,i}-\partial_pq(\uhat_{l,i}) \nabla^* \hat p_{l,i} \right) dt
\\
= \mathcal N_{r,N}(\varphi,\eta) + \mathcal N_{p,N}(\varphi,\eta) + \mathcal N_{1,N}(\varphi,\eta),
\end{gathered}
\end{equation}
where
\begin{equation} \label{eq:Nr}
\mathcal N_{r,N}(\varphi,\eta)= \int_0^\infty \sum_{i=l+1}^{N-l} (\tilde \varphi_i \partial_rq(\utilde_{l,i})- \bar \varphi_i \partial_rq(\uhat_{l,i}))\nabla \hat r_{l,i}dt,
\end{equation}
\begin{equation}\label{eq:Np}
\mathcal N_{p,N}(\varphi,\eta)=  \int_0^\infty \sum_{i=l+1}^{N-l} (\tilde \varphi_i \partial_pq(\utilde_{l,i}) - \bar \varphi_i \partial_pq(\uhat_{l,i}))\nabla \hat p_{l,i}dt,
\end{equation}
\begin{equation} \label{eq:N1}
\mathcal N_{1,N}(\varphi,\eta)= \int_0^\infty \sum_{i=l+1}^{N-l} \bar \varphi_i \partial_pq(\uhat_{l,i})
(\nabla\hat p_{l,i}+ \nabla^*\hat p_{l,i})dt.
\end{equation}
Since
\begin{equation}
\nabla \hat p_{l,i} + \nabla^* \hat p_{l,i} = \hat p_{l,i+1} + \hat
p_{l,i-1} -2 \hat p_{l,i} = \Delta \hat p_{l,i},
\end{equation}
the term $\mathcal N_{1,N}$ can be estimated exactly in the same way we estimated
\begin{equation}
\sigma \int_0^\infty \sum_{i=l+1}^{N-l} \bar \varphi_i \partial_p\eta(\hat
u_{l,i}) \Delta \hat V'_{l,i}dt
\end{equation}
in Lemma \ref{lem:symmetric}. The main difference here is that in
$\mathcal N_{1,N}$ does not have a factor $\sigma$.
Therefore we obtain
\begin{equation}
\mathcal N_{1,N} (\varphi,\eta) = Y_{n1,N}(\varphi,\eta) + Z_{n1,N}(\varphi,\eta),
\end{equation}
where
\begin{equation}
|Y_{n1,N}(\varphi,\eta)| \le \|\psi\|_{H^1} A_{n1,N}  , \quad |Z_{n1,N}(\varphi,\eta)| \le \|\psi\|_{H^1} B_{n1,N}  ,
\end{equation}
for some  $\mathbb R_+$-valued random variables $A_{n1,N}$ and $B_{n1,N}$ independent of $\psi$ and such that
\begin{equation}
\lim_{N \to\infty} \mathbb{E}[A_{n1,N}  ]= \lim_{N \to\infty} \mathbb{E}[B_{n1,N}  ]=0.
\end{equation}
We are left with estimating $\mathcal N_{r,N}$ and $\mathcal N_{p,N}$. We only evaluate $\mathcal N_{r,N}$, as $\mathcal N_{p,N}$ is dealt with in a similar way. From \eqref{eq:Nr}, we evaluate
\begin{equation}
\tilde \varphi_i \partial_rq(\utilde_{l,i})-\bar \varphi_i \partial_rq(\uhat_{l,i}) = (\tilde \varphi_i - \bar \varphi_i) \partial_rq(\uhat_{l,i})+ \tilde \varphi_i(\partial_rq(\utilde_{l,i})-\partial_rq(\uhat_{l,i})),
\end{equation}
which implies $\mathcal N_{r,N} = \mathcal N_{r1,N} + \mathcal N_{r2,N}$, where
\begin{equation} \label{eq:Nr1}
  \mathcal N_{r1,N}(\varphi,\eta)= \int_0^\infty \sum_{i=l+1}^{N-l} (\tilde \varphi_i -\bar \varphi_i) \partial_rq(\uhat_{l,i}) \nabla \hat r_{l,i}dt
 \end{equation} 
 and
 \begin{equation} \label{eq:Nr2}
  \mathcal N_{r2,N}(\varphi,\eta)=  \int_0^\infty \sum_{i=l+1}^{N-l} \tilde  \varphi_i(\partial_rq(\utilde_{l,i})-\partial_rq(\uhat_{l,i})) \nabla \hat r_{l,i}dt.
 \end{equation} 
Since $|\tilde \varphi_i| \le C \|\psi\|_{L^\infty}$, performing estimates
identical to the ones done in the proof of Lemma \ref{lem:symmetric}, we can write $\mathcal N_{r2,N} = Y_{r2,N}+Z_{r2,N}$, where
\begin{equation}
|Y_{r2,N}(\varphi,\eta)| \le \|\psi\|_{H^1} A_{r2,N}  , \quad |Z_{r2,N}(\varphi,\eta)| \le \|\psi\|_{L^\infty} A_{r2,N}  ,
\end{equation}
for some  $\mathbb R_+$-valued random variable $A_{r2,N}$ independent of $\psi$ and such that
\begin{equation}
\lim_{N\to\infty}\mathbb E [A_{r2,N}   ]=0.
\end{equation}
 {In estimating $\mathcal N_{r1,N}$, given by} \eqref{eq:Nr1} {, we evaluate $\tilde \varphi_i-\bar \varphi_i$ in the same fashion as} \eqref{eq:gradphi} {, obtaining}
\begin{equation}
 {\int_0^\infty \sum_{i=l+1}^{N-l} (\tilde \varphi_i - \bar \varphi_i)^2 dt \le \frac{1}{N} \normLdue[\partial_x\varphi].}
\end{equation}
Moreover, since
\begin{equation}
\lim_{N\to\infty}\E \left[ \frac{1}{N} \int_0^T \sum_{i=l+1}^{N-l}  (\hat r_{l,i+1}-\hat r_{l,i})^2 dt\right] = 0,
\end{equation}
for any $T>0$ and since the first derivatives of $q$ are bounded, we can follow the proof of Lemma \ref{lem:typeY} and obtain $\mathcal N_{r1,N} = Y_{r1,N} + Z_{r1,N}$, where
\begin{equation}
|Y_{r1,N}(\varphi,\eta)| \le \|\psi\|_{H^1} A_{r1,N}  , \quad |Z_{r1,N}(\varphi,\eta)| \le \|\psi\|_{L^\infty} A_{r1,N}  ,
\end{equation}
for some  $\mathbb R_+$-valued random variable $A_{r1,N}$ independent of $\psi$ and such that
\begin{equation}
\lim_{N\to\infty}\mathbb E [A_{r1,N}   ]=0.
\end{equation}
The proof is concluded once we write $\mathcal N_N =  Y_{n,N} +  Z_{n,N}$, with $ Y_{n,N} =  Y_{n1,N} +  Y_{r1,N} + Y_{r2,N}$ and $Z_{n,N} =  Z_{n1,N} +  Z_{r1,N} + Z_{r2,N}$

\end{proof}
\subsection{One and Two-Block Estimates} \label{sub:estimates}
\subsubsection{The Relative Entropy and the Dirichlet Forms}
Denote by $\lambda^N$ the Gibbs measure
\begin{equation}
\lambda^N(d{\bf r}, d{\bf p}):= \lambda^N_{1,0,0}(d{\bf r}, d{\bf p}) 
= \prod_{i=1}^N \exp\left(-\left(\frac{p_i^2}{2}+V(r_i) \right)-G(1,0)\right) dr_i\frac{dp_i}{\sqrt{2\pi}}.
\end{equation}
Denote by $\mu_t^N$ the probability measure, on $\R^{2N}$, of the system a time $t$. 
The density $f_t^N$ of $\mu_t^N$ with respect to $\lambda^N$ solves the Fokker-Plank equation
\begin{equation}\label{eq:ftN}
\frac{\partial f_t^N}{\partial t} = \left(\mathcal G_N^{\bar \tau(t)} \right)^* f_t^N.
\end{equation}
Here $\left(\mathcal G_N^{\bar \tau(t)} \right)^* = -N L_N^{\bar \tau(t)} + N\bar\tau(t) p_N
 +\sigma N(S_N+\tilde S_N)$ 
is the adjoint of $\mathcal G_N^{\bar \tau(t)}$ with respect to $\lambda^N$.

Recall the definition of the relative entropy given by \eqref{eq:HN}, and 
and the Dirichlet forms \eqref{eq:DN}.

\begin{theorem}
Assume there is a constant $C_0$ independent of $N$ such that $H_N(0) \le C_0 N$. 
Assume also the external tension $\bar \tau : \R \to \R$ is bounded with bounded derivative.

There exists $C(t)$ independent of $N$ such that
\begin{equation}
H_N(f_t^N) + \sigma \int_0^t {\mathcal D_N}(f_s^N)+\widetilde {\mathcal D}_N(f_s^N)ds \le C(t) N.
\end{equation}
\end{theorem}
\begin{proof}
The statement will follow by a Gr\"onwall argument. We calculate
\begin{equation}
  \begin{split}
    \frac{d}{dt} H_N(f_t^N) = \int (\partial_t f_t^N) \log f_t^N d
    \lambda^N + \int \partial_t f_t^N d \lambda^N \\
    = \int (\partial_t f_t^N) \log f_t^N d \lambda^N + \int \partial_t f_t^N d \lambda^N.
  \end{split}
\end{equation}
By \eqref{eq:ftN}:
\begin{equation} \label{eq:HNDN0}
\begin{gathered}
\frac{d}{dt} H_N(f_t^N)  =  {\int f_t^N \mathcal G_N^{\bar \tau(t)} \log f_t^N d \lambda^N}
\\
= N\int f_t^N L_N^{\bar \tau(t)} \log f_t^N d \lambda^N + N \sigma \int f_t^N S_N \log f_t^N d \lambda^N + N \sigma \int f_t^N \tilde S_N \log f_t^N d \lambda^N.
\end{gathered}
\end{equation}
We have
\begin{equation}
\int f_t^N L_N^{\bar \tau(t)} \log f_t^N d \lambda^N = \int L_N^{\bar \tau(t)} f_t^N d \lambda^N = N \bar \tau(t) \int p_N f_t^N d\lambda^N.
\end{equation}
We estimate the second term in \eqref{eq:HNDN0} (the third term will be analogous).
\begin{equation}
\begin{gathered}
\int f_t^N S_N \log f_t^N d \lambda^N = - \sum_{i=1}^{N-1} \int f_t  D^*_i D_i \log f_t^N d \lambda^N
\\
= - \sum_{i=1}^{N-1} \int (D_i f_t^N) (D_i \log f_t^N) d \lambda^N = - \sum_{i=i}^{N-1} \int \frac{(D_i f_t^N)^2}{f_t^N} d \lambda^N= - 4\mathcal D_N(f_t^N).
\end{gathered}
\end{equation}
Putting everything together, we obtain
\begin{equation}
\frac{d}{dt}H_N(f_t^N)  = N \bar \tau(t) \int p_N f_t^N d \lambda^N - 4N \sigma (\mathcal D_N(f_t^N) + \widetilde {\mathcal D}_N(f_t^N))
\end{equation}
which, after a time integration, becomes,
\begin{equation}
\begin{gathered}
H_N(f_t^N) = H_N(f_0^N) +  N\int_0^t ds \bar \tau(s) \int p_N f_s^N d \lambda^N  -  4N \sigma \int_0^t (\mathcal D_N(f_s^N)+\widetilde {\mathcal D}_N(f_s^N))ds.
\end{gathered}
\end{equation}
We estimate the term involving $p_N$.
\begin{equation}
\begin{gathered}
N \int p_N f_s^N d \lambda^N = N \int (L_N^{\bar \tau(t)} q_N )f_s^N d \lambda^N =  \int (G_N^{\bar \tau(s)}q_N) f_s^N d \lambda^N = \int q_N \partial_s f_s^N d\lambda^N 
\end{gathered}
\end{equation}
where we have used the nontrivial identity $\tilde S_N q_N = 0$. Hence we get
\begin{equation}
\begin{gathered}
N\int_0^t ds \bar \tau(s) \int p_Nf_s^N d \lambda^N   =\bar \tau(t) \int q_N f_t^N d \lambda^N- \bar \tau(0) \int q_N f_0^N d \lambda^N - \int_0^t ds {\bar \tau}'(s) \int q_N f_s^N d \lambda^N
\end{gathered}
\end{equation}
By the entropy inequality and for any $\alpha >0$,
\begin{equation}
\begin{gathered}
 \int |q_N| f_t^N d \lambda^N \le\frac{1}{\alpha} H_N(f_t^N) + \log \int e^{\alpha |q_N|} d \lambda^N  \le \frac{1}{\alpha } H_N(f_t^N) + \log \int \prod_{i=1}^N   e^{\alpha |r_i|} d \lambda^N
\\
= \frac{1}{\alpha} H_N(f_t^N) + N \log \int_{-\infty}^{\infty} e^{\alpha |r_1| - V(r_1)} d r_1 = \frac{1}{\alpha} H_N(f_t^N) + C(\alpha) N,
\end{gathered}
\end{equation}
with $C(\alpha)$ is independent of $N$. Therefore,
\begin{equation}
\begin{gathered}
N \left| \int_0^t ds \bar \tau(s) \int p_N f_s^N d \lambda^N \right| \le \frac{K_{\bar \tau}}{\alpha} \left( H_N(f_t^N) + H_N(f_0^N) + \int_0^t H_N(f_s^N) ds \right) + N(2+t) K_{\bar \tau} C(\alpha),
\end{gathered}
\end{equation}
where $K_{\bar \tau} = \sup_{t \ge 0}\{|\bar \tau(t)|+|\bar \tau'(t)|\}$. Thus, choosing $\alpha = 2K_{\bar \tau}$,
\begin{equation} \label{eq:HNstima}
H_N(f_t^N) \le 3 H_N(f_0^N) + \int_0^t H_N(f_s^N) ds + C'N -8 N \sigma \int_0^t(\mathcal D_N(f_s^N) + \widetilde {\mathcal D}_N(f_s^N))ds,
\end{equation}
where $C'$ does not depend on $N$. Since $\mathcal D_N$ and $\widetilde {\mathcal D}_N$ are non-negative and since $H_N(f_0^N) \le C_0N$, by Gr\"onwall's inequality we obtain
\begin{equation}
H_N(f_t^N) \le C'' e^t N,
\end{equation}
for some $C''$ independent of $N$. Using this, equation \eqref{eq:HNstima} becomes
\begin{equation}
N \sigma \int_0^t (\mathcal D_N(f_s^N) + \widetilde {\mathcal D}_N(f_s^N))ds \le C'''(t) N,
\end{equation}
for some $C'''(t)$ independent of $N$, which completes the proof.
\end{proof}
In order to obtain an explicit bound on the one-block estimate, we make use a logarithmic Sobolev inequality.

For $1 \le m \le i \le N$, denote by $\barmi[\mu]\in \mathcal M_1(\R^{2m})$ 
the projection of the probability measure $\mu_t^N$
on $\{r_{i-m+1}, p_{i-m+1}, \dots, r_i, p_i\}$ conditioned to
\begin{equation}\label{eq:locav}
\frac{1}{m}\sum_{j=0}^{m-1}r_{i-j}=\rho,\quad  \frac{1}{m}\sum_{j=0}^{m-1}p_{i-j}=\bar p.
\end{equation}
Denote also by $\barmi[\lambda]$ the measure analogously obtained from $\lambda^N$. Since the potential $V$ is  {uniformly convex}, the density of the measure $\lambda^N$ is log-concave. The same applies to the conditional measures $\bar \lambda^{\rho,\bar p}_{m,i}$. Thus,  {the Bakry-Emery criterion applies} and we have the following {logarithmic} Sobolev inequality (LSI):
\begin{equation}\label{eq:logSob}
\int g^2 \log g^2 \ d \barmi[\lambda] \le C_{lsi} m^2 \sum_{j=1}^{m-1} 
\int\left[\left(D_{i-j} g\right)^2 + \left(\tilde D_{i-j} g\right)^2 \right] d \barmi[\lambda],
\end{equation}
if  $g^2$ is a smooth probability density on  $\R^{2m} $  {(with respect to $ \barmi[\lambda] $) }
and $C_{lsi}$ is a universal constant depending on the interaction $V$ only. 
A straightforward consequence of the LSI is the following lemma.
\begin{lem} \label{lem:LSI}
Let $\barmi[f]$ be the density of $\barmi[\mu]$ with respect to $\barmi[\lambda]$. Then
\begin{equation} \label{eq:LSIutile}
\sum_{i=m}^N \int \barmi[f] \log \barmi[f] d \barmi[\lambda] \le m^3 C_{lsi} (\mathcal D_N(f_t^N)+\widetilde {\mathcal D}_N(f_t^N)).
\end{equation}
\end{lem}
\begin{proof}

Choosing $g^2 =\barmi[f]$ in \eqref{eq:logSob}, and using Jensen inequality we obtain
\begin{equation}
\begin{gathered}
\sum_{i=m}^N \int \barmi[f] \log \barmi[f] d \barmi[\lambda] \le m^2 C_{lsi} \sum_{i=m}^N \sum_{j=1}^{m-1}\int  \frac{1}{4f_t^N} \left( \left(D_{i-j} f_t^N\right)^2 + \left(\tilde D_{i-j} f_t^N\right)^2 \right) d \lambda^N.
\end{gathered}
\end{equation}
The last step is noting that, when summing over $i$, any of the terms $\left(D_{i-j} f_t^N\right)^2$ or $\left(\tilde D_{i-j} f_t^N\right)^2$ appear at most $m$ times. This gives the extra factor $m$ and re-constructs the Dirichlet forms:
\begin{equation}
\begin{gathered}
\sum_{i=m}^N \int \barmi[f] \log \barmi[f] d \barmi[\lambda]\le m^3 C_{lsi} \sum_{i=1}^{N-1} \int  \frac{1}{4f_t^N} \left( \left(D_i f_t^N\right)^2 + \left(\tilde D_i f_t^N\right)^2 \right) d \lambda^N
\\
=m^3 C_{lsi} \left(\mathcal D_N(f_t^N)+\widetilde {\mathcal D}_N(f_t^N)\right).
\end{gathered}
\end{equation}
\end{proof}
\subsubsection{Block Estimates}
In this section we prove the three main estimate we have used in proving our main result: the energy, one-block and two-block estimates. In what follows the expectations $\E$ at time $t$ shall be evaluated in terms of integrals with respect to the measure $\mu_t^N$.
\begin{lem}[ {Energy estimate}] \label{lem:enestimate}
 {There  exists $C'_e(t)$ independent of $N$ such that}
\begin{equation} \label{eq:enestimate}
 {\sum_{i=1}^N \int \left(\frac{p_i^2}{2}+V(r_i)\right) d \mu_t^N \le C'_e(t) N.}
\end{equation}
\end{lem}
\begin{proof}
Let $\alpha >0$. By the entropy inequality we have
\begin{equation}
\begin{gathered}
\alpha\int \sum_{i=1}^N \left(\frac{p_i^2}{2}+V(r_i)\right) d \mu_t^N \le H_N(f_t^N) +  \log \int \exp \left(\alpha\sum_{i=1}^N\left( \frac{p_i^2}{2} +  V(r_i)\right)\right) d \lambda^N,
\\
{  =H_N(f_t^N) +  N\log \int \exp \left(\left(\alpha-\frac{1}{2}\right)p_i^2 +\left(\alpha-1\right)  V(r)-G(1,0)\right) d r \frac{dp}{\sqrt{2\pi}}}
\end{gathered}
\end{equation}
Since the integral at the right-hand side is convergent for $\alpha < 1/2$ and $H(f_t^N) \le C(t)N$, we have obtained, after fixing $\alpha$,
\begin{equation}
 \sum_{i=1}^N \int \left(\frac{p_i^2}{2}+V(r_i)\right) d \mu_t^N \le C'_e(t) N,
\end{equation}
for some $C'_e(t)$ independent of $N$.
\end{proof}
\begin{cor}
{There exists $C_e(t)$ independent of $N$ such that}
\begin{equation} 
\sum_{i=1}^N \int \left(p_i^2+r_i^2\right) d \mu_t^N \le C_e(t) N.
\end{equation}
\end{cor}
\begin{proof}
It easily follows from Lemma \ref{lem:enestimate} and the fact that $V''(r) \ge c_1$, for some $c_1 >0$ and large enough $r$.
\end{proof}
We denote, for $1 \le l \le i \le N$,
\begin{equation}\label{eq:bars}
\bar r_{l,i}:= \frac{1}{l}\sum_{j=0}^{l-1}r_{i-j}, \quad \bar p_{l,i}:= \frac{1}{l}\sum_{j=0}^{l-1}p_{i-j}, \quad \bar V'_{l,i}:= \frac{1}{l}\sum_{j=0}^{l-1}V'(r_{i-j}).
\end{equation}
\begin{lem}[One-block estimate] \label{lem:oneblock}
There exists $l_0 \in \mathbb N$ and $C_1'(t)$ independent of $N$ such that
\begin{equation} \label{eq:oneblock}
\sum_{i=l}^N \int_0^t \int \left( \bar V'_{l,i} - \tau \left(\bar r_{l,i} \right) \right)^2 d \mu_s^N ds \le C_1'(t) \left( \frac{N}{l} + \frac{l^2}{\sigma} \right),
\end{equation}
whenever $N\ge l > l_0$.
\end{lem}
\begin{proof}
Fix $\alpha >0$. By the entropy inequality and Lemma \ref{lem:LSI}:
\begin{equation}
\begin{gathered}
\sum_{i=l}^N\alpha  \int_0^t \int \left( \bar V'_{l,i} - \tau \left(\bar r_{l,i} \right) \right)^2 d \mu_s^N 
 \\
\le l^3 C_{lsi}\int_0^t (\mathcal D_N(s)+\widetilde {\mathcal D}_N(s)) d s 
 +t \sum_{i=l}^N  \log \int \exp\left(\alpha \left( \bar V'_{l,i}- \tau(\bar r_{l,i})\right)^2\right) d \barli[\lambda]
 \\
{ \le   C(t)\frac{l^3}{\sigma}
 +t \sum_{i=l}^N  \log \int \exp\left(\alpha \left( \bar V'_{l,i}- \tau(\rho)\right)^2\right) d \barli[\lambda],}
\end{gathered}
\end{equation}
where we have used the bound on the time integral of the Dirichlet form and 
the fact that $\bar r_{l,i}=\rho$ when integrating with respect to $\barli[\lambda]$. 
It is a standard result (cf \cite{Landim:2002p9150}, corollary 5.5).
that there exists a universal constant $C'$ and $l_0$ depending on $V$ only such that
\begin{equation}\label{eq:grannum}
\int e^\varphi d \barli[\lambda] \le C' \int e^\varphi d \lambda^N_{1, \bar p, \tau(\rho)}
\end{equation}
for any  {integrable} function $\varphi$ and whenever $l >l_0$. Hence, we obtain
\begin{equation}
\begin{gathered}
\sum_{i=l}^N\alpha  \int_0^t \int \left( \bar V'_{l,i} - \tau \left(\bar r_{l,i} \right) \right)^2 d \mu_s^Nds \le t \sum_{i=l}^N  \log \int C' \exp\left(\alpha \left( \bar V'_{l,i}- \tau(\rho)\right)^2\right) d  \lambda^N_{1,\bar p,\tau(\rho)} + C(t) \frac{l^3}{\sigma},
\end{gathered}
\end{equation}
We are left to estimate the expectation with respect to $\lambda^N_{\beta, \bar p,\tau(\rho)}$. In order to do so we introduce a normally distributed random variable $\xi$ and write
\begin{equation}
\begin{gathered}
 \int \exp\left(\alpha \left( \bar V'_{l,i}- \tau(\rho)\right)^2\right) d \lambda^N_{1, \bar p,\tau(\rho)}= \mathbb E_\xi \left[ \int \exp \left( \xi \sqrt{2 \alpha}\left(\bar V'_{l,i}-\tau(\rho) \right) \right) d \lambda^N_{1, \bar p,\tau(\rho)} \right] 
\\
= \mathbb E_\xi \left[ e^{-\tau(\rho) \xi \sqrt{2\alpha}} \left(\int \exp \left( \frac{\xi \sqrt{2\alpha}}{l}V'(r_1) \right) d \lambda^N_{1, \bar p,\tau(\rho)} \right)^l \right],
\end{gathered}
\end{equation}
Since, by Lemma \ref{lem:1blockapp},
\begin{equation}
 \int \exp \left( \frac{\xi \sqrt{2\alpha}}{l} V'(r_1) \right) d \lambda^N_{\beta, \bar p,\tau(\rho)} \le \exp \left( \frac{c_2\alpha}{l^2}\xi^2+ \frac{\tau(\rho) \sqrt{2\alpha}}{l}\xi\right),
\end{equation}
we obtain
\begin{equation}
 \int \exp\left(\alpha \left( \bar V'_{l,i}- \tau(\rho)\right)^2\right) d  \lambda^N_{1, \bar p,\tau(\rho)} \le  \mathbb E_\xi \left[ \exp\left(\frac{c_2 \alpha}{l} \xi^2  \right) \right],
\end{equation}
and the right hand side is independent of $\rho$ and $\bar p$.
Putting everything together yields
\begin{equation}
\begin{gathered}
\alpha \sum_{i=l}^N \int_0^t \int \left( \bar V'_{l,i} - \tau \left(\bar r_{l,i} \right) \right)^2 d \mu_s^N \le \frac{C(t) l^3}{\sigma} + (N-l+1)t \log \left(C'  \mathbb E_\xi \left[ \exp\left(\frac{c_2\alpha}{l} \xi^2  \right) \right]\right).
\end{gathered}
\end{equation}
and the conclusion follows taking $\alpha = l/(4c_2)$.
\end{proof}
\begin{lem}[Two-block estimate] \label{lem:twoblock}
Let $l_0$ as in Lemma \ref{lem:oneblock}. There exists $C_2'(t)$ independent of $N$ such that, for $l_0 < l \le m < N$,
\begin{equation} \label{eq:twoblock}
\sum_{i=l}^{N-m} \int_0^t \int \left( \bar \eta_{l,i+m}-\bar \eta_{l,i}\right)^2 d \mu_s^Nd s \le C_2'(t) \left( \frac{N}{l}+\frac{m^2}{\sigma}\right),
\end{equation}
whenever $\bar \eta_{l,i} \in \{ \bar p_{l,i}, \bar V'_{l,i}, \tau(\bar r_{l,i}), \bar r_{l,i} \}$.
\end{lem}
\begin{proof}
We start with $\bar \eta_{l,i} = \bar V'_{l,i}$. Denote by $V'_i := V'(r_i)$. The integration by parts formula
\begin{equation}
\begin{gathered}
\int \left( V'_{i+m}-V'_i \right) \varphi f_s^N d \lambda^N = \int \left( \parr[i+m]\varphi - \parr[i] \varphi \right) f_s^N d \lambda^N + \int \varphi \left( \parr[i+m]f_s^N -\parr[i] f_s^N\right) d \lambda^N,
\end{gathered}
\end{equation}
gives
\begin{equation}
\begin{gathered}
\int \left( \bar V'_{l,i+m}-\bar V'_{l,i}\right)^2d \mu_s^N = \frac{1}{l}\sum_{j=0}^{l-1}\int \left(V'_{i+m-j}-V'_{i-j} \right)\left(\bar V'_{l,i+m}-\bar V'_{l,i}\right) f_s^N d\lambda^N
\\
=\frac{1}{l}\sum_{j=0}^{l-1} \int  \left( \parr[i+m-j] \left(\bar V'_{l,i+m}-\bar V'_{l,i}\right)- \parr[i-j] \left(\bar V'_{l,i+m}-\bar V'_{l,i}\right) \right) d \mu_s^N +
\\
+ \frac{1}{l}\sum_{j=0}^{l-1}\int \left(\bar V'_{l,i+m}-\bar V'_{l,i}\right) \left( \parr[i+m-j] f_s^N- \parr[i-j]f_s^N \right)d \lambda^N.
\end{gathered}
\end{equation}
We evaluate
\begin{equation}
\parr[i+m-j] \left(\bar V'_{l,i+m}-\bar V'_{l,i}\right)- \parr[i-j] \left(\bar V'_{l,i+m}-\bar V'_{l,i}\right) = \frac{1}{l}\left( V''_{i+m-j}+V''_{i-j}\right),
\end{equation}
 Moreover, by the Cauchy-Schwarz inequality:
\begin{equation}
\begin{gathered}
\frac{1}{l}\sum_{j=0}^{l-1}\int \left(\bar V'_{l,i+m}-\bar V'_{l,i}\right) \left( \parr[i+m-j] f_s^N- \parr[i-j]f_s^N \right)d \lambda^N 
\\
\le \frac{1}{2} \int (\bar V'_{l,i+m}-\bar V'_{l,i})^2 d\mu_s^N + \frac{2}{l}\sum_{j=0}^{l-1} \int \frac{1}{f_s^N} \left( \parr[i+m-j] f_s^N - \parr[i-j] f_s^N \right)^2 d \lambda^N.
\end{gathered}
\end{equation}
Finally, we estimate
\begin{equation}
\begin{gathered}
\frac{2}{l}\sum_{j=0}^{l-1} \int\frac{1}{f_s^N} \left( \parr[i+m-j] f_s^N - \parr[i-j] f_s^N \right)^2 d \lambda^N 
\\
{= \frac{2}{l}\sum_{j=0}^{l-1}  \int\frac{1}{f_s^N} \left(\sum_{k=i-j}^{i+m-j-1}\left( \parr[k+1]f_s^N - \parr[k]f_s^N\right) \right)^2 d \lambda^N}
\\
{\le \frac{2m}{l} \sum_{j=0}^{l-1} \sum_{k=i-j}^{i+m-j-1} \int \frac{1}{f_s^N} \left( \parr[k+1] f_s^N - \parr[k]f_s^N\right)^2 d \lambda^N }
\\
{\le 2m  \sum_{k=i-l+1}^{i+m-1} \int\frac{1}{f_s^N} \left( \parr[k+1] f_s^N - \parr[k]f_s^N\right)^2 d \lambda^N}
\end{gathered}
\end{equation}
Putting everything together we obtain
\begin{equation}
\begin{gathered}
\sum_{i=l}^{N-m} \int_0^t \left(\bar V'_{l,i+m}-\bar V'_{l,i} \right)^2 d \mu_s^N d s \le 2t\frac{N-l-m+1}{l}\norminf[V'']  +\frac{1}{2}\sum_{i=l}^{N-m} \int_0^t \left(\bar V'_{l,i+m}-\bar V'_{l,i} \right)^2 d \mu_s^N d s+
\\
{+ 2m\sum_{i=l}^{N-m} \sum_{k=i-l+1}^{i+m-1}\int_0^t \int\frac{1}{f_s^N} \left( \parr[k+1] f_s^N - \parr[k]f_s^N\right)^2 d \lambda^N  ds,}
\end{gathered}
\end{equation}
which leads to the conclusion, since we gain an extra factor $m+l-1$ in the last term order to rebuild the Dirichlet form $\tilde D(f_s^N)$ as in the proof of the Lemma \ref{lem:LSI}.

Thanks to the integration by parts formula
\begin{equation}
\int \left( p_{i+m}-p_i \right) \varphi f_s^N d \lambda^N = \int \left(  \parp[i+m]\varphi -  \parp[i] \varphi \right) \varphi f_s^N d \lambda^N + \int \varphi \left(  \parp[i+m]\varphi f_s^N - \parp[i] \varphi f_s^N\right) d \lambda^N,
\end{equation}
the case $\bar \eta_{l,i} = \bar p_{l,i}$ is analogous. Finally. we treat the cases $\bar \eta_{l,i} = \bar r_{l,i}$ and $\bar \eta_{l,i} = \tau(\bar r_{l,i})$ simultaneously. Since, by  Appendix \ref{app:tension}, $\tau'$ is bounded from below, we have, for some constant $C$,
\begin{equation}
\begin{gathered}
C^2 \left(\bar r_{l,i+m}- \bar r_{l,i} \right)^2 \le \left( \tau(\bar r_{l,i+m})-\tau(\bar r_{l,i}) \right)^2 
\\
\le 3 \left( \bar V'_{l,i+m}-\bar V'_{l,i} \right)^2+3 \left( \tau(\bar r_{l,i})- \bar V'_{l,i}\right)^2+3 \left( \tau(\bar r_{l,i+m})-\bar V'_{l,i+m} \right)^2,
\end{gathered}
\end{equation}
which imply the conclusion by the first part of the proof and the one block estimate.
\end{proof}
The two-block estimates can be written in terms of the averages $\hat \eta_{l,i}$ thanks to the following Lemma.
\begin{lem} \label{lem:etahat}
\begin{equation}
\hat \eta_{l,i+1}-\hat \eta_{l,i} = \frac{1}{l} \left(\bar \eta_{l,i+l}-\bar \eta_{l,i} \right).
\end{equation}
\end{lem}
\begin{proof}
We prove the statement by induction over $l$, for each fixed $k$. The statement for $l=1$ is obvious, since both $\hat\eta_{1,i+1}-\hat\eta_{1,i}$ and $\bar \eta_{1,i+1}-\bar \eta_{1,i}$ are equal to $\eta_{i+1} -\eta_i$.

Assume now the statement is true for some $l \ge 1$, that is
\begin{equation}
\hat\eta_{l,i+1}-\hat\eta_{l,i}= \frac{1}{l}\left(\bar \eta_{l,i+l}-\bar \eta_{l,i}\right).
\end{equation}
We prove it holds for $l+1$ as well. We have, in fact
\begin{equation}
\begin{gathered}
\hat\eta_{l+1,i+1}-\hat\eta_{l+1,i} =  \frac{1}{l+1} \sum_{|j| <l+1} \frac{l+1-|j|}{l+1}\eta_{i+1-j}-\frac{1}{l+1} \sum_{|j| <l+1} \frac{l+1-|j|}{l+1}\eta_{i-j}
\\
= \frac{1}{(l+1)^2} \sum_{|j| <l+1} (l+1-|j|)(\eta_{i+1-j}-\eta_{i-j})
\\
= \frac{l^2}{(l+1)^2} \frac{1}{l} \sum_{|j|<l} \frac{l-|j|}{l}(\eta_{i+1-j}-\eta_{i-j}) +\frac{1}{(l+1)^2} \sum_{|j| < l+1} (\eta_{i+1-j}-\eta_{i-j}).
\end{gathered}
\end{equation}
For the first summation we can use our inductive hypothesis, while the second summation is a telescopic one. Therefore we obtain
\begin{equation}
\begin{gathered}
\hat\eta_{l+1,i+1}-\hat\eta_{l+1,i} = \frac{1}{(l+1)^2} \sum_{j=0}^{l-1} (\eta_{i+l-j}-\eta_{i-j}) + \frac{1}{(l+1)^2} (\eta_{i+l+1}- \eta_{k-l})
\\
= \frac{1}{(l+1)^2} \left(\sum_{j=1}^l \eta_{i+l+1-j} +\eta_{i+l+1} - \sum_{j=0}^{l-1}\eta_{i-j}-\eta_{k-l}\right)
\\
= \frac{1}{(l+1)^2} \sum_{j=0}^l (\eta_{i+l+1-j}-\eta_{i-j})
\\
= \frac{1}{l+1}(\bar\eta_{l+1,i+l+1}-\bar\eta_{l+1,i}).
\end{gathered}
\end{equation}
\end{proof}
From the previous lemma and the two-block estimate it follows:
\begin{cor} \label{cor:twoblock}
Let $N \ge l > l_0$ and $\hat \eta_{l,i} \in \{ \hat p_{l,i}, \hat V'_{l,i}, \tau (\hat r_{l,i}), \hat r_{l,i}\}$. There is $C_2'(t)$ independent of $N$ such that
\begin{equation} \label{eq:twoblockuse}
\sum_{i=l}^{N-l} \int_0^t \int (\hat \eta_{l,i+1}-\hat\eta_{l,i})^2 d \mu_s^Nds \le C_2(t) \left( \frac{N}{l^3} +\frac{1}{\sigma} \right).
\end{equation}
\end{cor}
\begin{proof}
Let $\hat \eta_{l,i} \in \{\hat p_{l,i}, \hat V'_{l,i}, \hat r_{l,i}\}$. Then, from
\begin{equation}
\hat \eta_{l,i+1}-\hat \eta_{l,i} = \frac{1}{l}(\bar \eta_{l,i+l}-\bar \eta_{l,i}),
\end{equation}
where $\bar \eta_{l,i}$ is defined as in the previous lemma, we have
\begin{equation}
(\hat \eta_{l,i+1}-\hat\eta_{l,i})^2 = \frac{1}{l^2} (\bar \eta_{l,i+l}-\bar \eta_{l,i})^2,
\end{equation}
and the conclusion follows from Lemma \ref{lem:twoblock} with $m=l$.

For $\hat \eta_{l,i} = \tau(\hat r_{l,i})$ the conclusion follow once more from the lemma, since
\begin{equation}
(\tau(\hat r_{l,i+1})-\tau(\hat r_{l,i}))^2 \le C (\hat r_{l,i+1}-\hat r_{l,i})^2.
\end{equation}
\end{proof}
Finally, we compare the averages $\bar \eta_{l,i}$ and $\hat \eta_{l,i}$. This allows us to write the one-block estimate in terms of the averages $\hat \eta_{l,i}$.
\begin{lem}\label{lem:avcomp}
Let $l_0$ be as in Lemma \ref{lem:oneblock}. There is $C_3(t)$ independent of $N$ such that
\begin{equation} \label{eq:avcomp}
\sum_{i=l}^{N-l+1} \int_0^t \int \left(\hat \eta_{l,i} - \bar \eta_{l,i} \right)^2 d \mu_s^N \le C_3(t) \left( \frac{N}{l}+\frac{l^2}{\sigma} \right),
\end{equation}
for  $\eta_j \in \{ r_j, p_j, V'(r_j) \}$ and whenever $N \ge l > l_0$.
\end{lem}
\begin{proof}
We prove the statement for $\eta_j = V'(r_j)$, first. Define
\begin{equation}
{\bar  \partial_{l,i}:= \frac{1}{l}\sum_{j=0}^{l-1}\parr[i-j], \quad \hat  \partial_{l,i}:= \frac{1}{l}\sum_{|j|<l}\frac{l-|j|}{l}\parr[i-j].}
\end{equation}
 Integrating by parts we have
\begin{equation}
\begin{gathered} \label{eq:compeq1}
 \int \left( \hat V'_{l,i}-\bar V'_{l,i} \right)^2 d\mu_s^N = \int \left( \hat V'_{l,i}-\bar V'_{l,i} \right) \left( \hat V'_{l,i}-\bar V'_{l,i} \right) f_s^N d\lambda^N
\\
=  \int \left( \hat \partial_{l,i} - \bar  \partial_{l,i}\right) \left( \hat V'_{l,i}-\bar V'_{l,i} \right) d\mu_s^N +  \int \left( \hat V'_{l,i}-\bar V'_{l,i} \right) \left( \hat { \partial}_{l,i}f_s^N - \bar { \partial}_{l,i}f_s^N\right) d \lambda^N.
\end{gathered}
\end{equation}
We can write
\begin{equation}
\hat \partial_{l,i}-\bar \partial_{l,i} = \frac{1}{l} \sum_{|j|<l}c_j \parr[i-j], \quad \hat V'_{l,i}-\bar V'_{l,i} = \frac{1}{l} \sum_{|j|<l}c_j V'_{i-j}
\end{equation}
where the numbers $c_j$ have the following properties:  $c_j^2 \le 1$, and $\underset{|j|<l}{\sum} c_j =0$. 
This allows us to estimate
\begin{equation}
 {\left( \hat  \partial_{l,i} - \bar  \partial_{l,i}\right)} \left( \hat V'_{l,i}-\bar V'_{l,i} \right) = \frac{1}{l^2}\sum_{|i|<l} c_i^2 V''(r_{i-j}) \le \frac{2\norminf[V'']}{ l}.
\end{equation}
For the remaining term in \eqref{eq:compeq1} we use Cauchy-Schwarz:
\begin{equation}
\int \left( \hat V'_{l,i}-\bar V'_{l,i} \right) \left( \hat { \partial}_{l,i}f_s^N - \bar { \partial}_{l,i}f_s^N\right) d \lambda^N \le \frac{1}{2}  \int \left( \hat V'_{l,i}-\bar V'_{l,i} \right)^2 d\mu_s^N  + \frac{1}{2} \int \frac{1}{f_s^N} \left( \hat { \partial}_{l,i}f_s^N - \bar { \partial}_{l,i}f_s^N\right)^2 d \lambda^N.
\end{equation}
The last term at the right-hand side evaluates as
\begin{equation}
\hat { \partial}_{l,i}f_s^N - \bar { \partial}_{l,i}f_s^N = \frac{1}{l} \sum_{j=1}^l \frac{j}{l} \left( \parr[i-j+l]f_s^N- \parr[i-j]f_s^N \right).
\end{equation}
Therefore
\begin{equation}
\int \frac{1}{f_s^N} \left( \hat{ \partial}_{l,i} f_s^N - \bar{ \partial}_{l,i}f_s^N \right)^2 d\lambda^N = \frac{4}{l^2} \int\frac{1}{f_s^N} \left(\sum_{j=1}^l \frac{j}{l} \left(  \parr[i-j+l]f_s^N- \parr[i-j]f_s^N \right)\right)^2 d \lambda^N
\end{equation}
is estimated by the Dirichlet form as in the proof of the two block estimate, since $j/l < 1$, leading to the conclusion.

The proof of the statement for $\eta_j=p_j$ is analogous. We are left with the case $\eta_j = r_j$. Since we do not have an integration by parts formula involving $r_j$ alone, we follow the proof of Lemma \ref{lem:oneblock}: for any $\alpha >0$,
\begin{equation}
\begin{gathered}
\sum_{i=l}^N\alpha  \int_0^t \int \left(\hat r_{l,i}-\bar r_{l,i} \right)^2 d \mu_s^N \le t \sum_{i=l}^N  \log \int e^{\alpha \left( \hat r_{l,i}-\bar r_{l,i}\right)^2} d \bar \lambda^{\rho,\bar p}_{2l-1,i+l-1} + C(t) \frac{(2l-1)^3}{\sigma},
\end{gathered}
\end{equation}
We write
\begin{equation}
 \int \exp \left(\alpha (\hat r_{l,i}-\bar r_{l,i})^2 \right) d \bar \lambda^{\rho,\bar p}_{2l-1,i+l-1} \le C'  \int \exp \left(\alpha (\hat r_{l,i}-\bar r_{l,i})^2 \right) d  \lambda^N_{1, \bar p,\tau(\rho)} 
\end{equation}
\begin{equation}
= C' \mathbb E_\xi \left[ \int \exp \left(\sqrt{2\alpha} \xi (\hat r_{l,i}-\bar r_{l,i}) \right) d  \lambda^N_{1, \bar p, \tau(\rho)} \right],
\end{equation}
where $\xi$ is a normally distributed random. In order to calculate the last integral, we define $G(\tau):=G(1,\tau)$. Recalling that $G$ is smooth and $G''$ is bounded (see Lemma \ref{lem:tau'bound}), we write
\begin{equation}
\begin{gathered}
\int \exp \left(\sqrt{2\alpha} \xi (\hat r_{l,i}-\bar r_{l,i}) \right) d  \lambda^N_{1, \bar p,\tau(\rho)}  =\int \exp \left(\frac{\sqrt{2\alpha} \xi}{l} \sum_{|j|<l} c_j r_{i-j}\right) d  \lambda^N_{1, \bar p,\tau(\rho)} 
\\
= \prod_{|j|<l} \int \exp\left( \frac{\sqrt{2\alpha}\xi}{l} c_j r_{i-j} + \tau(\rho) r_{i-j} -V(r_{i-j})-G(\tau(\rho)) \right) d r_{i-j}
\\
= \prod_{|j|<l} \exp \left( G\left(\tau(\rho)+\frac{\sqrt{2\alpha}\xi}{l} c_j \right) - G(\tau(\rho)) \right)
\\
= \exp \sum_{|j|<l} \left(G'(\tau(\rho)) \frac{\sqrt{2\alpha}\xi}{l} c_j + G''(\tilde \tau) \frac{\alpha \xi^2}{l^2} c_j^2 \right),
\end{gathered}
\end{equation}
for some intermediate value $\tilde \tau$. Since $\sum_j c_j = 0$, $c_j^2 \le 1$ we have
\begin{equation}
\exp \sum_{|j|<l} \left(G'(\tau(\rho)) \frac{\sqrt{2\alpha}\xi}{l} c_j + G''(\tilde \tau) \frac{\alpha \xi^2}{l^2} c_j^2 \right) \le \exp \left(\frac{ 2 \alpha(2l+1) \| G'' \|_{L^\infty} }{l^2} \xi^2 \right).
\end{equation}
Therefore we have obtained
\begin{equation}
 \int \exp \left(\alpha (\hat r_{l,i}-\bar r_{l,i})^2 \right) d \bar \lambda^{\rho,\bar p}_{2l-1,i+l-1} \le C'  \mathbb E_\xi \left[ \exp \left( \frac{6 \alpha \| G'' \|_{L^\infty} }{l}\xi^2 \right) \right],
\end{equation}
and again the right hand side is independent of $\rho$ and $\bar p$. The conclusion then follows as in the proof of the one block estimate.
\end{proof}
We end this section by stating the one block estimate in terms of the averages $\hat \eta_{l,i}$.
\begin{cor} \label{cor:V'tau}
Let $l_0$ be as in Lemma \ref{lem:oneblock}. There is $C_1(t)$ independent of $N$ such that
\begin{equation} \label{eq:V'tau}
\sum_{i=l}^{N-l+1} \int_0^t \int (\hat V'_{l,i}-\tau(\hat r_{l,i}))^2 d \mu_s^N ds \le C_1(t) \left( \frac{N}{l} + \frac{l^2}{\sigma} \right),
\end{equation}
whenever $N \ge l > l_0$.
\end{cor}
\begin{proof}
It follows immediately from the first one block estimate and the average comparison, since
\begin{equation}
(\hat V'_{l,i}-\tau(\hat r_{l,i}))^2 \le 3 (\hat V'_{l,i}-\bar V'_{l,i})^2 + 3(\bar V'_{l,i}-\tau (\bar r_{l,i}))^2 + 3(\tau (\bar r_{l,i})-\tau(\hat r_{l,i}))^2
\end{equation}
and
\begin{equation}
(\tau (\bar r_{l,i})-\tau(\hat r_{l,i}))^2 \le C (\bar r_{l,i}- \hat r_{l,i})^2.
\end{equation}
\end{proof}

\section{Thermodynamic Consequences} \label{sec:thermo}

In this final section we want to prove that any limit distribution $\frak Q$ of $\frak Q_N$
 satisfy the thermodynamic principles applied to isothermal transformations. Throughout this section, we shall restore $\beta$.



In order to perform a isothermal thermodynamic transformation we fix $\tau_0,
\tau_1, t_1 \in \mathbb R$ and take the external tension $\bar \tau$
to be a smooth function such that $\bar \tau(0) = \tau_0$ and $\tilde
\tau(t) = \tau_1$ for all $t \ge t_1$. This corresponds to the
following physical situation: at time $0$ the system is at
equilibrium, and the equilibrium state is determined by
$(\beta,\tau_0)$. Then we vary the external tension and we eventually
bring the system to another equilibrium state (reached asymptotically
as  $t \to \infty$), identified by $(\beta,\tau_1)$ (we are performing
isothermal transformations, so the temperature does not change). This
is the way we define a a thermodynamic isothermal transformation
between two equilibrium states $(\beta,\tau_0)$ and $(\beta,\tau_1)$. 

Recall the definition of the Gibbs potential:
\begin{equation}\label{eq:thermoG}
 {G(\beta,\tau) = \log \int_{-\infty}^{\infty}\exp(- \beta V(r)+\beta \tau r)dr}.
\end{equation}
Moreover, the free energy $F$ is defined as
\begin{equation}\label{eq:thermoF}
F(\beta,\rho) = \sup_{\tau \in \mathbb R} \{\tau \rho - \beta^{-1}G(\beta,\tau) \}
\end{equation}
and the tension $\tau_\beta$ is given by
\begin{equation}\label{eq:thermotau}
\tau_\beta(\rho) = \partial_\rho F(\beta,\rho).
\end{equation}
Finally, the \emph{internal energy} $U$ is defined by
\begin{equation} \label{eq:internal}
U(\beta,\tau) = \E_{\lambo}\left[\frac{p^2}{2}+V(r)\right].
\end{equation}
Throughout this section
we need the following assumption on the convergence of the energy:
\begin{ass}
 {For}
\begin{equation}
 {E_N(t) := \frac{1}{N} \sum_{i=1}^N \left( \frac{p_i(t)^2}{2}+V(r_i(t)) \right)},
\end{equation}
 {assume}
  \begin{equation}\label{eq:internalass}
    \lim_{N \to \infty}\mathbb E\left[E_N(t) - \int_0^1 U (\beta,  \tau_\beta(r_N(t,x))dx \right] \ =\ 0,
  \end{equation}
  and
  \begin{equation}
    \lim_{N \to \infty}E_N(0) =   U(\beta,\tau_0).
  \end{equation}
  Furthermore, under the proper convergent subsequence,
  \begin{equation} \label{eq:energyconv} \lim_{N \to \infty} \int_0^1
    \int_0^\infty \mathbb \psi(t,x) \mathbb E^{\frak Q_N} \left[ f(\uhat_N(t,x)) \right] dx dt =\int_0^1
    \int_0^\infty \psi(t,x)  \mathbb E^{\frak Q}\left[ f(\utilde(t,x)) \right] dx dt.
  \end{equation}
  for all test functions $\psi$ and all continuous $f$ with
  \emph{quadratic} growth.
\end{ass}

\begin{oss}
Assumption A is necessary because all our bounds rely on relative
entropy that is not sufficient to give the uniform integrability for
the convergence of second moments. 
\end{oss}
We also need some assumptions on the weak solutions considered:
\begin{itemize}
\item When tension is held at the constant value $\tau_1$,
  \begin{equation}
    \lim_{t \to\infty} \tau_\beta(\tilde r(t,x)) = \tau_1, \quad \lim_{t
      \to \infty} \tilde p(t,x) = 0
  \end{equation}
  for almost all $ x \in [0,1]$.
\item \begin{equation}
  \label{eq:1}
  \mathcal L(t) := \int_0^1 r(t,y) \; dy. 
\end{equation}
 is a bounded variation function of $t$. This is necessary in order to define 
the macroscopic work $W(t)$ below.
\end{itemize}

The first law of thermodynamics is 
an energy balance which takes into account energy loss (or gain) via
heat exchange. It reads as follows. 
\begin{equation} \label{eq:firstLaw}
\Delta U = W + Q,
\end{equation}
where $\Delta U$ is the difference of internal energy between two equilibrium states, $W$ is the work done on the system (which depends on the external force $\bar \tau$) and $Q$ is the heat exchanged (which depends on the noise, i.e. on $\sigma$). In order to deduce the first principle, we use the equations \eqref{eq:SDE} to obtain, from a direct calculation,
\begin{equation}
E_N(t) - E_N(0) = \int_0^t d E_N(s) = W_N(t) + Q_N(t),
\end{equation}
where
\begin{equation} \label{eq:WN0}
W_N(t) = \int_0^t \bar \tau(s)  p_N(s)  \; ds 
\end{equation}
and
\begin{equation} \label{eq:QN}
\begin{gathered}
Q_N(t) = - \sigma \int_0^t \sum_{i=1}^{N-1} \left( (\nabla p_i(s))^2 -
  2 \beta^{-1} \right) ds 
+ \sqrt{ \frac{2\sigma}{\beta N}} \sum_{i=1}^{N-1} \int_0^t (\nabla p_i(s)) d w_i(s).
\end{gathered}
\end{equation}
$W_N$ can be rewritten in the more expressive form
\begin{equation} \label{eq:WN1}
W_N(t) = \int_0^t \bar \tau(s) d\left( \frac{q_N(s)}{N} \right) =  
\int_0^t \bar \tau(s) d \left(\frac{1}{N} \sum_{i=1}^N r_i(s) \right),
\end{equation}
so that we have
\begin{equation} \label{eq:lavoro}
W(t)=\lim_{N \to \infty}W_N(t) = \int_0^t \bar \tau(s) d \mathcal L(s) ,
\end{equation}
$W(t)$ is the macroscopic work done by the external tension up to time $t$.

By our Assumption A about the convergence of the energy, we obtain that $E_N(t)-E_N(0)$ converges to the difference of internal energy, and so $Q_N(t)$ converges, as $N \to\infty$ to the quantity
\begin{equation} \label{eq:Q}
Q(t) = \int_0^1 ( U(\beta,\tau_\beta(\tilde r(t,x))) - U(\beta,\tau_0)) dx- W(t).
\end{equation}
Therefore, taking the limit $t \to\infty$ we obtain the first principle of thermodynamics
\begin{equation} \label{eq:Iprinciple}
 U(\beta, \tau_1) - U(\beta,\tau_0)= W + Q.
\end{equation}
Let us move now to the second principle. It states that, during a isothermal
thermodynamic transformation, 
\begin{equation} \label{eq:IIprinciple}
\Delta S \ge  \beta  Q,
\end{equation}
where $\Delta S$ is the difference of entropy and $Q$ is the heat. The equality holds only for \emph{reversible} (or
quasistatic) transformations.
The entropy $S$ is defined by
\begin{equation}\label{eq:Sdefn}
S = \beta (U-F).
\end{equation}
where $F$ is the free energy and $U$ is the internal energy. We can
combine the first principle \eqref{eq:Iprinciple} and \eqref{eq:Sdefn}
to have an equivalent formulation of the second principle for an isothermal transformation. In fact, we have
\[
\Delta F = \Delta U - T \Delta S = W + Q - \beta^{-1} \Delta S.
\]
Therefore, since $\beta$ is positive, the second principle is equivalent to the following
\emph{inequality of Clausius}
\begin{equation} \label{eq:Clausius}
\Delta F \le W.
\end{equation}
We show  that \eqref{eq:Clausius} for our system is a consequence of the above assumptions and 
the assumption that the hydrodynamic limit concentrates on the the vanishing viscosity solutions.

Define the free energy at time $t$ by
\begin{equation} \label{eq:freet}
\mathcal F(t) = \int_0^1 \left( \frac{\tilde p(t,x)^2}{2}+ F(\beta, \tilde r(t,x)) \right) dx.
\end{equation}
Notice the presence of the macroscopic kinetic term in \eqref{eq:freet}, that eventually 
disappears when the system reach global equilibrium. 
It follows from the initial and asymptotic conditions on $\tilde r$ and $\tilde p$ that
\begin{equation}
\mathcal F(0)=F(\beta, \tau_\beta^{-1}(\tau_0))  ,\quad \lim_{t \to\infty}
\mathcal F(t)=  F(\beta,\tau_\beta^{-1}(\tau_1))  .
\end{equation}
In Appendix \ref{sec:visc-appr} we show that the vanishing viscous solutions satisfy 
$\mathcal F(t) - \mathcal F(0) \le W(t)$, and consequently 
\begin{equation} \label{eq:Clausiusnoi}
F(\beta,\tau_\beta^{-1}(\tau_1)) - F(\beta,\tau_\beta^{-1}(\tau_0)) \le W,
\end{equation}
where $W$ is defined in \eqref{eq:Iprinciple}.

\appendix
\section{Properties of the Tension} \label{app:tension}

In this section we shall give some technical properties about the tension $\tau$. In order to simplify the notation we set $\beta=1$ once again. Thus, we define
\begin{equation} \label{eq:apptau}
F(\rho) = \sup_{\tau \in \mathbb
  R}\{\tau \rho - G(\tau)\}, \qquad \tau(\rho) = F'(\rho),
\end{equation}
where
\begin{equation} \label{eq:appF}
 {G(\tau) =\log \int_{-\infty}^{\infty} e^{\tau r - V(r)}dr}.
\end{equation}
 We will prove the following
\begin{prop} \label{prop:appmain}
Let the potential $V \in C^2(\mathbb R)$ be  {uniformly convex with quadratic growth}, in the sense that there exist $c_1,c_2 \in \mathbb R$ such that
\begin{equation} \label{eq:appV''}
0<c_1 \le V''(r) \le c_2, \quad \forall r \in \mathbb R.
\end{equation}
Moreover, assume there exist some positive constants $V_+'', V_-'', \alpha, R$ such that
\begin{equation} \label{eq:appAN}
\begin{gathered}
\left| V''(r)-V_+'' \right| < e^{-\alpha r},  \quad \forall r>R
\\
\left| V''(r)-V_-'' \right| < e^{\alpha r}, \quad \forall r<-R
\end{gathered} \ .
\end{equation}
Then the following properties hold true.
\begin{itemize}
\item [i)] The p-system \eqref{eq:psystem} is strictly hyperbolic, meaning $\tau'(\rho) \ge c_1 >0$ for all $\rho \in \mathbb R$;
\item [ii)] $\tau''(\rho)(\tau'(\rho))^{-5/4}$ and $\tau'''(\rho)(\tau'(\rho))^{-7/4}$ are in $L^2(\mathbb R)$, while $\tau''(\rho)(\tau'(\rho))^{-3/4}$ and $\tau'''(\rho)(\tau'(\rho))^{-2}$ are in $L^\infty(\mathbb R)$.
\item [iii)] $\tau'(\rho) \le c_2$ for all $\rho \in \mathbb R$. Moreover, $\tau(\rho)/F(\rho) \to 0$ as $|\rho| \to \infty$.
\end{itemize}
 {Finally, we let $V$ be a mollification of the function}
\begin{equation} \label{eq:Vvera}
 {r \longmapsto \frac{1}{2}(1-\kappa)r^2+\frac{1}{2}\kappa r|r|_+,}
\end{equation}
 {where $|r|_+=\max\{r,0\}$ and $0<\kappa < 1/3$.} 
\begin{itemize}
\item[iv)]  {$\tau''(\rho) > 0$ for all $\rho \in \mathbb R$. In particular, the p-system }\eqref{eq:psystem} { is genuinely nonlinear.}
\end{itemize}
\end{prop}
We prove the previous propositions through a series of lemmas. 

Fix $\bar p,\tau \in \mathbb R$. We denote by
$\lambda_{\bar p,\tau}$ the probability measure on $\mathbb
R^2$ defined by
\begin{equation} \label{eq:lambdapirho}
\int_{\mathbb R^2} f(r,p) d \lambda_{\bar p,\tau}(r,p) =
 \int_{\mathbb R^2} f(r,p)
e^{-\frac{(p-\bar p)^2}{2}+ \tau r - V(r)-G(\tau)}dr\frac{dp}{\sqrt{2\pi}},
\end{equation}
for any measurable $f : \mathbb R^2 \to \mathbb R$.

The first lemma we state is used in the proof of the one block estimate.
\begin{lem}\label{lem:1blockapp}
Let $\alpha \in \mathbb R$. Then
\begin{equation}
\int e^{\alpha V'(r)} d \lambda_{\bar p,\tau} \le\exp\left( \alpha \tau + \frac{c_2}{2}\alpha^2\right).
\end{equation}
\end{lem}
\begin{proof}
Let
\begin{equation}
A(\alpha,\tau) = \log  \int e^{\alpha V'(r)} d \lambda_{\bar p,\tau}.
\end{equation}
Then, integrating by parts, we have
\begin{equation}
\begin{gathered}
\partial_\alpha A(\alpha, \tau) = \int_{-\infty}^\infty V'(r) \exp \left( \alpha V'(r) + \tau r-V(r)-G(\tau)-A(\alpha,\tau) \right)dr
\\
=\int_{-\infty}^\infty (\alpha V''(r)+\tau)\exp \left( \alpha V'(r) + \tau r-V(r)-G(\tau)-A(\alpha,\tau) \right)dr.
\end{gathered}
\end{equation}
Since
\begin{equation}
c_1 \le V''(r) \le c_2, \quad \forall r \in \mathbb R,
\end{equation}
if $\alpha >0$ we obtain
\begin{equation} \label{eq:alphapos}
\partial_\alpha A(\alpha, \tau) \le c_2 \alpha+ \tau,
\end{equation}
while, if $\alpha \le 0$
\begin{equation} \label{eq:alphaneg}
\partial_\alpha A(\alpha,\tau) \ge c_2 \alpha+ \tau.
\end{equation}
\eqref{eq:alphapos}, together with \eqref{eq:alphaneg} and $A(0,\tau) =0$ imply
\begin{equation}
A(\alpha,\tau) \le \alpha \tau + \frac{c_2}{2}\alpha^2,
\end{equation}
for all $\alpha \in \mathbb R$, from which the claim follows.
\end{proof}
\begin{lem} \label{lem:tau'bound}
Let $\tau$ and $G$ as in \eqref{eq:apptau} and \eqref{eq:appF}. Moreover, let $c_1$ and $c_2$ be as in \eqref{eq:appV''}. Then $c_2^{-1} \le G''(\tau) \le c_1^{-1}$ for all $\tau \in \mathbb R$. Moreover, $c_1 \le \tau'(\rho) \le c_2$ for all $\rho \in \mathbb R$.
\end{lem}
\begin{proof}
Let $\rho(\tau)$ be the expectation value of
$r$ with respect to $\lambda_{\bar p,\tau}$. We have
\begin{equation} \label{eq:rhogamma}
\rho(\tau) = \int r \exp \left(\tau r-V(r)-G(\tau)\right)dr = G'(\tau)
\end{equation}
and so
\begin{equation}
\begin{gathered}
G''(\tau)=  \int r^2 \exp\left(\tau r-V(r)-G(\tau)\right)dr -\int r G'(\tau)\exp\left(\tau r-V(r)-G(\tau)\right)dr
\\
= \int r^2 d\lambda_{\bar p,\tau}- \left( \int r d \lambda_{\bar p,\tau}\right)^2 
\\
= \int (r^2-\rho(\tau)^2) d\lambda_{\bar p,\tau}
\\
= \int (r-\rho(\tau))^2 d \lambda_{\bar p,\tau} + 2\rho \int (r-\rho(\tau)) d \lambda_{\bar p,\tau}
\\
= \int (r-\rho(\tau))^2 d \lambda_{\bar p,\tau} > 0.
\end{gathered}
\end{equation}
Therefore $G$ is smooth and convex on $\mathbb R$, and so is its Legendre transform $F$. Then, integrating by parts yields
\begin{equation}
\begin{gathered}
1 = \int (r-\rho(\tau)) \left(V'(r)-\tau\right)d\lambda_{\bar p,\tau} 
\\
= \int (r-\rho(\tau)) \left( V'(r)-V'(\rho(\tau))\right)d\lambda_{\bar p,\tau} = \int (r-\rho(\tau))^2 V''(\tilde r) d\lambda_{\bar p,\tau},
\end{gathered}
\end{equation}
where $\tilde r$ is between $r$ and $\rho(\tau)$. Recalling that $c_1 \le V''(r) \le c_2$, this implies
\begin{equation}
c_2^{-1} \le \int (r-\rho(\tau))^2d\lambda_{\bar p,\tau} \le c_1^{-1},
\end{equation}
that is
\begin{equation} \label{eq:F''bound}
c_2^{-1} \le G''(\tau) \le c_1^{-1}.
\end{equation}
Finally, since $G$ is smooth, the supremum in \eqref{eq:apptau} is attained when $G'(\tau)=\rho$. But since we have
proven that $G'$ is invertible ($G''$ is strictly positive), the
equation $G'(\tau) = \rho$ has exactly one solution for any $\rho
\in \mathbb R$. We claim that this solution is precisely
$\tau(\rho)$, as defined in \eqref{eq:apptau}. In fact, let $\rho \in \mathbb R$ and let $\tau = \tau(\rho)$ solve
$G'(\tau) = \rho$. We have
\begin{equation}
F(\rho) = \rho  \tau- G(\tau).
\end{equation}
This implies
\begin{equation}
\tau(\rho)=F'(\rho) =  \tau +\rho  \tau'(\rho)-G'( \tau) \tau'(\rho) = \tau
\end{equation}
and, in turn,
\begin{equation}
\tau'(\rho)=  \tau'(\rho) = \frac{1}{G''( \tau)}.
\end{equation}
Therefore we have the desired bound on $\tau'$ and the proof is complete.
\end{proof}
\begin{oss}
Since there is a 1:1 correspondence between $\tau$ and $\rho$ via
the equation $\rho = G'(\tau)$, we can always express $\tau$ as a
function of $\rho$ and viceversa. For this reason, we shall adopt the
following notation. When writing a chain of equalities or
inequalities, the object at the far left tells us which, between
$\tau$ and $\rho$ is the independent variable. To be precise, the writing
\[
f(\tau) = g (\tau, \rho)
\]
stands for
\[
f(\tau) = g (\tau, G'(\tau))
\]
while
\[
f(\rho) = g(\tau, \rho)
\]
stands for
\[
f(\rho) = g(\tau(\rho), \rho).
\]

\end{oss}
\begin{cor} \label{cor:apptau}
Let $\tau$ and $F$ be defined by \eqref{eq:apptau}. Moreover, let $c_2$ be as in \eqref{eq:appV''}. Then $\tau'(\rho) \le c_2$ for all $\rho \in \mathbb R$ and $\tau(\rho)/F(\rho) \to 0$ as $|\rho| \to \infty$.
\end{cor}
\begin{proof}
$\tau'$ is bounded from above thanks to \eqref{lem:tau'bound}. Since $\tau = F'$, it is enough to show that $F(\rho)$ grows at least quadratically and $F'(\rho)$ grows at most linearly in $\rho$. We consider $\rho \to \infty$, as $\rho \to -\infty$ will be analogous. Since $\tau = F'$, we have  $c_1\le F''(\rho)\le c_2$. Integrating this twice we obtain
\begin{equation}
F'(0)+ c_1 \rho\le F'(\rho) \le  F'(0) + c_2 \rho.
\end{equation}
and
\begin{equation}
F(0)+ F'(0) \rho+ \frac{c_1}{2} \rho^2\le F(\rho) \le F(0)+ F'(0) \rho+ \frac{c_2}{2} \rho^2.
\end{equation}
Therefore, since $c_1, c_2 >0$, $F$ grows at least quadratically and $F'$ at most linearly, and the conclusion follows.
\end{proof}
Since we have shown that $\tau'$ is bounded from below, the $L^\infty$ bounds in part $iii)$ of Proposition \ref{prop:appmain} follow from the following
\begin{lem} \label{lem:tauLinf}
Let $\tau$ be as in \ref{eq:apptau}. Then both $\tau''$ and $\tau'''$ are bounded.
\end{lem}
\begin{proof}
First of all let us note that
\begin{equation}
\tau''(\rho) = - \frac{G'''( \tau) \tau'(\rho)}{G''(\tau)^2} = - G'''( \tau) \tau'(\rho)^3
\end{equation}
and
\begin{equation}
\begin{gathered}
\tau'''(\rho)=-G^{(iv)}( \tau)  \tau'(\rho)^4-3 G'''( \tau) \tau'(\rho)^2 \tau''(\rho)
\\
=3\frac{\tau''(\rho)^2}{\tau'(\rho)} -G^{(iv)}( \tau) \tau'(\rho)^4.
\end{gathered}
\end{equation}
Therefore it is enough to prove that $G'''$ and $G^{(iv)}$ are bounded. We have
\begin{equation}\begin{gathered}
G'''(\tau)= \frac{d}{d\tau} \int (r-\rho)^2d\lambda_{\bar p,\tau}
\\
= \int (r-\rho)^2(r-G'(\tau))d\lambda_{\bar p,\tau} -2\rho'(\tau) \int (r-\rho)d\lambda_{\bar p,\tau}
\\
= \int (r-\rho)^3d\lambda_{\bar p,\tau}
\end{gathered}\end{equation}
and
\begin{equation} \begin{gathered}
G^{(iv)}(\tau)= \int (r-\rho)^4d\lambda_{\bar p,\tau}-3 \rho'(\tau) \int (r-\rho)^2 d\lambda_{\bar p,\tau}
\\
=\int (r-\rho)^4d\lambda_{\bar p,\tau}-3 \left( \int (r-\rho)^2 d\lambda_{\bar p,\tau}\right)^2
\\
=\int (r-\rho)^4d\lambda_{\bar p,\tau}-3  G''(\tau)^2
\end{gathered}\end{equation}
Moreover, since
\begin{equation}
\int|r-\rho|^3 d\lambda_{\bar p,\tau} \le \frac{1}{2}\int(r-\rho)^2d\lambda_{\bar p,\tau}+\frac{1}{2} \int (r-\rho)^4 d\lambda_{\bar p,\tau},
\end{equation}
it is sufficient to show that $\int (r-\rho)^4 d\lambda_{\bar p,\tau}$
is a bounded function of $\tau$. In order to do so, let $\delta = \delta(\tau)$ be the point at which the function $ r \mapsto \tau r -
V(r)$ attains its maximum. Since $V$ is strictly convex,  $\delta$ is the unique root of
the equation $V'(\delta)=\tau$. We claim that
$\int(r-\rho)^4d\lambda_{\bar p,\tau}$ is bounded provided $\int
(r-\delta)^2 d\lambda_{\bar p,\tau}$ and $|\rho-\delta|$ are. In fact,
from
\begin{equation}
\int (r-\rho)^4d\lambda_{\bar p,\tau} \le 8 \int(r-\delta)^4d\lambda_{\bar p,\tau}+8\int(\rho-\delta)^4d\lambda_{\bar p,\tau}
\end{equation}
and
\begin{equation}
\begin{gathered}
3 \int(r-\delta)^2 d\lambda_{\bar p,\tau} = \int (r-\delta)^3 \left(V'(r)-\tau\right)d\lambda_{\bar p,\tau}
\\
= \int (r-\delta)^3 (V'(r)-V'(\delta))d\lambda_{\bar p,\tau} = \int
V''(\tilde r)(r-\delta)^4 d\lambda_{\bar p,\tau} \ge c_1
\int(r-\delta)^4 d\lambda_{\bar p,\tau},
\end{gathered}
\end{equation}
for some $\tilde r$ between $r$ and $\delta$, we obtain
\begin{equation}
\int(r-\rho)^4 d\lambda_{\bar p,\tau} \le \frac{24}{c_1}
\int(r-\delta)^2 d\lambda_{\bar p,\tau} + 8 (\rho-\delta)^4.
\end{equation}
 The boundedness of $\int (r-\delta)^2 d\lambda_{\bar p,\tau}$ is in turn given by
\begin{equation}
1 = \int(r-\delta)(V'(r)-\tau)d\lambda_{\bar p,\tau} = \int V''(\tilde r)(r-\delta)^2d\lambda_{\bar p,\tau}.
\end{equation}
Finally, a bound for $|\rho-\delta|$ follows from
\begin{equation}
\begin{gathered}
\int(r-\delta)^2d\lambda_{\bar p,\tau} = \int(r-\rho)^2d\lambda_{\bar p,\tau} + (\rho-\delta)^2+2 (\rho-\delta) \int (r-\rho)d\lambda_{\bar p,\tau}
\\
=G''(\tau)+(\rho-\delta)^2.
\end{gathered}
\end{equation}
\end{proof}
Let us now prove the $L^2$ bounds in part $iii)$ of Proposition \ref{prop:appmain}. By
\begin{equation}
\tau''(\rho)= - \tau'(\rho)^3 G'''( \tau),
\end{equation}
\begin{equation}
\tau'''(\rho)= \frac{3}{\tau'(\rho)}\tau''(\rho)^2-\tau'(\rho)^4 G^{(iv)}( \tau),
\end{equation}
and the fact that $\tau'$ is bounded away from zero, $\tau''$ is in $L^2$ if and only if $G'''( \tau)$ is. Moreover $\tau'''$ is in $L^2$ provided both $(\tau'')^2$ and $G^{(iv)}( \tau)$ are. But $\tau'' \in L^\infty \cap L^2$ implies  $\tau'' \in L^p$ for all $p\ge 2$: in particular $\tau'' \in L^4$, and so $(\tau'')^2 \in L^2$  Finally, via the substitution $\tau= \tau(\rho)$ (or, equivalently, $\rho=G'(\tau)$), for $f \in \{G''',G^{(iv)}\}$ we have
\begin{equation}
\int_{-\infty}^{\infty}f( \tau(\rho))^2 d \rho = \int_{-\infty}^{\infty}\frac{f(\tau)^2}{\tau'(G'(\tau))} d \tau,
\end{equation}
Therefore using once more the boundedness from below of $\tau'$, the $L^2$ bounds in part $iii)$ of Proposition \ref{prop:appmain} follow from the next lemma.
\begin{lem} \label{lem:appFL2}
Let $G$ be defined by \ref{eq:appF}. Then both $G'''$ and $G^{(iv)}$ are in $L^2(\mathbb R)$.
\end{lem}
\begin{proof}
We observe that, since $G'''$ and $G^{(iv)}$ are bounded, it is enough to prove that they vanish quickly enough at infinity. We shall prove that they indeed vanish and the decay rate is exponential. Since
\begin{equation} \label{eq:appG'''}
G'''(\tau) = \int(r-\rho)^3 d\lambda_{\bar p,\tau}
\end{equation}
and
\begin{equation}\label{eq:appGiv}
G^{(iv)}(\tau)= \int (r-\rho)^4d\lambda_{\bar p,\tau} -3 \left( \int (r-\rho)^2 d\lambda_{\bar p,\tau}\right)^2,
\end{equation}
we need to estimate the quantities
\begin{equation}
\int (r -\rho)^m  d\lambda_{\bar p,\tau} = \int_{-\infty}^\infty(r-\rho)^m \exp(\tau r-V(r)-G(\tau))dr
\end{equation}
for integers $2\le m \le 4$, as well as
\begin{equation}
\rho = \int r d\lambda_{\bar p,\tau} = \int_{-\infty}^\infty r \exp(\tau r-V(r)-G(\tau))dr,
\end{equation}
as $|\tau| \to \infty$. We will deal with $\tau\to\infty$, as the case $\tau\to-\infty$ is analogous. Recalling that $\delta$ is such that $\tau=V'(\delta)$, we can write
\begin{equation}
\exp(\tau r - V(r) - G(\tau)) = \exp\left(\tau \delta - V(\delta) - G(\tau) - \frac{V''(\tilde r)}{2}(r-\delta)^2\right),
\end{equation}
for some $\tilde r$ between $r$ and $\delta$. Also, integrating $c_1 \le V''(r) \le c_2$ between $0$ and $\delta$ gives
\begin{equation}
c_2^{-1}\tau - c_2^{-1}V'(0) \le \delta \le c_1^{-1}\tau -c_1^{-1}V'(0),
\end{equation}
and so $\delta\to\infty$ as $\tau \to\infty$. 

Next, we show the following:
\begin{equation} \label{eq:asintotico}
\left| \int_{-\infty}^\infty (r-\delta)^m \exp \left(-\frac{V''(\tilde r)}{2}(r-\delta)^2 \right)d r- \int_{-\infty}^\infty r^m \exp \left(-\frac{V''_+}{2}r^2 \right)d r\right| \le  e^{-\tilde \alpha \tau}
\end{equation}
for integers $0 \le m \le 4$, some $0<\tilde \alpha \le \alpha$ and
$\tau$ large enough. The constant $V''_+$ is defined in \eqref{eq:appAN}. Let $a \in (0,\tau)$ be a multiple of $\tau$. We divide the domain of integration as follows: $(-\infty,\infty)= (-\infty, \delta-a) \cup (\delta-a,\delta+a) \cup (\delta+a,\infty)$. The integrals over the unbounded domains vanish exponentially fast as $\tau \to \infty$. In fact, since there exists $0<\gamma <c_1$ such that, for $\tau$ large enough,
\begin{equation}
|r-\delta|^m \exp\left(-\frac{V''(\tilde r)}{2}(r-\delta)^2\right) \le \exp \left( -\frac{\gamma}{2}(r-\delta)^2\right),
\end{equation}
we have
\begin{equation}
\begin{gathered}
\int_{\delta+a}^\infty |r-\delta|^m \exp\left( - \frac{V''(\tilde r)}{2}(r-\delta)^2\right)dr \le \int_{\delta+a}^\infty \exp\left( - \frac{\gamma}{2}(r-\delta)^2\right)dr
\\
= \int_0^\infty \exp \left(- \frac{\gamma}{2}(r+a)^2\right)dr \le \exp\left(-\frac{\gamma}{2} a^2\right) \int_0^\infty \exp\left(-\frac{\gamma}{2} r^2\right)dr,
\end{gathered}
\end{equation}
which vanishes exponentially fast since $a$ is a multiple of $\tau$. The case of $(-\infty,\delta-a)$ is analogous. From similar calculations we obtain also that
\begin{equation}
\begin{gathered}
\int_{-\infty}^\infty r^m \exp \left(-\frac{V''_+}{2}r^2 \right)dr =\int_{-\infty}^\infty (r-\delta)^m \exp\left( - \frac{V''_+}{2}(r-\delta)^2\right)dr 
\\
=  \int_{\delta-a}^{\delta+a}(r-\delta)^m \exp\left(-\frac{V''_+}{2}(r-\delta)^2\right)dr + R(\tau),
\end{gathered}
\end{equation}
with $R(\tau)$ vanishing exponentially fast in $\tau$. Therefore
\eqref{eq:asintotico} follows provided
\begin{equation}
\left| \int_{\delta-a}^{\delta+a}(r-\delta)^m \left( \exp\left( - \frac{V''(\tilde r)}{2}(r-\delta)^2\right) - \exp\left( - \frac{V''_+}{2}(r-\delta)^2\right)\right)dr\right| \le  e^{-\tilde \alpha \tau}.
\end{equation}
Recall that $\tilde r$ is between $r$ and $\delta$, and therefore is in $(\delta-a,\delta+a)$. Recall also that $\delta \to\infty$ as $\tau\to \infty$. Moreover, $\delta-a$ goes to $\infty$ as well, provided $a < c_2^{-1}\tau$. For such a choice of $a$, $\tilde r$ goes to $\infty$ as $\tau \to \infty$ and so, thanks to \eqref{eq:appAN}, for $\tau$ large enough we have
\begin{equation}
\frac{1}{2}|V''(\tilde r) - V''_+| (r-\delta)^2 \le e^{-\tilde \alpha \tau},
\end{equation}
for some positive $\tilde \alpha$. This implies
\begin{equation}
\begin{gathered}
\left| \int_{\delta-a}^{\delta+a} (r-\delta)^m \exp\left( - \frac{V''(\tilde r)}{2}(r-\delta)^2\right)dr-\int_{\delta-a}^{\delta+a}(r-\delta)^m  \exp\left( - \frac{V''_+}{2}(r-\delta)^2\right)dr\right| 
\\
\le  \int_{\delta-a}^{\delta+a} |r-\delta|^m \exp\left( - \frac{V''_+}{2}(r-\delta)^2\right)\left| \exp\left( -\frac{1}{2}(V''(\tilde r)-V''_+)(r-\delta)^2\right)-1\right| dr 
\\
\le 2 a^m e^{-\tilde \alpha \tau} \int_{\delta-a}^{\delta+a}  \exp\left( - \frac{V''_+}{2}(r-\delta)^2\right)dr
\\
\le 2 \sqrt\frac{2\pi}{V''_+} a^m e^{-\tilde \alpha \tau} \le e^{-\tilde \alpha \tau}
\end{gathered}
\end{equation}
for $\tau$ large enough and a possibly different choice of $\tilde
\alpha$. This proves \eqref{eq:asintotico}.

We have now all we need to prove the actual lemma.
\begin{equation}\begin{gathered}
1 = \int_{-\infty}^\infty \exp(\tau r - V(r)-G(\tau)) dr 
\\
= \int_{-\infty}^\infty \exp\left(\tau \delta - V(\delta) - G(\tau) - \frac{V''(\tilde r)}{2}(r-\delta)^2\right)dr,
\end{gathered}\end{equation}
implies
\begin{equation}
\exp(G(\tau)-\tau \delta + V(\delta)) =  \int_{-\infty}^\infty \exp\left( - \frac{V''(\tilde r)}{2}(r-\delta)^2\right)dr,
\end{equation}
and the left hand side is bounded away from zero, as the right hand side
is. In particular, \\ $\exp(\tau
\delta - V(\delta)-G(\tau))$ is bounded.

Next, we show that $\rho-\delta\to 0$ exponentially fast. We write
\begin{equation} \begin{gathered}
\rho-\delta = \int_{-\infty}^{\infty}(r-\delta) \exp \left( \tau\delta-V(\delta)-G(\tau)-\frac{V''(\tilde r)}{2}(r-\delta)^2 \right)dr
\\
=  \exp \left( \tau\delta-V(\delta)-G(\tau)\right) \int_{-\infty}^\infty  (r-\delta) \exp\left( - \frac{V''(\tilde r)}{2}(r-\delta)^2\right)dr,
\end{gathered}\end{equation}
which converges to zero exponentially fast, as $ \exp \left( \tau\delta-V(\delta)-G(\tau)\right)$ is bounded and the integral converges to
\begin{equation}
 \int_{-\infty}^\infty  r \exp\left( - \frac{V''_+}{2}r^2\right)dr =0
\end{equation}
exponentially fast. Next, we write
\begin{equation} \begin{gathered}
G''(\tau) = \int (r-\rho)^2 d \lambda_{\bar p,\tau}  =\int (r^2-\rho^2) d \lambda_{\bar p,\tau} 
\\
= \int_{-\infty}^{\infty}(r-\delta)^2 \exp \left( \tau\delta-V(\delta)-G(\tau)-\frac{V''(\tilde r)}{2}(r-\delta)^2 \right)dr - (\rho-\delta)^2.
\end{gathered}\end{equation}
The term $(\rho-\delta)^2$ goes to zero, while the integral, and so $G''(\tau)$, converges to
\begin{equation}
\sqrt\frac{V''_+}{2\pi}\int_{-\infty}^\infty r^2 \exp \left(-\frac{V''_+}{2}r^2\right)dr = \frac{1}{V''_+}
\end{equation}
exponentially fast. $G'''(\tau)$ goes to zero exponentially fast. In fact from
\begin{equation}
(r-\rho)^3 = (r-\delta)^3+(\delta-\rho) \left( (r-\rho)^2 +(r-\rho)(r-\delta)+(r-\delta)^2 \right)
\end{equation}
and after integration, the first term vanishes, in the limit, by symmetry. Moreover, all the other terms vanish, after integration, as they are bounded terms multiplied by $\delta-\rho$.

Finally $G^{(iv)}$ vanishes exponentially fast as well. This time, though, we have the difference of two non-vanishing terms, so we do need to pay some extra attention. The quadratic term $-3G''(\tau)^2$ converges to $-3/(V''_+)^2$. On the other hand, the quartic term decomposes as
\begin{equation} \begin{gathered}
(r-\rho)^4 =(r-\delta)^4 + (\delta-\rho)(2r-\rho-\delta)\left((r-\rho)^2+(r-\delta)^2\right)
\\
= (r-\delta)^4 + (\delta-\rho) \left(2(r-\rho)^3 +2(r-\delta)^3+(\rho-\delta)(r-\rho)^2+(\delta-\rho)(r-\delta)^2\right).
\end{gathered}\end{equation}
Again, all the terms that multiply $\delta-\rho$ are, after integration, bounded, and therefore the only term which survives is $(r-\delta)^4$, whose integral converges to
\begin{equation}
\sqrt\frac{V''_+}{2\pi} \int_{-\infty}^\infty r^4  \exp \left(-\frac{V''_+}{2}r^2\right)dr = \frac{3}{(V''_+)^2}.
\end{equation}
Putting everything together we obtain that $G^{(iv)}(\tau)$ converges exponentially fast to zero, and the proof is complete.
\end{proof}
 {We now prove that $\tau$ is strictly convex}. First of all we make the following remark.
\begin{oss}
Suppose $ {V(r) = r^2/2+  U(r)}$. Then, if $U$ is an even function, $\tau''$ vanishes at the origin. In particular we can never have genuine nonlinearity.

In order to see this, it is enough to show that $F$ is even: in fact in this case its third derivative, $\tau''$, is odd and so $\tau''(0) = 0$. $F$ is indeed even:
\begin{equation}
F(-\rho) = \sup_{\tau \in \mathbb R} \{-\tau \rho- G(\tau) \} = \sup_{\tau \in \mathbb R}\{ \tau \rho - G(-\tau)\} = F(\rho),
\end{equation}
since
\begin{equation} \begin{gathered}
G(-\tau) = \log \int_{-\infty}^\infty \exp \left(-\tau r - V(r)\right)dr
\\
 = \log \int_{-\infty}^\infty \exp \left(\tau r - V(-r)\right)dr 
= G(\tau).
\end{gathered}\end{equation}
\end{oss}
\begin{oss}
 {In order to get the LSI} \eqref{eq:logSob}  {we may relax the assumption of uniform convexity of the potential. In this case, $V$ is a compactly supported perturbation of the harmonic interaction. With such a potential, however, the tension fails to be strictly convex. In fact, setting $V(r) = r^2/2+U(r)$,}
\begin{equation}
\tau(\rho) = \rho + \zeta(\rho),
\end{equation}
 {where}
\begin{equation}
\zeta(\rho) =\int U'(r) d \lambda_{\bar p, \tau}
\end{equation}
 {is bounded ($U$ is smooth and bounded). But $\tau$ is strictly convex if and only if $\zeta$ is, and this is impossible, as the latter is bounded.}
\end{oss}
 {Thanks to the remarks, in order to have genuine nonlinearity we must look among non-symmetric interactions. Even if we only consider unbounded perturbations of the harmonic potential, it is not known which features a potential must possess in order to ensure $\tau'' > 0$. Therefore, we shall only give one family of functions which work.}
 \begin{prop}
 {Let $V$ be a mollification of the function}
\begin{equation} \label{eq:Vvera}
 {r \longmapsto \frac{1}{2}(1-\kappa)r^2+\frac{1}{2}\kappa r|r|_+,}
\end{equation}
 {where $|r|_+=\max\{r,0\}$ and $0<\kappa <1/3$. Then, $\tau''(\rho)>0$ for all $\rho \in \mathbb R$.}

\end{prop}
\begin{proof}
Since
\begin{equation}
\tau''(\rho) = -\tau'(\rho) G'''(\tau),
\end{equation}
with $\tau'>0$, the sign of $\tau''$ is the same as the sign of $-G'''$. Therefore we need to study 
\begin{equation}
-G'''(\tau) = -\int (r-\rho)^3 d \lambda_{\bar p,\tau},
\end{equation}
with $\tau \in \mathbb R$. Let $|r|_+ = \max\{r,0\}$.  {In order to make things slightly less technical, we take directly $V(r) =1/2(1-\kappa)r^2+1/2\kappa r|r|_+$ instead of its mollification (note that $V$ is already twice differentiable except at the origin).} Write
\begin{equation}
 {V(r) = \frac{a}{2}r^2+ W(r)},
\end{equation}
where $a = 1- \kappa$ and $W(r) = \kappa r |r|_+/2$. Then we notice that, by the usual integration by parts trick, we have
\begin{equation}
\int (r-\rho)^2 (V'(r)-\tau) d \lambda_{\bar p,\tau}= 2 \int (r-\rho) d \lambda_{\bar p,\tau}=0.
\end{equation}
Therefore we write
\begin{equation} \begin{gathered}
a\int (r-\rho)^3 d \lambda_{\bar p,\tau} = \int (r-\rho)^2 (ar-a\rho -V'(r)+\tau)d \lambda_{\bar p,\tau}
\\
= \int (r-\rho)^2 (\tau - a \rho- W'(r))d \lambda_{\bar p,\tau},
\end{gathered}\end{equation}
from which, together with $W'(r) =\kappa |r|_+$, it follows
\begin{equation} \label{eq:last0} \begin{gathered}
-a G'''(\tau) =  \int (r-\rho)^2 W'(r) d \lambda_{\bar p,\tau}-(\tau-a\rho) G''(\tau)
\\
= \kappa \int |r|_+ (r-\rho)^2  d \lambda_{\bar p,\tau}+ (a \rho-\tau)G''(\tau).
\end{gathered}\end{equation}
We evaluate 
\begin{equation}
\int |r|_+ (r-\rho)^2 d \lambda_{\bar p,\tau} = e^{-G(\tau)}\int_0^\infty r (r-\rho)^2 e^{ \tau r - \frac{r^2}{2}}dr =
\end{equation}
\begin{equation}
= e^{-G(\tau)}\int_0^\infty r^3 e^{ \tau r - \frac{r^2}{2}}dr -2\rho e^{-G(\tau)}\int_0^\infty r^2 e^{ \tau r - \frac{r^2}{2}}dr +\rho^2e^{-G(\tau)} \int_0^\infty r e^{ \tau r - \frac{r^2}{2}}dr.
\end{equation}
Now, for $m \in \mathbb N$, $m \ge 1$,
\begin{equation}
\int_0^\infty r^m e^{ \tau r - \frac{r^2}{2}}dr =\frac{d}{d \tau}  \int_0^\infty r^{m-1} e^{ \tau r - \frac{r^2}{2}}dr,
\end{equation}
with
\begin{equation}
 \int_0^\infty e^{ \tau r - \frac{r^2}{2}}dr =  e^\frac{\tau^2}{2} \int_{-\tau}^\infty e^{-\frac{r^2}{2}}dr.
\end{equation}
Therefore, setting $\Phi(\tau):=\int_{-\tau}^\infty e^{-\frac{r^2}{2}}dr$ and noting that $\Phi'(\tau) = e^{-\frac{\tau^2}{2}}$ yields
\begin{equation} \begin{gathered}
\int_0^\infty r e^{ \tau r - \frac{r^2}{2}}dr =\frac{d}{d \tau} \left( e^\frac{\tau^2}{2}\Phi(\tau) \right) = \tau e^\frac{\tau^2}{2}\Phi(\tau) + 1,
\\
\int_0^\infty r^2 e^{ \tau r - \frac{r^2}{2}}dr = \frac{d}{d\tau}\int_0^\infty r e^{ \tau r - \frac{r^2}{2}}dr =  \tau^2 e^\frac{\tau^2}{2}\Phi(\tau) +e^\frac{\tau^2}{2}\Phi(\tau) + \tau
\\
\int_0^\infty r^3 e^{ \tau r - \frac{r^2}{2}}dr = \frac{d}{d\tau} \int_0^\infty r^2 e^{ \tau r - \frac{r^2}{2}}dr 
\\
= \tau e^\frac{\tau^2}{2}\Phi(\tau)+ 1 + 2\tau e^\frac{\tau^2}{2}\Phi(\tau) + \tau^3 e^\frac{\tau^2}{2}\Phi(\tau) + \tau^2+1
\\
= \tau^3 e^\frac{\tau^2}{2}\Phi(\tau)+ 3\tau e^\frac{\tau^2}{2}\Phi(\tau) +\tau^2+ 2.
\end{gathered}\end{equation}
Putting everything together we obtain
\begin{equation}\begin{gathered}
e^{G(\tau)}\int |r|_+ (r-\rho)^2 d\lambda_{\bar p,\tau} =  \tau^3 e^\frac{\tau^2}{2}\Phi(\tau)+ 3\tau e^\frac{\tau^2}{2}\Phi(\tau) + 2+
\\
-2\rho (\tau^2 e^\frac{\tau^2}{2}\Phi(\tau) +e^\frac{\tau^2}{2}\Phi(\tau) + \tau)+\rho^2 ( \tau e^\frac{\tau^2}{2}\Phi(\tau) + 1)
\\
=(2+\tau^2-2\rho\tau+\rho^2) + (\tau^3-2\rho\tau^2+(3+\rho^2)\tau-2\rho)e^\frac{\tau^2}{2}\Phi(\tau)
\\
= (2+(\tau-\rho)^2)+ \left( 3\tau -2 \rho + \tau(\tau-\rho)^2\right)e^\frac{\tau^2}{2}\Phi(\tau).
\end{gathered}\end{equation}
Next, we have to evaluate the term proportional to $G''(\tau)$ in \eqref{eq:last0}, which can be written as follows
\begin{equation}\begin{gathered}
a \rho- \tau = a \rho - \int V'(r) d \lambda_{\bar p,\tau} 
\\
= a \rho - \int (ar +  W'(r) ) d \lambda_{\bar p,\tau}
\\
= - \int W'(r) d \lambda_{\bar p,\tau}
\\
= -  \kappa e^{-G(\tau)} \int_0^\infty r e^{\tau r- \frac{r^2}{2}}dr
\\
= -\kappa e^{-G(\tau)} (1+ \tau e^\frac{\tau^2}{2}\Phi(\tau)).
\end{gathered}\end{equation}
Note that $1+\tau e^\frac{\tau^2}{2}\Phi(\tau) >0$ for $\tau \ge 0$. On the other hand, for $\tau <0$,
\begin{equation}
\int_{-\tau}^\infty e^{-\frac{r^2}{2}} dr < \frac{e^{-\frac{\tau^2}{2}}}{-\tau},
\end{equation}
which implies 
\begin{equation}
\tau\Phi(\tau) > -e^{-\frac{\tau^2}{2}}.
\end{equation}
Therefore we have $1+\tau e^\frac{\tau^2}{2}\Phi(\tau) >0$ also for $\tau <0$ and, as a consequence
\begin{equation}
\tau > a \rho = (1-\kappa)\rho, \quad \forall \tau \in
\mathbb R.
\end{equation}

Putting everything together we obtain
\begin{equation}
\begin{gathered}
-\frac{a e^{G(\tau)}}{ \kappa} G'''(\tau)= \left(2-G''(\tau)+(\tau-\rho)^2 \right)+
\\
 +\left( \tau (3- G''(\tau))-2\rho+ \tau(\tau-\rho)^2 \right) e^\frac{\tau^2}{2}\Phi(\tau).
\end{gathered}
\end{equation}
 {We show that $-G'''$ (and therefore $\tau''$) is positive fo $\tau \le 0$. Since $1-\kappa \le V''(r) \le 1$ and $\kappa < 1/2$,  Lemma }\ref{lem:tau'bound} { implies}
\begin{equation}
1 \le G''(\tau) \le \frac{1}{1- \kappa} < 2.
\end{equation}
{This gives}
\begin{equation}
-\frac{ae^{G(\tau)}}{ \kappa}G'''(\tau) >  (2\tau-2\rho + \tau(\tau-\rho)^2) e^\frac{\tau^2}{2}\Phi(\tau).
\end{equation}
Moreover, using $\tau > (1- \kappa)\rho$ and $\tau\Phi(\tau) > -e^{-\frac{\tau^2}{2}}$ yields
\begin{equation} \begin{gathered}
-\frac{ae^{G(\tau)}}{ \kappa}G'''(\tau) > - \frac{2\kappa}{1- \kappa} \tau e^\frac{\tau^2}{2}\Phi(\tau)+ (\tau-\rho)^2 \tau e^\frac{\tau^2}{2}\Phi(\tau)
\\
> - \frac{2 \kappa}{1- \kappa} \tau e^\frac{\tau^2}{2}\Phi(\tau) \ge 0.
\end{gathered}\end{equation}
Therefore $\tau''(\rho) >0$ if $\tau \le 0$.

For $\tau >0$ we have to be more careful. First of all we note that
\begin{equation}
a \rho-\tau = -   \kappa e^{-G(\tau)} \left(1+\tau e^\frac{\tau^2}{2} \Phi(\tau) \right)
\end{equation}
implies
\begin{equation}
\tau -\rho = \frac{  \kappa}{1-  \kappa} e^{-G(\tau)} \left( \tau \left( e^\frac{\tau^2}{2}\Phi(\tau)-e^{G(\tau)}\right)+1\right).
\end{equation}
This, together with
\begin{equation} \begin{gathered}
e^{G(\tau)} = \int_{-\infty}^\infty e^{\tau r -V(r)} dr
\\
= \int_0^\infty e^{\tau r-\frac{r^2}{2}}dr + \int_{-\infty}^0 e^{\tau r - \frac{1-  \kappa}{2}r^2}dr
\\
= e^\frac{\tau^2}{2}\Phi(\tau) + \frac{1}{\sqrt{1- \kappa}} e^\frac{\tau^2}{2(1-  \kappa)} \Phi \left(- \frac{\tau}{\sqrt{1-  \kappa}}\right)
\end{gathered}\end{equation}
and $\tau >0$ gives
\begin{equation}\begin{gathered}
\tau - \rho = \frac{  \kappa}{1-  \kappa} e^{-G(\tau)}\left(1+ \dfrac{-\tau}{\sqrt{1- \kappa}} e^\frac{\tau^2}{2(1-  \kappa)} \Phi \left(- \dfrac{\tau}{\sqrt{1-  \kappa}}\right) \right)
\\
> e^{-G(\tau)} \frac{  \kappa}{1-  \kappa} \left( 1 - e^\frac{\tau^2}{2(1- \kappa)} e^{-\frac{\tau^2}{2(1- \kappa)}} \right) = 0.
\end{gathered}\end{equation}
Therefore $\tau- \rho>0$ if $\tau >0$ (note here that $\tau-\rho$ is trivially positive for $\tau \le 0$, too). From this we get
\begin{equation}
\begin{gathered}
- \frac{a e^{G(\tau)}}{  \kappa} G'''(\tau) = (2- G''(\tau) + (\tau-\rho)^2)+ (\tau(1-G''(\tau))+2(\tau-\rho) +
\\
+\tau(\tau-\rho)^2) e^\frac{\tau^2}{2}\Phi(\tau)
\\
> (2-G''(\tau))+ (1-G''(\tau)) \tau e^\frac{\tau^2}{2}\Phi(\tau).
\end{gathered}\end{equation}
$2-G''(\tau)$ is positive, while $1-G''(\tau)$ is negative, so we need to perform a careful estimate. First of all, the estimate $G''(\tau)<2$ is too blunt, and will be replaced by $G''(\tau) \le 1/(1-  \kappa)$, so that
\begin{equation}
2-G''(\tau) \ge 2- \frac{1}{1-  \kappa} = 1-\frac{  \kappa}{1-  \kappa},
\end{equation}
which is positive, since $  \kappa < 1/2$. In order to estimate $1-G''(\tau)$ we calculate
\begin{equation} \begin{gathered}
a G''(\tau) = 1+   \kappa \rho e^{-G(\tau)} (1-\tau e^\frac{\tau^2}{2}\Phi(\tau))-   \kappa e^{-G(\tau)}(e^\frac{\tau^2}{2}\Phi(\tau)+\tau^2 e^\frac{\tau^2}{2}\Phi(\tau) + \tau )
\\
=1+ \rho (\tau- a\rho) -   \kappa e^{-G(\tau)}e^\frac{\tau^2}{2}\Phi(\tau)-\tau (\tau- a\rho)
\\
= 1 - (\tau-a\rho)(\tau-\rho)  -   \kappa e^{-G(\tau)}e^\frac{\tau^2}{2}\Phi(\tau).
\end{gathered}\end{equation}
Therefore
\begin{equation}
1-G''(\tau) =1-\frac{1}{a}+\frac{1}{a} (\tau-a\rho)(\tau-\rho)  + \frac{  \kappa}{a} e^{-G(\tau)}e^\frac{\tau^2}{2}\Phi(\tau)
\end{equation}
\begin{equation}
> -\frac{  \kappa}{1-  \kappa} + \frac{  \kappa}{1-  \kappa} e^{-G(\tau)}e^\frac{\tau^2}{2}\Phi(\tau),
\end{equation}
which implies
\begin{equation}\begin{gathered}
(1-G''(\tau)) \tau e^\frac{\tau^2}{2}\Phi(\tau) > \frac{  \kappa}{1-  \kappa} \frac{e^\frac{\tau^2}{2}\Phi(\tau)}{e^{G(\tau)}} \tau \left( e^\frac{\tau^2}{2}\Phi(\tau)-e^{G(\tau)}\right)
\\
= \frac{  \kappa}{1-  \kappa} \frac{e^\frac{\tau^2}{2}\Phi(\tau)}{e^{G(\tau)}}\dfrac{-\tau}{\sqrt{1- \kappa}} e^\frac{\tau^2}{2(1-  \kappa)} \Phi \left(- \dfrac{\tau}{\sqrt{1-  \kappa}}\right)
\\
> - \frac{  \kappa}{1-  \kappa} \frac{e^\frac{\tau^2}{2}\Phi(\tau)}{e^{G(\tau)}} >- \frac{  \kappa}{1-  \kappa},
\end{gathered}\end{equation}
since $e^\frac{\tau^2}{2}\Phi(\tau) < e^{G(\tau)}$. Putting everything together we obtain
\begin{equation}
 {-\frac{ae^{G(\tau)}}{  \kappa}G'''(\tau) > \frac{1-3  \kappa}{1-  \kappa}},
\end{equation}
 {and the right hand side is positive, since $ \kappa <1/3$}.
\end{proof}

\section{On the viscous approximation}
\label{sec:visc-appr}

If in the dynamics \eqref{eq:SDE} we choose $\sigma_N = N\delta$, for fixed $\bar\delta= (\delta_1, \delta_2)$, 
$\delta_j>0, j= 1,2$, the macroscopic equation will be given by the diffusive system:
\begin{equation}
  \label{eq:vpsystem}
\begin{cases}
 \partial_t  r (t,x) -  \partial_x  p(t,x)= \delta_1 \partial_{xx} \tau_\beta(r(t,x)) \qquad x\in (0,1)
 \\
 \partial_t  p(t,x)- \partial_x \tau_\beta(r(t,x)) = \delta_2 \partial_{xx} p(t,x),
\end{cases}
\end{equation}
with the boundary conditions:
\begin{equation*}
 p(t,0)=0, \quad \tau(r(t,1)) = \bar \tau(t),
\quad \partial_x  p(t,1)= 0, \quad  \partial_x  r(t,0)= 0,
\end{equation*}
 {Assume the existence of a strong solution of} \eqref{eq:vpsystem}.
 {For the infinite volume case, we refer to}  \cite{BianchiniB},  {but we could not find an explicit reference for these particular boundary conditions.}


The derivative of the total length is given by 
\begin{equation}
  \label{eq:v1}
  \frac d{dt} L(t) =  \frac d{dt} \int_0^1 r(t,x) \; dx = p(t,1) + \delta_1 \partial_{x} \tau_\beta(r(t,x))\Big|_{x=1}
\end{equation}
and the macroscopic work up to time $t$ is given by 
\begin{equation}
  \label{eq:v22}
  W(t) = \int_0^t \bar\tau(s) dL(s) = 
  \int_0^t \bar\tau(s) \left(p(s,1) + \delta_1 \partial_{x} \tau_\beta( r(t,x))\Big|_{x=1}\right)\; ds
\end{equation}
Then a direct calculation of the free energy time change gives:
\begin{equation}
  \label{eq:v2}
   \mathcal F(t) - \mathcal F(0) = W(t) - \int_0^t ds \int_0^1 \left[\delta_1 
\left(\partial_x \tau_\beta(r(s,x))\right)^2 + \delta_2 \left(\partial_x p(s,x)\right)^2\right]
\end{equation}
Letting $t\to\infty$, this gives the Clausius relation
\begin{equation}
  \label{eq:vclausius}
  F(\beta, \tau_1) - F(\beta,\tau_0) = W - \int_0^{\infty} dt \int_0^1 \left[\delta_1 
\left(\partial_x \tau_\beta( r(t,x))\right)^2 + \delta_2 \left(\partial_x p(t,x)\right)^2\right] \le W.
\end{equation}

Let $r^{\bar\delta}(t,x), p^{\bar\delta}(t,x)$ the solution of \eqref{eq:vpsystem}. 
We cannot prove the uniqueness of the limit an $\bar\delta \to 0$, but any limit point should satisfy the 
inequality of Clausius 
\begin{equation}
  \label{eq:v20}
   \mathcal F(t) - \mathcal F(0) \le  W(t), 
\end{equation}
where $W(t)$ is defined as the limit of \eqref{eq:v22}. 
Any such limit point $r(t,x), p(t,x)$ with the corresponding boundary 
layers are natural candidates for being the thermodynamic entropy solution of the equation 
\eqref{eq:vpsystem} and one can conjecture that such limit is unique.

\section*{Acknowledgments} 
This work has been partially supported by the grants ANR-15-CE40-0020-01 LSD 
of the French National Research Agency.

\addcontentsline{toc}{chapter}{References}

\renewcommand{\bibname}{References}
\nocite{*}
\bibliography{bibliografia.bib}
\bibliographystyle{plain}

\noindent
{Stefano Olla\\
CEREMADE, UMR-CNRS, Universit\'e de Paris Dauphine, PSL Research University}\\
{\footnotesize Place du Mar\'echal De Lattre De Tassigny, 75016 Paris, France}\\
{\footnotesize \tt olla@ceremade.dauphine.fr}\\
\\
{Stefano Marchesani\\
GSSI, \\
{\footnotesize Viale F. Crispi 7, 67100 L'Aquila, Italy}}

\end{document}